\documentclass[a4paper,twoside,10pt]{article}
\usepackage[a4paper,left=2cm,right=2cm, top=3cm, bottom=3cm]{geometry}
\usepackage{cuted}

\usepackage{enumerate}
\usepackage{orcidlink}
\usepackage{amsthm}
\usepackage{mathrsfs}
\usepackage{mathtools}
\usepackage{accents}
\usepackage{graphicx}
\usepackage{amscd,amsmath,amssymb,mathrsfs,bbm,listings}
\usepackage{cancel}
\usepackage{tikz}
\usetikzlibrary{decorations.pathreplacing}
\usepackage{epstopdf}
\definecolor{darkblue}{rgb}{0.0, 0.0, 0.8}
\allowdisplaybreaks
\graphicspath{{Figs/}}
\usepackage{amsmath}
\usepackage{footmisc}
\usepackage{amsthm}
\usepackage{amssymb}
\usepackage{stmaryrd}
\SetSymbolFont{stmry}{bold}{U}{stmry}{m}{n}
\usepackage{float}
\usepackage{bigints}
\usepackage{cite}
\usepackage{color}
\usepackage[abs]{overpic}
\usepackage[font=footnotesize,labelfont=bf]{caption}
\usepackage{cases}
\usepackage{tikz}
\usepackage{algorithmic}
\usepackage{rotating}
\usepackage{blkarray}
\usetikzlibrary{matrix,calc,arrows}
\usepackage{soul,xcolor}
\usepackage{enumitem}  
\usepackage{verbatim}
\usepackage{subfig}
\usepackage{graphicx}
\usepackage{mwe}
\usepackage{algorithm}

\newtheorem{theorem}{Theorem}[section]
\newtheorem{lemma}[theorem]{Lemma}
\newtheorem{proposition}[theorem]{Proposition}
\newtheorem{assumption}{Assumption}[section]

\newtheorem{remark}[theorem]{Remark}

\newcommand{\eremk}{\hbox{}\hfill\rule{0.8ex}{0.8ex}}

\numberwithin{equation}{section}
\hypersetup{
    colorlinks=true,
    citecolor=red,
    linkcolor=darkblue,
    pdfborder={0 0 0}
}

\def\Sp{\mathcal{S}_p}
\def\Mp{\boldsymbol{\mathcal{M}}_p}
\def\Th{\mathcal{T}_h}
\def\dt{\,\mathrm{d}t}
\def\dx{\,\mathrm{d}\bx}
\def\dS{\mathrm{d}S}

\def\Fh{\mathcal{F}_h}
\def\Fho{\mathcal{F}_h^\mathcal{I}}
\def\FhN{\mathcal{F}_h^{\mathcal{N}}}

\def\ba{\boldsymbol{a}}
\def\bat{\boldsymbol{\widetilde{a}}}
\def\bm{\boldsymbol{m}}
\def\bc{\boldsymbol{c}}

\def\calU{\boldsymbol{\mathcal{U}}}
\def\calA{\bunderline{\boldsymbol{\mathcal{A}}}}
\def\calN{\bunderline{\boldsymbol{\mathcal{N}}}}
\def\calE{\bunderline{\boldsymbol{\mathcal{E}}}}
\def\calAt{\bunderline{\boldsymbol{\widetilde{\mathcal{A}}}}}
\def\bF{\boldsymbol{\mathcal{F}}}
\def\bM{\boldsymbol{M}}
\def\Nh{\calN_h}
\def\tNh{\widetilde{\calN}_h}
\def\Eh{\calE_h}

\def\vQh{\mathbf{Q}_h}
\def\vSigmah{{\boldsymbol{\Sigma}}_h}
\def\vWh{\mathbf{W}_h}
\def\vVh{\mathbf{V}_h}
\def\vZh{\boldsymbol{Z}_h}

\def\brho{\boldsymbol{\rho}}
\def\bdsigma{\bunderline{\boldsymbol{\sigma}}}
\def\bdsigmah{\bdsigma_h} 
\def\bdq{\bunderline{\boldsymbol{q}}} 
\def\bdqh{\bdq_h}
\def\bdz{\bunderline{\boldsymbol{\zeta}}}
\def\bdvarphi{\bunderline{\boldsymbol{\varphi}}}
\def\bdzh{\bdz_h}
\def\bw{\boldsymbol{w}}

\def\bwh{\boldsymbol{w}_h}
\def\bvh{\boldsymbol{v}_h}
\def\bu{\boldsymbol{u}}
\def\bdvarphih{\bunderline{\boldsymbol{\varphi}}_h}
\def\blambdah{\boldsymbol{\lambda}_h}
\def\blambda{\boldsymbol{\lambda}}
\def\bdphih{\boldsymbol{\phi}_h} 

\def\bdthetah{\bunderline{\boldsymbol{\theta}}_h}
\def\bdpsih{\bunderline{\boldsymbol{\psi}}_h}

\def\lambdah{\lambda_h} 
\def\whi{w_{h,i}}
\def\lambdahi{\lambda_{h,i}}
\def\dpsih{\bunderline{\psi}_h}
\def\dzetahi{\bunderline{\zeta}_{h,i}}
\def\dpsihi{\bunderline{\psi}_{h,i}}
\def\dqhi{\bunderline{q}_{h,i}}
\def\dthetahi{\bunderline{\theta}_{h,i}}

\def\wFlux{\widehat{\bw}}
\newlength{\dhatheight}
\def\qFlux{
    \settoheight{\dhatheight}{\ensuremath{\widehat{q_h}}}
    \addtolength{\dhatheight}{-0.40ex}
    \widehat{\vphantom{\rule{1pt}{\dhatheight}}
    \smash{\widehat{\bdq}}}
}

\def\IR{\mathbb{R}}
\def\IN{\mathbb{N}}
\def\QT{Q_T}
\def\bx{\boldsymbol{x}}
\def\bn{\boldsymbol{n}}
\def\vn{\bunderline{n}}
\def\vnOmega{\vn_\Omega}
\def\vnK{\vn_K}
\def\dpt{\partial_t}
\def\bz{\mathbf{z}}
\def\by{\mathbf{y}}
\def\calD{\mathcal{D}}
\newcommand{\vf}{\boldsymbol{f}}
\def\Linfty{L^{\infty}}
\def\calH{\mathcal{H}}
\def\calB{\mathcal{B}}
\def\calR{\mathcal{R}}
\newcommand{\frob}{\,\mathbin{:}\,}

\newcommand{\Norm}[1]{\| #1 \|}
\newcommand{\bigNorm}[1]{\big\| #1 \big\|}
\newcommand{\mvl}[1]{\{ \!\!\{#1\}\!\!\}}  
\newcommand{\jump}[1]{\llbracket #1\rrbracket}
\newcommand{\jumpx}[1]{\llbracket #1\rrbracket_{\sf{n}_x}}

\newcommand{\bigjump}[1]{\big\llbracket #1\big\rrbracket}
\newcommand{\bunderline}[1]{\underaccent{\bar}{#1}}
\newcommand{\Pp}[2]{\mathbb{P}^{#1}(#2)}
\newcommand{\abs}[1]{\left|#1\right|}
\newcommand{\EFC}[3]{\mathcal{C}^{#1}\left(#2; #3\right)}

\newcommand{\DotProd}[2]{\left\langle #1, #2 \right\rangle}

\definecolor{DarkGreen}{rgb}{0.0, 0.5, 0.0}
\definecolor{DarkYellow}{rgb}{1.0, 0.75, 0.0}
\definecolor{DarkBlue}{rgb}{0.0, 0.5, 1.0}

\def\tGamma{\widetilde{\Gamma}}
\def\ox{\overline{x}}
\def\oy{\overline{y}}
\def\tOmega{\widetilde{\Omega}}
\def\tTh{\widetilde{\mathcal{T}}_h}
\def\tK{\widetilde{K}}
\def\pxy{\partial_{xy}}

\def\tFho{\widetilde{\mathcal{F}}_h^{\mathcal{I}}}
\def\tFhN{\widetilde{\mathcal{F}}_h^{\mathcal{N}}}

\newcommand{\sg}[1]{\textcolor{orange}{#1}}

\title{Structure-preserving Local Discontinuous Galerkin method for nonlinear cross-diffusion systems}

\author{Sergio G\'omez\thanks{Department of Mathematics and Applications, University of Milano-Bicocca, 20125 Milan, Italy \newline (\href{mailto:sergio.gomezmacias@unimib.it}{sergio.gomezmacias@unimib.it})} \thanks{IMATI-CNR ``E. Magenes", Via Ferrata 5, 27100 Pavia, Italy} \orcidlink{0000-0001-9156-5135}
\and 
Ansgar J\"ungel\thanks{Institute of Analysis and Scientific Computing, TU Wien, 1040 Wien, Austria (\href{mailto:juengel@tuwien.ac.at}{juengel@tuwien.ac.at})} \orcidlink{0000-0003-0633-8929}
\and 
Ilaria Perugia\thanks{Faculty of Mathematics, University of Vienna, 1090 Wien, Austria (\href{mailto:ilaria.perugia@univie.ac.at}{ilaria.perugia@univie.ac.at})} \orcidlink{0000-0003-1368-2883}
}

\date{}

\begin{document}
\maketitle

\begin{abstract}
\noindent We present and analyze a structure-preserving method for the approximation of solutions
to nonlinear cross-diffusion systems, which combines a Local Discontinuous Galerkin spatial discretization with the backward Euler time-stepping scheme.
The proposed method makes use of the underlying entropy structure of the system, expressing
the main unknown in terms of the entropy variable by means of a nonlinear transformation. Such a transformation allows for 
imposing the physical  positivity or boundedness constraints on the approximate solution in a strong sense. A key advantage of our scheme is that nonlinearities do not appear explicitly within differential operators or interface terms in the scheme,
which significantly improves its efficiency and eases its implementation.
We prove the existence of discrete solutions and their asymptotic convergence to a weak solution to the continuous problem.
Numerical results for some one- and two-dimensional problems illustrate the accuracy and entropy stability of the proposed method.
\end{abstract}



\section{Introduction}
We consider the following nonlinear reaction--diffusion system on a space--time cylinder~$\QT = \Omega \times (0, T]$, where~$\Omega \subset \IR^d$ ($d\in \{1,2,3\}$) is a bounded, polytopic domain with Lipschitz boundary~$\partial \Omega$, and~$T>0$:
\begin{subequations}
\label{EQN::MODEL-PROBLEM}
\begin{align} 
\label{EQN::MODEL-PROBLEM-1}
\dpt \brho - \nabla \circ \left(A(\brho) \nabla \brho \right)  & = \vf(\brho)  \quad \text{in } \QT,\\
\label{EQN::MODEL-PROBLEM-2}
(A(\brho) \nabla \brho) \,\vnOmega  
& = \mathbf{0}  \qquad\, \text{ on } \partial \Omega \times (0, T),\\
\label{EQN::MODEL-PROBLEM-3}
\brho & = \brho_0\qquad\, \text{on } \Omega \times \{0\}.
\end{align}
\end{subequations}
Here, the unknown is~$\brho := (\rho_1, \ldots, \rho_N)^{{\sf T}}: \QT \rightarrow \IR^{N}$ for some number of species~$N \in \IN$, $A : \IR^N \rightarrow \IR^{N \times N}$ is the diffusion matrix, 
$\vf: \IR^N \rightarrow \IR^N$ describes the nonlinear interaction between the~$N$ species, and~$\brho_0 \in \Linfty(\Omega)^N$ is a given initial datum. We denote by~$\nabla (\cdot)$ the~$\IR^{N \times d}$ matrix, whose rows contain the componentwise spatial gradients, by~$\nabla\circ(\cdot)$ the row-wise spatial divergence operator, and by~$\vnOmega$ the~$d$-dimensional vector of the spatial components of the unit normal vector at~$\partial\Omega\times (0, T)$ pointing outside~$\Omega\times (0, T)$. Some examples of problems of the form~\eqref{EQN::MODEL-PROBLEM} are given in Appendix \ref{app}; also see \cite[\S4.1]{Jungel_2016}.

The main challenges in the numerical approximation of the solution to nonlinear cross-diffusion systems are twofold: \emph{i)} the diffusion matrix~$A(\cdot)$ may not be symmetric nor positive definite, and \emph{ii)} a maximum principle may not be available. 
These issues prevent the use of standard techniques for the analysis of such systems, and make it difficult to guarantee that even weak solutions to the continuous problem respect the positivity or boundedness constraints of the physical unknowns. 
The boundedness-by-entropy framework in~\cite{Jungel_2015}, which we describe below, circumvents these issues by exploiting the underlying entropy structure of the system.

We focus on discontinuous Galerkin (DG) methods, which are characterized by the use of discrete broken spaces without any enforced conformity. Among many other advantages, DG methods offer great versatility for the treatment of nonlinearities. 
In particular, the Local Discontinuous Galerkin (LDG) method, originally introduced in~\cite{Cockburn_Shu_1998} for nonlinear convection\sg{--}diffusion systems, does not require nonlinearities to appear within differential operators or interface terms, leading to nonlinear operators that can be evaluated naturally in parallel.
Such a property is the result of appropriately rewriting the original problem in terms of auxiliary variables, and making use of~$L^2$-orthogonal projections in the discrete space of the nonlinear terms (see, e.g., \cite{Cockburn_Shu_1998,Dawson_Aizinger_Cockburn_2000}).

In order to obtain physically consistent discrete solutions, it is of utmost importance to design numerical methods that are not only accurate and efficient, but also reproduce, at the discrete level, the geometric and physical properties of the phenomenon being modeled.
Such numerical methods are called~\emph{structure preserving}.
One of the most difficult properties 
to reproduce at the discrete level is 
the physically expected positivity or boundedness of the continuous solution in finite element discretizations, especially for high-order approximations.
Although this is a well-known issue (see, e.g., the recent review in~\cite{Barrenechea_John_Petr_2024} on finite element methods (FEM) respecting the discrete maximum principle for convection\sg{--}diffusion equations),
only in last years has major progress been made in the literature. 
We briefly mention some recent works on this subject that do not rely on slope limiters or postprocessing techniques. 
In~\cite{Barrenechea_Georgoulis_Pryer_Vesser:2023}, the authors proposed a nodally bound-preseving FEM, whose discrete solution belongs to the convex set of piecewise polynomials satisfying the physical bound constraints on the mesh nodes. 
While this suffices to ensure strong (pointwise) positivity of the discrete solution for linear approximations, it does not provide any control on the values of the discrete solution away from the mesh nodes for higher-order approximations.
Motivated by the underlying entropy structure of the concerned PDEs, nonlinear transformations in terms of the entropy variable have been used to enforce positivity on the approximate solution of interior-penalty DG~\cite{Bonizzoni_Braukhoff_Jungel_Perugia:2020,Corti_Bonizzoni_Antoniett_2023}, conforming FEM~\cite{Braukhoff_Perugia_Stocker_2022},  hybrid high-order (HHO) \cite{Lemaire_Moatti_2024}, and Proximal Galerkin \cite{Keith_Surowiec:2024,Dokken_Farrel_Keith_Papadopoulos_Suroweic:2025} discretizations.

In this work, we propose an LDG method for the numerical approximation of the nonlinear cross-diffusion system~\eqref{EQN::MODEL-PROBLEM}, which is based on the framework of~\cite{Jungel_2015}, and possesses the following desirable properties:
\begin{itemize}[topsep = 0.3em, parsep = 0.2em, itemsep = 0.2em]
\item it allows for arbitrary degrees of approximation in space;
\item it preserves the boundedness of the physical unknowns without requiring any postprocessing or slope limiter; 
\item nonlinearities do not appear explicitly within differential operators or interface terms, which endows the method with a natural parallelizable structure and high efficiency;
\item it respects a discrete version of the entropy stability estimate of the continuous problem.
\end{itemize}
Although numerical methods for nonlinear cross-diffusion systems with some of these properties can be found in the literature, to the best of our knowledge, the proposed method is the first one satisfying all of them. 
For instance, finite volume methods for cross-diffusion systems have been proposed in~\cite{Andreianov_etal_2011,Cances_Gaudeul_2020,Cances_EM_2024,Jungel_Zurek_2021}, but at most second-order convergence rates in space are numerically obtained, whereas the entropy stable high-order DG method introduced in~\cite{Sun_Carrillo_Shu_2019} guarantees only weak positivity on Cartesian meshes by means of scaling limiters.

The proposed approach has been further developed for the Fisher-Kolmogorov equation in~\cite{Antonietti_Corti_Gomez_Perugia:2026} and for a class of conformational conversion systems in~\cite{Antonietti_Corti_Gomez_Perugia:2025}. In the latter, the ``degeneracy" arising in the entropy estimate for problems where the physical solutions are positive but not necessarily bounded is addressed.

\paragraph{The boundedness-by-entropy framework.} 
Henceforth, we make the following assumptions:

\begin{enumerate}[label=(H\arabic*), ref=(H\arabic*)]
    \item \label{H1} $A \in \EFC{0}{\overline\calD}{\IR^{N \times N}}$ and~$\vf \in \EFC{0}{\overline{\calD}}{\IR^N}$, for a bounded domain~$\calD \subset (0, \infty)^N$.
    \item \label{H2} There exists a convex function~$s \in \EFC{2}{\calD}{(0, \infty)} \cap \EFC{0}{\overline{\calD}}{(0, \infty)}$, with~$s' : \calD \rightarrow \IR^N$ invertible and inverse~$\bu := (s')^{-1} \in \EFC{1}{\IR^N}{\calD}$ such that the following three conditions are satisfied:
    \begin{enumerate}[label = (H2\alph{enumii}), ref = (H2\alph{enumii})]
    \item \label{H2a} There exists a constant~$\gamma > 0$ such that
    \begin{equation*}
        \label{EQN::COERCIVITY-A}
        \bz \cdot
        \left(s''(\brho) A(\brho) \bz\right) \geq \gamma \abs{\bz}^2 \qquad \forall \bz \in \IR^N, \ \brho \in \calD.
    \end{equation*}
    \item \label{H2b} There exists a constant~$C_f \geq 0$ such that
    \begin{equation*}
    \label{EQN::CONTINUITY-f}
    \vf(\brho) \cdot s'(\brho) \leq C_f \quad \forall \brho\in \calD.
    \end{equation*}
    \item \label{H2c} The initial datum $\brho_0$ is integrable and satisfies $\brho_0(\bx)\in \calD$ for a.e.~$\bx\in\Omega$. 
    \end{enumerate}
\end{enumerate}
Observe that Assumption~\ref{H2c} implies that 
\[
\int_\Omega s(\brho_0) \dx < \infty,
\]
since $\Omega$ is bounded, $s$ is bounded on $\overline{\calD}$, and $\brho_0$ takes values in~$\calD$. 

Let us comment on these assumptions. The boundedness of the domain $\calD$ guarantees the boundedness of the solutions; see below. Examples of models satisfying rigorously this property are given in Appendix \ref{app} and \cite[\S4.1]{Jungel_2016}.
The function~$s$ can be interpreted as the entropy or free energy of the underlying physical problem. Assumption~\ref{H2a} requires that the product~$s''(\brho) A(\brho)$ is positive definite. 
This condition can be generalized to include degenerate or singular continuous problems \cite[Ch.~4]{Jungel_2015}. We present some examples of cross-diffusion systems that satisfy Assumption~\ref{H2a} in Appendix~\ref{app}. Assumption~\ref{H2b} can be generalized to $\vf(\brho) \cdot s'(\brho) \leq C_f(1+s(\brho))$ for all $\brho\in\calD$, which yields an additional constant depending on time in inequality~\eqref{EQN::CONTINUOUS-ENTROPY-STABILITY} below. The condition that the initial datum lies pointwise in $\calD$ means that vacuum is initially excluded. This condition can be generalized to include vacuum, i.e.,\ $\brho_0\in\overline{\calD}$; see Remark \ref{rem:firststep}.

The main idea of the boundedness-by-entropy framework in~\cite{Jungel_2015} consists in  introducing the entropy variable~$\bw := s'(\brho)$ and then use the invertibility of~$s'(\cdot)$ in Assumption~\ref{H2} to write the original unknown as~$\brho = (s')^{-1}(\bw) = \bu(\bw)$. 
In this way, the boundedness of~$\calD$ in Assumption~\ref{H1} implies the pointwise boundedness of~$\bu(\bw)$, without requiring a maximum principle. 
Due to the regularity of the entropy density function~$s(\cdot)$ in Assumption~\ref{H2}, the following chain rule is valid:
\begin{equation}\label{eq:chainrule}
\nabla \bw = \nabla\left(s'(\brho)\right)
=s''(\brho) \nabla\brho.
\end{equation}
Taking~$\bw$ as the test function of the weak formulation of~\eqref{EQN::MODEL-PROBLEM} and using the chain rule \eqref{eq:chainrule}, we find that, for any~$0 < \tau \le T$,
\begin{align}
\label{EQN:FIRST-ENTROPY-FORM}
  \int_\Omega s(\brho(\bx, \tau)) \dx
  + \int_0^\tau\int_\Omega \nabla\brho \frob \big( s''(\brho)A(\brho)
  \nabla\brho  \big) \dx \dt = \int_{\Omega} s(\brho_0) \dx 
  + \int_0^\tau\int_\Omega \vf(\brho) \cdot s'(\brho)\dx \dt,
\end{align}
where~$B \frob C = \mathrm{tr}(C^T B)$ is the Frobenius inner product for matrices.

Then, Assumptions~\ref{H2a}--\ref{H2c} imply the entropy stability estimate
\begin{equation}
\label{EQN::CONTINUOUS-ENTROPY-STABILITY}
    \int_{\Omega} s(\brho(\bx, \tau)) \dx + \gamma \int_0^{\tau} \Norm{\nabla \brho}_{[L^2(\Omega)^d]^N}^2 \dt \leq \int_{\Omega} s(\brho_0) \dx + C_f \tau |\Omega| \qquad \text{ for all } 0 < \tau \leq T.
\end{equation}
The formal gradient-flow structure motivates another formulation of the entropy production integral:
\begin{align}
\label{EQN:SECOND-ENTROPY-FORM}
  \int_\Omega s(\brho(\bx, \tau)) \dx
  + \int_0^\tau\int_\Omega \nabla\bw \frob \big( A(\brho)s''(\brho)^{-1}
  \nabla\bw \big) \dx \dt = \int_{\Omega} s(\brho_0) \dx 
  + \int_0^\tau\int_\Omega \vf(\brho) \cdot s'(\brho)\dx \dt.
\end{align}
It turns out that in applications, the first formulation \eqref{EQN:FIRST-ENTROPY-FORM} is more convenient. In fact, we cannot formulate a discrete chain rule in the second formulation~\eqref{EQN:SECOND-ENTROPY-FORM} that avoids nonlinear terms under the differential operator like in our approach.

The change to entropy variables is well known in the theory of hyperbolic conservation laws~\cite{Mock_1980} and in the existence and numerical analysis of Nernst--Planck-type equations \cite{Gajewski_1985,Metti_Xu_Liu_2016}. In nonequilibrium thermodynamics, the entropy variables are known as \emph{chemical potentials}. The novelty of our approach is that Assumption~\ref{H2} automatically yields pointwise lower and upper bounds for the solution $\brho$, thus endowing the numerical scheme with inherent stability.
\paragraph{Outline of the paper.}
In Section~\ref{SECT::DG-METHOD}, we first rewrite the nonlinear cross-diffusion system in~\eqref{EQN::MODEL-PROBLEM} in terms of some suitably chosen auxiliary variables. 
In Section~\ref{SECT:SEMIDISCR}, we present an LDG semidiscrete-in-space formulation of the rewritten system and prove its entropy stability. 
In Section~\ref{SECT::FULLY-DISCRETE}, such a semidiscrete LDG formulation is combined with the backward Euler time discretization and a regularizing term to get a fully discrete scheme. Section~\ref{SECT::WELL-POSEDNESS} is devoted to the proof of the existence of discrete solutions.
In Section~\ref{SEC::ASS_CONV}, we introduce the assumptions on the regularizing term and the discrete spaces that are used to prove the convergence to semidiscrete-in-time solutions in Section~\ref{SEC::hCONVERGENCE}, and to weak solutions to the continuous problem in Section~\ref{SECT::CONVERGENCE-epsilon-tau}.
The validity of such assumptions for different cases is discussed in Section~\ref{SEC::VALIDITY}. 
Some numerical experiments in one and two dimensions are presented in Section~\ref{SECT::NUMERICAL-EXP} to assess the accuracy and entropy stability of the scheme.
We finish with some concluding remarks in Section~\ref{SECT::CONCLUSIONS}.
\section{Definition of the method \label{SECT::DG-METHOD}}
We use the following notation for functions with~$N$ scalar-valued components
and with~$N$ $d$-vector-valued components, respectively:
\begin{equation*}
\boldsymbol{\mu} = (\mu_1, \ldots, \mu_N)^{{\sf T}}, \qquad \bunderline{\boldsymbol{\mu}} = (\bunderline{\mu}_1, \ldots, \bunderline{\mu}_N)^{{\sf T}}.
\end{equation*}

For the discretization in space, we introduce a DG approximation of problem~\eqref{EQN::MODEL-PROBLEM}, where nonlinearities do not appear within differential operators or interface terms, and a discrete version of the chain rule in~\eqref{eq:chainrule}
is satisfied. 
To this aim, we
introduce the auxiliary variables $\bw$, $\bdz$, $\bdsigma$, and $\bdq$ defined by
\begin{subequations}
\label{eq:variables}
\begin{align}
\label{eq:variables-1}
    \brho & := \bu(\bw), \\
\label{eq:variables-2}
    \bdz & := -\nabla \bw, \\
\label{eq:variables-3}
    A(\brho)^{{\sf T}}s''(\brho) \bdsigma & := - A(\brho)^{{\sf T}} s''(\brho) \nabla \brho = A(\brho)^{{\sf T}} \bdz, \quad \\
\label{eq:variables-4}
    \bdq & := A(\brho) \bdsigma,
\end{align}
\end{subequations}
and rewrite problem~\eqref{EQN::MODEL-PROBLEM} as
\begin{equation*}
\label{EQN::REWRITTEN}
\begin{cases}
    \dpt \brho + \nabla \circ \bdq = \vf(\brho) & \text{in } \QT,\\
    \bdq \, \vnOmega 
    = \mathbf{0} & \text{on } \partial \Omega \times (0, T), \\
    \brho = \brho_0 & \text{on } \Omega \times \{0\}.
\end{cases}
\end{equation*}

As $A(\brho)^{{\sf T}} s''(\brho)$ is positive definite by assumption~\ref{H2a},
on the continuous level, definition~\eqref{eq:variables-3} is equivalent to $\bdsigma = - \nabla \brho$.
Moreover, from~\eqref{eq:variables-1} and~\eqref{eq:variables-2}, we have that $\bdz=-\nabla\left(s'(\brho)\right)$. Therefore,
definition~\eqref{eq:variables-3} is a reformulation of the chain rule~\eqref{eq:chainrule} in terms of the auxiliary variables, which will guarantee that a discrete version of~\eqref{eq:chainrule} suitable for the analysis of the method is satisfied.

\subsection{Semi-discretization in space \label{SECT:SEMIDISCR}}

Let~$\{\Th\}_{h > 0}$ be a family of conforming simplicial meshes of the spatial domain~$\Omega$ with maximum element diameter (mesh size) $h$. If~$d = 2, 3$, we assume that the family~$\{\Th\}_{h>0}$ satisfies the shape-regularity condition, i.e., there exists a constant~$\Upsilon > 0$ independent of~$h$ such that, for all~$K \in \Th$, 
\begin{equation}\label{EQ::SHAPEREG}
\Upsilon h_K \le \varrho_K,
\end{equation}
where~$h_K$ denotes the diameter of~$K$ and~$\varrho_K$ is the radius of the inball
of~$K$.
We denote the set of all the mesh facets in~$\Th$ by~$\Fh = \Fho \cup \FhN$, where~$\Fho$ and~$\FhN$ are the sets of internal and (Neumann) boundary facets, respectively. In the following, we will use the short-hand notation for integrals on~$\Fho$ and~$\FhN$: for~$\star \in \{\mathcal{I}, \mathcal{N}\}$,
\begin{equation*}
\int_{\Fh^{\star}} \varphi \dS := \sum_{F \in \Fh^{\star}} \int_F \varphi \dS.
\end{equation*}

We define the following piecewise polynomial spaces:
\begin{alignat*}{3}
\Sp(\Th) & := \prod_{K \in \Th} \Pp{p}{K}, & \qquad \Mp(\Th) & := \prod_{K \in \Th} \Pp{p}{K}^d,\\
\Sp(\Th)^N & := \prod_{K \in \Th} \Pp{p}{K}^N, & \qquad \Mp(\Th)^N & := \prod_{K \in \Th} \Pp{p}{K}^{N \times d},
\end{alignat*}
where~$\Pp{p}{K}$ denotes the space of scalar-valued polynomials of degree at most~$p$ on the spatial domain~$K$. We further denote by~$(\partial K)^\circ$ the union of the facets of~$K$ that belong to~$\Fho$ and define the piecewise constant function~$\mathsf{h} \in L^{\infty}(\Fho)$ as
\begin{equation}
\label{EQN::DEF-h}
\mathsf{h}(\bx) := \eta^{-1}\min\{h_{K_1}, h_{K_2}\} \quad \text{ if }\bx \in F, \text{ and~$F\in \Fho$ is shared by~$K_1, K_2\in \Th$},
\end{equation}
for some constant~$\eta>0$ independent of the mesh size.

For any element~$K\in\Th$,  let~$\vn_K$ be the unit normal~$d$-dimensional vector to~$\partial K$ pointing outside~$K$. Moreover, for each interior facet~$F \in \Fho$, we set~$\vn_F$ as one of the two unit $d$-dimensional vectors orthogonal to~$F$.
For any piecewise smooth, scalar-valued function~$\mu$ and any~$\alpha_F \in [0, 1]$, we define jumps and weighted averages, respectively, on each facet~$F \in \Fho$, where~$F = \partial K_1 \cap \partial K_2$ for some~$K_1, K_2 \in \Th$ and with~$\vn_F$ pointing outward~$K_1$, by
\begin{align*}
\jump{{\mu}}_{\sf N} := ({{\mu}}_{|_{K_1}} -  
{{\mu}}_{|_{K_2}}) \vn_F, \quad 
\mvl{\mu}_{\alpha_F}  :=  (1 - \alpha_F) \mu_{|_{K_1}} + \alpha_F \mu_{|_{K_2}}. 
\end{align*}
The subscript~${\sf N}$ in the jumps~$\jump{\cdot}_{\sf N}$ emphasizes that the definition contains the normal to the facet.
For piecewise smooth functions~$\boldsymbol{\mu}$ with~$N$ scalar-valued components,
$\jump{\boldsymbol{\mu}}_{\sf N}$ and~$\mvl{\boldsymbol{\mu}}_{\alpha_F}$ are defined componentwise. Similarly, for piecewise smooth functions $\bunderline{\boldsymbol{\mu}}$ with $N$ $d$-vector-valued components,~$\mvl{\bunderline{\boldsymbol{\mu}}}_{\alpha_F}$ is defined componentwise. 

We emphasize that, for scalar-valued functions~$\mu$, the jumps~$\jump{{\mu}}_{\sf N}$ are vector-valued functions with~$d$ components, while for functions~$\boldsymbol{\mu}$ with~$N$ scalar-valued components, $\jump{\boldsymbol{\mu}}_{\sf N}$ are~$N\times d$ tensors. 

We propose the following structure-preserving LDG-like semidiscrete formulation: for any fixed~$t \in (0, T]$, find~$(\bwh(\cdot, t), \bdzh(\cdot, t), \bdsigmah(\cdot, t), \bdqh(\cdot, t)) \in \Sp(\Th)^N\times \Mp(\Th)^N\times \Mp(\Th)^N\times \Mp(\Th)^N$ 
such that, on each element~$K \in \Th$, 
\begin{subequations}
\label{EQN::SEMI-DISCRETE}
\begin{align}
\label{EQN::SEMI-DISCRETE-1}
\int_K \bdzh 
\frob \bdpsih \dx & = -\int_{\partial K} \wFlux_h \cdot \left(\bdpsih \, \vnK\right) \dS  + \int_K \bwh \cdot (\nabla \circ \bdpsih) \dx, \\
\label{EQN::SEMI-DISCRETE-2}
\int_K A(\bu(\bwh))^{{\sf T}} s''(\bu(\bwh)) \bdsigmah 
\frob \bdvarphih \dx & = \int_K A(\bu(\bwh))^{{\sf T}} \bdzh 
\frob \bdvarphih \dx,\\
\label{EQN::SEMI-DISCRETE-3}
\int_K \bdqh 
\frob \bdthetah  \dx & = \int_K A\left(\bu(\bw_h) \right) \bdsigmah 
\frob \bdthetah\dx, \\
\label{EQN::SEMI-DISCRETE-4}
\int_K \dpt (\bu(\bwh)) \cdot \blambdah \dx + \int_{\partial K} \left(\qFlux_h \, \vnK\right)\cdot & \blambdah \dS -  \int_K  \bdqh 
\frob \nabla \blambdah \dx   = \int_K \vf\left(\bu(\bw_h)\right) \cdot \blambdah \dx, 
\end{align}
\end{subequations}
for all test functions~$(\blambdah, \bdpsih, 
\bdvarphih, \bdthetah) \in \Sp(\Th)^N\times\Mp(\Th)^N\times\Mp(\Th)^N\times\Mp(\Th)^N$, with~$\bwh(\cdot, 0)\in \Sp(\Th)^N$ an approximation of~$s'(\brho_0)$.
Here, the numerical fluxes~$\wFlux_h$ and~$\qFlux_h$ are approximations of the traces of~$\bwh$ and~$\bdqh$, respectively, on the skeleton of~$\Th$. 
They are defined on each facet~$F \in \Fh$ as
\begin{subequations}
\label{EQN::NUMERICAL-FLUXES}
\begin{align}
    \wFlux_h & := 
    \begin{cases}
\mvl{\bwh}_{\alpha_F}
         & \text{ if } F\in \Fho \text{ and } F 
         = \partial K_1 \cap \partial K_2 \text{ for some }K_1, K_2 \in \Th,\\
        \bwh & \text{ if } F \in \FhN,
    \end{cases} \\
    \qFlux_h & := 
    \begin{cases}
    \mvl{\bdqh}_{1-\alpha_F}
         + \eta_{F} 
        \jump{\bwh}_{\sf N} & \text{ if } F\in \Fho \text{ and } F = \partial K_1 \cap \partial K_2 \text{ for some }K_1, K_2 \in \Th,\\
        \boldsymbol{0} & \text{ if } F \in \FhN,
    \end{cases}
\end{align}
\end{subequations}
where the weighted-average 
parameter~$\alpha_F \in [0, 1]$ and the stabilization function~$\eta_F$ are defined on each facet~$F \in \Fho$. 
We define~$\eta_F$ 
as 
\begin{equation}\label{eq:stab}
\eta_F = \mathsf{h}_{F}^{-1}\Norm{A}_{L^{\infty}(\calD)^{N\times N}},
\end{equation}
where~$\mathsf{h}_{F}$ denotes the restriction of~$\mathsf{h}$ to~$F$.
Taking the~$L^\infty$ norm of~$A$ in~\eqref{eq:stab} may introduce additional diffusion. However, it avoids a nonlinear dependence of the stability term on~$\bu(\bwh)$.

\begin{remark}[Choice of~$\alpha_F$]
The choice of the weighted-average parameters~$\alpha_F$ has an effect on the stencil of the LDG discretization of the diffusion term. 
It is well-known that, 
with the standard choice~$\alpha_F = 1/2$ for all the internal facets in~$\Fh$, 
the degrees of freedom in a given element are coupled not only with those of the immediate neighbors but also with those of their neighbors (see, e.g., \cite[\S4.1]{Castillo:2002}). Strategies to reduce the stencil by appropriately choosing~$\vn_F$ and setting~$\alpha_F = 0$ or~$\alpha_F = 1$ are discussed in~\cite[\S3.3]{Sherwin_etal:2006} and~\cite{Castillo_2010}. 
\eremk
\end{remark}

\begin{remark}[Computation of~$\bdsigmah$]
The definition of~$\bdsigmah$ in~\eqref{EQN::SEMI-DISCRETE-2} is local. More precisely, given~$\bwh$, the construction of~$\bdsigmah$ requires only the solution of completely independent (naturally parallelizable) linear (in~$\bdsigmah$) problems on each element~$K \in \Th$. 
In each of these local problems, 
the components of~$\bdsigmah$ for the~$N$ species are coupled. 
This is a consequence of the presence of the matrices~$A(\bu(\bwh))^{{\sf T}}s''(\bu(\bwh))$ and~$A(\bu(\bwh))^{{\sf T}}$ on the left- and right-hand side integrals of~\eqref{EQN::SEMI-DISCRETE-2}, respectively.
\eremk
\end{remark}

Given fixed bases of~$\Sp(\Th)$ and~$\Mp(\Th)$, let~$\bM$, $B$, and~$S$ denote the corresponding matrix representations, respectively, of the bilinear forms
\begin{subequations}
\label{EQN::BILINEAR-FORMS}
\begin{align}
    \label{EQN::Mh}
    \bm_h(\bunderline{\zeta}_h, \dpsih) & : = 
    \int_\Omega \bunderline{\zeta}_h \cdot \dpsih \dx & & \forall \bunderline{\zeta}_h, \dpsih \in \Mp(\Th),\\
    \nonumber
    b_h(w_h, \dpsih) & := - \int_{\Fho} \mvl{w_h}_{\alpha_F} \jump{\dpsih}_{\sf N} \dS -\int_{\FhN} w_h \dpsih \cdot \vnOmega \dS  \\
    \label{EQN::Bh}
    & \qquad + \sum_{K \in \Th} \int_K w_h \nabla \cdot \dpsih \dx & & \forall (w_h, \dpsih) \in \Sp(\Th) \times \Mp(\Th),\\
    \label{EQN::Sh}
    s_h(w_h, \lambdah) & := \int_{\Fho} \eta_F \jump{w_h}_{\sf N} \cdot \jump{\lambdah}_{\sf N} \dS & &\forall w_h,\ \lambdah \in \Sp(\Th),
\end{align}
\end{subequations}
and let~$\calU_h$, $\calN_h$, $\calAt_h$, $\calA_h$, and~$\bF_h$ be the operators associated with the nonlinear functionals
\begin{subequations}
\label{EQN::NONLINEAR-FUNCTIONALS}
\begin{align}
\bu_h(\bwh, \bdphih) & := \int_\Omega \bu(\bwh) \cdot \bdphih \dx & & \forall \bwh, \bdphih \in \Sp(\Th)^N \\
\bn_h(\bwh; \bdsigmah, \bdvarphih) & :=  \int_\Omega A(\bu(\bwh))^{{\sf T}} s''(\bu(\bwh)) \bdsigmah 
\frob \bdvarphih \dx & & \forall (\bwh, \bdsigmah, \bdvarphih) \in \Sp(\Th)^N\times\Mp(\Th)^N\times\Mp(\Th)^N, \\
\bat_h(\bwh; \bdzh, \bdvarphih) & :=  \int_\Omega A(\bu(\bwh))^{{\sf T}} \bdzh 
\frob \bdvarphih \dx & & \forall (\bwh, \bdzh, \bdvarphih) \in \Sp(\Th)^N\times\Mp(\Th)^N\times\Mp(\Th)^N,\\
\ba_h(\bwh; \bdsigmah, \bdthetah) & :=  \int_\Omega A(\bu(\bwh)) \bdsigmah 
\frob \bdthetah \dx & & \forall (\bwh, \bdsigmah, \bdthetah) \in \Sp(\Th)^N\times\Mp(\Th)^N\times\Mp(\Th)^N, \\
\vf_h(\bwh, \blambdah) & := \int_\Omega \vf(\bu(\bwh)) \cdot \blambdah \dx & & \forall \bwh, \blambdah \in \Sp(\Th)^N. 
\end{align}
\end{subequations}
After summing~\eqref{EQN::SEMI-DISCRETE-1}--\eqref{EQN::SEMI-DISCRETE-4} over all the elements~$K \in \Th$, by the average-jump identity
\begin{equation}
\label{eq:flux-jump-identity}
\mvl{\lambdah}_{\alpha_F} \jump{\dpsih}_{\sf N} + \mvl{\dpsih}_{1 - \alpha_F} \cdot \jump{\lambdah}_{\sf N} = \jump{\lambdah \dpsih}_{\sf N},
\end{equation}
we get
\begin{subequations}
\begin{align}
\label{eq:semi_a}
   \sum_{i = 1}^N \bm_h(\dzetahi, \dpsihi) & = \sum_{i = 1}^N b_h(\whi, \dpsihi), 
   \\ 
   \label{eq:semi_b}
    \bn_h(\bwh; \bdsigmah, \bdvarphih) & = \bat_h(\bwh; \bdzh, \bdvarphih),\\ 
    \label{eq:semi_c}
    \sum_{i = 1}^N \bm_h(\dqhi, \dthetahi) & = \ba_h(\bwh; \bdsigmah, \bdthetah), \\ 
    \label{eq:semi_d}
    \frac{d}{dt} \bu_h(\bwh, \blambdah) + \sum_{i = 1}^N b_h(\lambdahi, \dqhi) + \sum_{i = 1}^N s_h(\whi, \lambdahi) & = \vf_h(\bwh, \blambdah).
\end{align}
\end{subequations}
The ordinary differential equation (ODE) system~\eqref{eq:semi_a}--\eqref{eq:semi_d} can be written in operator form as
\begin{subequations}
\label{EQN::NON-REDUCED-ODE-SYSTEM}
\begin{align}
\label{EQN::NON-REDUCED-ODE-SYSTEM-1}
(I_N \otimes \bM) \vZh & =  (I_N \otimes B) \vWh, \\
\label{EQN::NON-REDUCED-ODE-SYSTEM-2}
\calN_h(\vWh; \vSigmah) & =  \calAt_h(\vWh; \vZh), \\
\label{EQN::NON-REDUCED-ODE-SYSTEM-3}
(I_N \otimes \bM) \vQh & = \calA_h(\vWh; \vSigmah), \\
\label{EQN::NON-REDUCED-ODE-SYSTEM-4}
\frac{d}{dt} \calU_h(\vWh) + (I_N \otimes B^{{\sf T}}) \vQh + (I_N \otimes S) \vWh  & = \bF_h(\vWh), 
\end{align}
\end{subequations}
where~$I_N$ denotes the identity matrix of size $N$, $\otimes$ the Kronecker product
and~$\vWh,$ $\vZh,$ $\vSigmah,$ $\vQh$ are the vector representations of~$\bwh,\, \bdzh,\, \bdsigmah,\, \bdqh$, respectively. 

Since the nonlinear operators~$\calA_h$, $\calAt_h$, and~$\calN_h$ are linear with respect to their second argument, equations~\eqref{EQN::NON-REDUCED-ODE-SYSTEM-2} and~\eqref{EQN::NON-REDUCED-ODE-SYSTEM-3} can be rewritten as
\begin{equation*}
\begin{split}
\widehat{\calN}_h(\vWh) \vSigmah = \widehat{\calA_h}(\vWh)^{{\sf T}} \vZh, \qquad (I_N \otimes \bM) \vQh = \widehat{\calA_h}(\vWh) \vSigmah,
\end{split}
\end{equation*}
for some 
block-diagonal matrices $\widehat{\calN}_h(\vWh)$ and $\widehat{\calA_h}(\vWh)$. Moreover, due to Assumption~\ref{H2a}, the matrix~$\widehat{\calN}_h(\vWh)$ is positive definite.

Eliminating~$\vZh$ and~$\vQh$, we can write the ODE system~\eqref{EQN::NON-REDUCED-ODE-SYSTEM} in the compact form
\begin{subequations}
\label{EQN::SEMI-DISCRETE-FORMULATION-ODE}
\begin{align}
\label{EQN::SEMI-DISCRETE-FORMULATION-ODE-1}
    \Nh(\vWh; \vSigmah) & = \calAt_h(\vWh; (I_N \otimes \bM^{-1} B) \vWh), \\
    \label{EQN::SEMI-DISCRETE-FORMULATION-ODE-2}
    \frac{d}{dt} \calU_h(\vWh) + \left(I_N \otimes B^{{\sf T}}\bM^{-1} \right) \calA_h\left(\vWh; \vSigmah \right) & + (I_N \otimes S) \vWh = \bF_h(\vWh).
\end{align}    
\end{subequations}

In the following Lemma~\ref{LEMMA::BILINEAR-FORMS-NONLINEAR-FUNCTIONALS}, we prove some properties of the bilinear forms and nonlinear functionals defined in~\eqref{EQN::BILINEAR-FORMS} and~\eqref{EQN::NONLINEAR-FUNCTIONALS}, respectively. From here on, we denote by~$\nabla_h(\cdot)$ the elementwise~$\nabla(\cdot)$ operator, and by~$\Norm{\cdot}_{[L^2(\Omega)^d]^N}$ the~$L^2(\Omega)$ norm of functions with~$N$ $d$-vector-valued components.

\begin{lemma}
\label{LEMMA::BILINEAR-FORMS-NONLINEAR-FUNCTIONALS}
The bilinear forms defined in~\eqref{EQN::BILINEAR-FORMS} and the nonlinear functionals defined in~\eqref{EQN::NONLINEAR-FUNCTIONALS} satisfy the following continuity bounds:
\begin{subequations}
\label{EQN::BOUND-BILINEAR-FORMS}
\begin{align}
\label{EQN::BOUND-BILINEAR-FORMS-M}
    \sum_{i = 1}^N \bm_h(\dzetahi, \dpsihi) & \leq \Norm{\bdzh}_{[L^2(\Omega)^d]{}^N} \Norm{\bdpsih}_{[L^2(\Omega)^d]{}^N}, \\
\label{EQN::BOUND-BILINEAR-FORMS-B}
    \sum_{i = 1}^N b_h(\whi, \dpsihi) & \lesssim \left(
    \Norm{ \nabla_h \bwh}_{[L^2(\Omega)^d]^N}^2
    + \bigNorm{
    \mathsf{h}^{-\frac12}\jump{\bwh}_{\sf N}}_{[L^2(\Fho)^{d}]{}^N}^2 \right)^{\frac12}
    \Norm{\bdpsih}_{[L^2(\Omega)^d]{}^N}, \\
\label{EQN::BOUND-BILINEAR-FORMS-S}
    \sum_{i = 1}^N s_h(\whi, \lambdahi) & \lesssim \bigNorm{\eta_F^{\frac12} \jump{\bwh}_{\sf N}}_{[L^2(\Fho)^{d}]{}^N} \bigNorm{\eta_F^{\frac12} \jump{\blambdah}_{\sf N}}_{[L^2(\Fho)^{d}]{}^N}, \\
    \bu_h(\bwh, \bdphih) & \lesssim \Norm{\bdphih}_{L^2(\Omega)^N}, \\ 
\label{EQN::BOUND-BILINEAR-FORMS-atilde}
    \bat_h(\bwh; \bdzh, \bdvarphih) & \lesssim \Norm{\bdzh}_{[L^2(\Omega)^d]^N} \Norm{\bdvarphih}_{[L^2(\Omega)^d]^N},  \\
    \ba_h(\bwh; \bdsigmah, \bdthetah) & \lesssim \Norm{\bdsigmah}_{[L^2(\Omega)^d]^N} \Norm{\bdthetah}_{[L^2(\Omega)^d]^N},  \\
    \label{EQN::BOUND-BILINEAR-FORMS-F}
    \vf_h(\bwh, \blambdah) & \lesssim \Norm{\blambdah}_{L^2(\Omega)^N}
\end{align}
\end{subequations}
for all functions in the corresponding discrete spaces, with hidden constants independent of the mesh size $h$.
Moreover, the nonlinear functional~$\bn_h(\ \cdot;\  \cdot,\ \cdot)$ satisfies the following coercivity property: for all $\bwh\in \Sp(\Th)^N$, 
\begin{equation}
\label{EQN::COERCIVITY-Nh}
    \bn_h(\bwh; \bdsigmah, \bdsigmah) \geq \gamma \Norm{\bdsigmah}_{[L^2(\Omega)^d]{}^N}^2
    \qquad\forall \bdsigmah\in \Mp(\Th)^N,
\end{equation}
where~$\gamma$ is the constant in Assumption~\ref{H2a}.
\end{lemma}
\begin{proof}
The coercivity property~\eqref{EQN::COERCIVITY-Nh} follows from Assumption~\ref{H2a}. 
For~\eqref{EQN::BOUND-BILINEAR-FORMS-B}, the 
average--jump identity~\eqref{eq:flux-jump-identity} and integration by parts give
\begin{align*}
b_h(\whi, \dpsihi) = - 
\int_{\Omega} \nabla_h \whi \cdot \dpsihi \dx
+ \int_{\Fho} \jump{\whi}_{\sf N} \cdot \mvl{\dpsihi}_{1 - \alpha_F} \dS.
\end{align*}
We estimate the volume term on the right-hand side with the Cauchy--Schwarz inequality. 
For the interface term, on each~$F\in\Fho$, we use the weighted Cauchy--Schwarz inequality with weights $\eta_F^{1/2}$ and $\eta_F^{-1/2}$ and the inverse trace inequality for $\dpsihi$, taking into account that, due to the definition of~$\eta_F$ in~\eqref{eq:stab}, $\eta_F^{-1/2}\lesssim\min\{h_{K_1}^{1/2},h_{K_2}^{1/2}\}$, where $K_1$ and $K_2$ are the two elements sharing~$F$. 
Estimate~\eqref{EQN::BOUND-BILINEAR-FORMS-B} readily follows.
The remaining bounds in~\eqref{EQN::BOUND-BILINEAR-FORMS} follow from Assumptions~\ref{H1}, the boundedness of~$\bu$ (see~\ref{H2}), and the Cauchy--Schwarz inequality.  
\end{proof}

We prove that, given~$\bwh\in \Sp(\Th)^N$, equations~\eqref{eq:semi_a} and~\eqref{eq:semi_b} define $\bdsigmah\in\Mp(\Th)^N$ in a unique way.
In vector representation, this entails that, given~$\vWh$, equation~\eqref{EQN::SEMI-DISCRETE-FORMULATION-ODE-1} defines~$\vSigmah=\vSigmah(\vWh)$ in a unique way.

\begin{proposition}\label{prop:elim_Sigma}
Given~$\bwh\in \Sp(\Th)^N$, equations~\eqref{eq:semi_a} and~\eqref{eq:semi_b} define $\bdsigmah\in\Mp(\Th)^N$ in a unique way. Moreover, $\bdsigmah$ satisfies
\begin{equation}\label{eq:stab_sigma}
\Norm{\bdsigmah}_{[L^2(\Omega)^d]{}^N}^2\lesssim
\Norm{ \nabla_h \bwh}_{[L^2(\Omega)^d]^N}^2
+ \bigNorm{
\mathsf{h}^{-\frac12}\jump{\bwh}_{\sf N}}_{[L^2(\Fho)^{d}]{}^N}^2,
\end{equation}
with hidden constant independent of the mesh size $h$.
\end{proposition}
\begin{proof}
\emph{(i)} Given $\bwh\in \Sp(\Th)^N$, there exists a unique~$\bdzh\in \Mp(\Th)^N$ solution to~\eqref{eq:semi_a}. Moreover, $\bdzh$ satisfies
\begin{equation*}\label{eq:stab(i)}
\Norm{\bdzh}_{[L^2(\Omega)^d]{}^N}^2\lesssim
\Norm{ \nabla_h \bwh}_{[L^2(\Omega)^d]^N}^2
+ \bigNorm{
\mathsf{h}^{-\frac12}\jump{\bwh}_{\sf N}}_{[L^2(\Fho)^{d}]{}^N}^2.
\end{equation*}
This follows from the Lax--Milgram lemma, which is applicable owing to~\eqref{EQN::BOUND-BILINEAR-FORMS-M} and~\eqref{EQN::BOUND-BILINEAR-FORMS-B}.\\
\emph{(ii)} Given $\bwh\in \Sp(\Th)^N$ and $\bdzh\in \Mp(\Th)^N$ from step \emph{(i)},
there exists a unique~$\bdsigmah = \bdsigmah(\bwh)\in\Mp(\Th)^N$ solution to~\eqref{eq:semi_b} that satisfies~\eqref{eq:stab_sigma}.
This follows again from the Lax--Milgram lemma, which is applicable owing to~\eqref{EQN::COERCIVITY-Nh}, \eqref{EQN::BOUND-BILINEAR-FORMS-atilde}, 
and~\eqref{EQN::BOUND-BILINEAR-FORMS-B}.
\end{proof}

We prove the following space-discrete entropy inequality, which is a discrete version of inequality~\eqref{EQN::CONTINUOUS-ENTROPY-STABILITY}.

\begin{proposition}\label{prop:entropy_ineq}
Any solution~$(\bwh, \bdsigmah)$ to the semidiscrete formulation~\eqref{EQN::SEMI-DISCRETE-FORMULATION-ODE} satisfies the following entropy inequality for all~$\tau \in (0, T]$: 
\begin{equation*}
\label{EQN::ENTROPY-STABILITY}
    \int_{\Omega} s(\bu(\bwh(\bx, \tau)) \dx + \gamma \int_{0}^{\tau}\Norm{\bdsigmah}_{[L^2(\Omega)^d]^N}^2\, \dt + \int_{0}^{\tau} \bigNorm{\eta_F^{\frac12} \jump{\bwh}_{\sf N}}_{[L^2(\Fho)^{d}]^N}^2\,\dt
    \leq \int_\Omega s(\bu(\bwh(\bx, 0))) \dx + \tau C_f |\Omega|.
\end{equation*}
\end{proposition}
\begin{proof}
Let~$\tau \in (0, T]$. Multiplying~\eqref{EQN::SEMI-DISCRETE-FORMULATION-ODE-2} by~$\vWh$ we get
\begin{equation}
\label{EQN::TIMES-WH}
        \DotProd{\frac{d}{dt} \calU_h(\vWh)}{\vWh} + \DotProd{\left(I_N \otimes B^{{\sf T}}\bM^{-1} \right) \calA_h\left(\vWh; \vSigmah \right)}{\vWh} + \DotProd{(I_N \otimes S) \vWh}{\vWh}  = \DotProd{\bF_h(\vWh)}{\vWh}.
\end{equation}
We treat each term in identity~\eqref{EQN::TIMES-WH} separately.

Since~$\bu = (s')^{-1}$, we can write~$\bwh$ as~$s'(\bu(\bwh))$. This, together with the chain rule, gives
\begin{align}
    \nonumber
    \DotProd{\frac{d}{dt} \calU_h(\vWh)}{\vWh} & = \int_\Omega \dpt (\bu(\bwh)) \cdot \bwh \dx = \int_\Omega \dpt(\bu(\bwh)) \cdot s'(\bu(\bwh)) \dx \\
    \label{EQN::AUX-ENTROPY-1}
    &    = \int_\Omega \dpt (s(\bu(\bwh))) \dx.
\end{align}
By using standard algebraic manipulations, equation~\eqref{EQN::SEMI-DISCRETE-2}, and Assumption~\ref{H2a}, we obtain
\begin{align}
    \nonumber
    \DotProd{\left(I_N \otimes B^{{\sf T}}\bM^{-1} \right) \calA_h\left(\vWh; \vSigmah \right)}{\vWh} & = \DotProd{\calA_h\left(\vWh; \vSigmah \right)}{\left(I_N \otimes \bM^{-1} B \right)  \vWh} \\
    \nonumber
    & = \int_\Omega A(\bu(\bwh)) \bdsigmah 
    \frob \bdzh \dx 
    = \int_\Omega A(\bu(\bwh))^{{\sf T}}  \bdzh 
    \frob \bdsigmah \dx \\
    \label{EQN::AUX-ENTROPY-2}
    & = \int_\Omega A(\bu(\bwh))^{{\sf T}} s''(\bu(\bwh)) \bdsigmah 
    \frob \bdsigmah \dx \geq \gamma \Norm{\bdsigmah}_{[L^2(\Omega)^d]^N}^2.
\end{align}
By the definition of the bilinear form~$s_h(\cdot, \cdot)$ in~\eqref{EQN::Sh}, we have
\begin{equation}
\label{EQN::AUX-ENTROPY-3}
    \DotProd{(I_N \otimes S) \vWh}{\vWh} = \bigNorm{\eta_F^{\frac12} \jump{\bwh}_{\sf N}}_{[L^2(\Fho)^{d}]^N}^2.
\end{equation}
Finally, the following upper bound follows from Assumption~\ref{H2b}:
\begin{equation}
\label{EQN::AUX-ENTROPY-4}
\begin{split}
    \DotProd{\bF_h(\vWh)}{\vWh} & = \int_\Omega \vf(\bu(\bwh)) \cdot \bwh \dx = \int_\Omega \vf(\bu(\bwh)) \cdot s'(\bu(\bwh)) \dx \leq C_f |\Omega|.
\end{split}
\end{equation}
Integrating in time~\eqref{EQN::TIMES-WH} from~$0$ to~$\tau$, and using bounds~\eqref{EQN::AUX-ENTROPY-1}, \eqref{EQN::AUX-ENTROPY-2}, \eqref{EQN::AUX-ENTROPY-3}, and~\eqref{EQN::AUX-ENTROPY-4}, we obtain the desired result. 
\end{proof}

\begin{remark}
The definition of~$\bu$ in Assumption~\ref{H2} guarantees that, in the semidiscrete formulation~\eqref{EQN::SEMI-DISCRETE}, the argument~$\bu(\bwh)$ in the nonlinear terms~$A(\cdot)$, $s''(\cdot)$, and~$\vf(\cdot)$ takes values in~$\calD$. 
Such a  property is essential in the existence and convergence results in Theorems~\ref{thm:existence_discrete} and~\ref{THM::H-CONVERGENCE}, and could not be guaranteed if a discrete approximation~$\brho_h \in \Sp(\Th)^N$ of~$\brho = \bu(\bw)$ were used instead.
\eremk
\end{remark}
\begin{remark}[Constant diffusion tensor~$A$]
If~$A$ is a constant diffusion tensor, the semidiscrete formulation~\eqref{EQN::SEMI-DISCRETE} reduces to 
\begin{subequations}
    \begin{align}
    \label{eq:aux1}
    \tNh(\vWh; \vSigmah) & = (I_N \otimes B) \vWh, \\
    \label{eq:aux2}
    \frac{d}{dt} \calU_h(\vWh) + \left(A \otimes B^{{\sf T}} \bM^{-1}\right)\vSigmah + (I_N \otimes S) \vWh & = \bF_h(\vWh),
\end{align}
\end{subequations}
where~$\tNh(\cdot, \cdot)$ is the operator associated with the nonlinear functional
$$ \widetilde{\boldsymbol{n}}_h(\bwh; \bdsigmah, \bdvarphih) := \int_\Omega s''(\bu(\bwh)) \bdsigmah 
\frob \bdvarphih \dx \qquad \forall (\bwh, \bdsigmah, \bdvarphih) \in \Sp(\Th)^N\times \Mp(\Th)^N\times \Mp(\Th)^N.$$
Moreover, if the entropy density is given by
$$s(\brho) = \sum_{i = 1}^N s_i(\rho_i),$$
matrix~$s''(\brho)$ is diagonal.
In such a case, the~$N$ components of~$\bdsigmah$ are no longer coupled \sg{in~\eqref{eq:aux1}}.
\eremk
\end{remark}
\begin{remark}[Differential operators]
Rewriting model~\eqref{EQN::MODEL-PROBLEM} in terms of the auxiliary variables~$\bwh$, $\bdzh$, $\bdsigmah$, and~$\bdqh$ allows us to localize the influence of the nonlinear terms in the semidiscrete formulation~\eqref{EQN::SEMI-DISCRETE}. 
More precisely, nonlinearities do not appear in interface terms, but only on local volume integrals. 
Consequently, the only non-block-diagonal operators in the method that have to be computed are the scalar matrices~$B$ and~$S$, which are the standard LDG  gradient and stability matrices, respectively. The resulting method is more efficient,
compared to interior-penalty discretizations with nonlinearities under the differential operators (and thus in the interface terms); \emph{cf.} \cite{Bonizzoni_Braukhoff_Jungel_Perugia:2020,Corti_Bonizzoni_Antoniett_2023}.
\eremk
\end{remark}
\begin{remark}[Discrete positivity and boundedness]
Obtaining a discrete approximation~$\brho_h \in \Sp(\Th)^N$ that respects the positivity (or boundedness) of the physical unknown~$\brho$ in a strong sense (i.e., pointwise) is a very difficult task. 
In fact, for high-order approximations, even if~$\brho_h$ is enforced to satisfy such bounds on the nodes (weak positivity), the physical constraints might still be violated; \emph{cf.} \cite{Barrenechea_Georgoulis_Pryer_Vesser:2023}. 
Our method provides an approximate solution~$\widetilde{\brho}_h = \bu(\bwh) \not\in \Sp(\Th)^N$ that satisfies the physical constraints for any degree of approximation.
\eremk
\end{remark}


\subsection{Fully discrete scheme \label{SECT::FULLY-DISCRETE}}

We discretize the ODE system~\eqref{EQN::SEMI-DISCRETE-FORMULATION-ODE} in time by 
the backward Euler method on a partition of the time interval~$(0,T)$ into~$N_t$ subintervals~$\{(t_{n-1},t_n)\}_{n=1}^{N_t}$, with $t_0=0$, $t_{N_t}=T$ and time steps $\tau_n:=t_n-t_{n-1}>0$. 
Moreover, we add a regularizing term with multiplicative parameter~$\varepsilon>0$, which is defined  
in terms of a symmetric, $h$-uniformly positive definite matrix~$C$ only depending on the space discretization. The parameter $\varepsilon$ does not depend on the diffusion matrix nor on $C$, and can be choosen arbitrarily small.
Such a regularizing term is essential in the existence and convergence results in Theorems~\ref{thm:existence_discrete} and~\ref{THM::H-CONVERGENCE}.
The fully discrete, regularized method
reads as follows:
\begin{itemize}
\item define~$\mathbf{R}_h^{0}$
as the vector representation of the~$L^2(\Omega)$-orthogonal projection of~$\brho_0$ in~$\Sp(\Th)^N$ denoted by~$\Pi_p^0\brho_0$,
and compute $(\vWh^{\varepsilon,1},\vSigmah^{\varepsilon,1})$ by solving
\begin{subequations}
\label{EQN::FULLY-DISCRETE-SCHEME-INIT}
\begin{align}
\Nh(\vWh^{\varepsilon,1}; \vSigmah^{\varepsilon,1}) = \calAt_h(\vWh^{\varepsilon,1}; (I_N \otimes \bM^{-1} B) \vWh^{\varepsilon,1}),  \\
    \label{EQN::FULLY-DISCRETE-SCHEME-2-INIT}
    \varepsilon \tau_1(I_N \otimes C)\vWh^{\varepsilon,1} +\left(\calU_h(\vWh^{\varepsilon,1})  - \mathbf{R}_h^{0}\right)  \hspace{7cm} \nonumber\\
     + \tau_1\left(I_N \otimes B^{{\sf T}}\bM^{-1} \right) \calA_h\left(\vWh^{\varepsilon,1}; \vSigmah^{\varepsilon,1} \right) + \tau_{1}(I_N \otimes S) \vWh^{\varepsilon,1}  
 =\tau_1 \bF_h(\vWh^{\varepsilon,1});
\end{align}
\end{subequations} 
\item for~$n=1,\ldots, N_t-1$, compute $(\vWh^{\varepsilon,n+1},\vSigmah^{\varepsilon,n+1})$ by solving
\begin{subequations}
\label{EQN::FULLY-DISCRETE-SCHEME}
\begin{align}
\Nh(\vWh^{\varepsilon,n+1}; \vSigmah^{\varepsilon,n+1})  = \calAt_h(\vWh^{\varepsilon,n+1}; (I_N \otimes \bM^{-1} B) \vWh^{\varepsilon,n+1}),  
    \\
    \label{EQN::FULLY-DISCRETE-SCHEME-2}
    \varepsilon \tau_{n+1}(I_N \otimes C)\vWh^{\varepsilon,n+1} 
+\left(\calU_h(\vWh^{\varepsilon,n+1})  -\calU_h(\vWh^{\varepsilon,n}) \right) \hspace{6cm} \nonumber\\
    + \tau_{n+1}\left(I_N \otimes B^{{\sf T}}\bM^{-1} \right) \calA_h\left(\vWh^{\varepsilon,n+1}; \vSigmah^{\varepsilon,n+1} \right) + \tau_{n + 1}(I_N \otimes S) \vWh^{\varepsilon,n+1}  
    =\tau_{n + 1}\bF_h(\vWh^{\varepsilon,n+1}).
\end{align}
\end{subequations} 
\end{itemize}

The symmetric, positive definite matrix $C$ defines a scalar product and a norm in $\Sp(\Th)$: given $\bwh$ and $\bvh$ in $\Sp(\Th)^N$ with vector representations $\vWh$ and $\vVh$, respectively, we set
\begin{equation}
\label{EQN::C-BILINEAR-FORM}
\sum_{i=1}^N c_h(w_{h,i},v_{h,i})
:=\bc_h(\bwh,\bvh)
:=\DotProd{(I_N \otimes C)\vWh}{\vVh} \quad \text{ and } \quad 
\Norm{\bwh}_{C}^2:=\DotProd{(I_N \otimes C)\vWh}{\vWh}.
\end{equation}

\begin{remark}[Discrete initial condition]
\label{rem:firststep}
The proposed fully discrete scheme imposes the initial condition~$\brho_0$ only weakly.
This subtle yet crucial difference from the standard backward Euler scheme allows us to naturally handle initial conditions that could not be accommodated if an initial datum~$\bw_0 = s'(\brho_0)$ were to be imposed strongly and~$\brho_0$ took values on~$\partial\calD$. This property significantly improves the stability of the method in such situations and avoids the need for artificial initial data employed in previous approaches.

Therefore, the use of~$\mathbf{R}_h^0$ in the first step of the fully discrete scheme~\eqref{EQN::FULLY-DISCRETE-SCHEME-INIT}--\eqref{EQN::FULLY-DISCRETE-SCHEME} has two motivations:
\begin{itemize}
\item it allows for an initial datum~$\brho_0(\bx) \in \overline{\calD}$ for a.e.~$\bx\in\Omega$, whereas~$\bw_0 = s'(\brho_0)$ may be not well defined if~$\brho_0$ takes values on~$\partial \calD$;
\item it leads to an~$h$-independent bound in the discrete entropy inequality in Theorem~\ref{THM::ENTROPY-INEQUALITY} below. 
\eremk
\end{itemize}
\end{remark}

\begin{remark}[Reduced nonlinear system]\label{rem:oneunknown}
Setting~$\widehat{\Eh}(\vWh) := \widehat{\calA_h}(\vWh) \widehat{\calN}_h(\vWh)^{-1} \widehat{\calA_h}(\vWh)^{{\sf T}}$, 
the fully discrete scheme~\eqref{EQN::FULLY-DISCRETE-SCHEME-INIT}--\eqref{EQN::FULLY-DISCRETE-SCHEME} can be written in terms of the~$\vWh^{\varepsilon,n+1}$-unknown only as follows:
\begin{itemize}
\item define~$\mathbf{R}_h^{0}$
as the vector representation of~$\Pi_p^0\brho_0$, and compute $\vWh^{\varepsilon,1}$ by solving 
\[
\begin{split}
    &\varepsilon \tau_1(I_N \otimes C)\vWh^{\varepsilon,1}
+\left(\calU_h(\vWh^{\varepsilon,1})  - \mathbf{R}_h^{0}\right) \\
&\qquad
+\tau_1\left[
(I_N \otimes B^{{\sf T}} \bM^{-1}) \widehat{\Eh}(\vWh^{\varepsilon,1})(I_N \otimes \bM^{-1} B) + (I_N \otimes S)
\right]\vWh^{\varepsilon,1}
=\tau_1\bF_h(\vWh^{\varepsilon,1});
\end{split}
\]
\item for~$n=1,\ldots, N_t-1$, compute $\vWh^{\varepsilon,n+1}$ by solving
\[
\begin{split}
&\varepsilon \tau_{n+1}(I_N \otimes C)\vWh^{\varepsilon,n+1} 
    +\left(\calU_h(\vWh^{\varepsilon,n+1})  -\calU_h(\vWh^{\varepsilon,n}) \right)\\
&\qquad    +\tau_{n+1}\left[
(I_N \otimes B^{{\sf T}} \bM^{-1}) \widehat{\Eh}(\vWh^{\varepsilon, n + 1})(I_N \otimes \bM^{-1} B) + (I_N \otimes S)
\right]\vWh^{\varepsilon,n+1}
=\tau_{n+1}\bF_h(\vWh^{\varepsilon,n+1}).    
\end{split}
\]
\end{itemize}
Due to the structure of~$\widehat{\calN_h}(\vWh)$ and~$\widehat{\calA_h}(\vWh)$, the matrix~$\widehat{\Eh}(\vWh)$ is block diagonal.
\eremk
\end{remark}

\begin{remark}[Higher-order time discretizations]
The convexity of the entropy in Assumption~\ref{H2} allows for a proof of a discrete entropy inequality for the first-order backward Euler scheme via the elementary inequality $s'(\by)\cdot(\by-\bx)\ge s(\by)-s(\bx)$ for all $\bx,\by\in\calD$. 
The use of arbitrary higher-order time discretizations and their analysis is more delicate. We believe that the natural extension of the proposed scheme is to use the entropy dissipative, high-order discontinuous Galerkin time discretization proposed in \cite{Egger:2019}. However, the analysis of the resulting method is considerably more challenging, as it requires the development of new \emph{ad hoc} discrete compactness results. We also refer to~\cite{Jungel_Miliv_2015,Jungel_Vetter_2024} for structure-preserving second-order BDF (Backward Differentiation Formula) schemes for certain classes of cross-diffusion systems. These schemes are based on entropy functions~$s$ that do not satisfy Hypothesis (H2), i.e., they may fail to preserve positivity or boundedness. 
\eremk
\end{remark}

\section{Discrete entropy stability and existence of discrete solutions  \label{SECT::WELL-POSEDNESS}}

In this section, we prove the entropy stability and existence of solutions to the fully discrete, regularized problem~\eqref{EQN::FULLY-DISCRETE-SCHEME-INIT}--\eqref{EQN::FULLY-DISCRETE-SCHEME}.

\begin{theorem}[Discrete entropy inequalities]
\label{THM::ENTROPY-INEQUALITY}
Any solution~$\{\vWh^{\varepsilon, n}\}_{n = 1}^{N_t}$ to problem~\eqref{EQN::FULLY-DISCRETE-SCHEME-INIT}--\eqref{EQN::FULLY-DISCRETE-SCHEME} satisfies
\begin{align}
\nonumber
\varepsilon  \tau_{1} \Norm{\bwh^{\varepsilon,1}}_C^2  
 + \int_{\Omega} s(\bu(\bwh^{\varepsilon,1})) \dx + \gamma  \tau_{1} \Norm{\bdsigmah^{\varepsilon, 1}}_{[L^2(\Omega)^d]^N}^2 
 +  \tau_{1} \bigNorm{\eta_F^{\frac12} \bigjump{\bwh^{\varepsilon,1}}_{\sf N}}_{[L^2(\Fho)^d]^N}^2
 & \\
 \leq  \int_{\Omega} s(\brho_0) \dx + C_f \tau_1 \abs{\Omega},
\label{EQN::DISCRETE-ENTROPY1-INIT}
\\
\nonumber
\varepsilon  \tau_{n+1} \Norm{\bwh^{\varepsilon, n + 1}}_C^2  
 + \int_{\Omega} s(\bu(\bwh^{\varepsilon, n+1})) \dx + \gamma  \tau_{n+1} \Norm{\bdsigmah^{\varepsilon, n+1}}_{[L^2(\Omega)^d]^N}^2 
 +  \tau_{n+1} \bigNorm{\eta_F^{\frac12} \bigjump{\bwh^{\varepsilon, n + 1}}_{\sf N}}_{[L^2(\Fho)^d]^N}^2
 &\\
 \leq  \int_{\Omega} s(\bu(\bwh^{\varepsilon, n})) \dx + C_f \tau_{n+1} \abs{\Omega},
 \label{EQN::DISCRETE-ENTROPY1}
\end{align}
and
\begin{align}\label{EQN::DISCRETE-ENTROPY2}
\varepsilon \sum_{n = 0}^{N_t - 1} \tau_{n+1} \Norm{\bwh^{\varepsilon, n + 1}}_C^2  
 &+ \int_{\Omega} s(\bu(\bwh^{\varepsilon, N_t})) \dx 
 + \gamma \sum_{n = 0}^{N_t - 1} \tau_{n+1}\Norm{\bdsigmah^{\varepsilon, n+1}}_{[L^2(\Omega)^d]^N}^2 \\
 &+ \sum_{n = 0}^{N_t - 1} \tau_{n+1} \bigNorm{\eta_F^{\frac12} \bigjump{\bwh^{\varepsilon, n + 1}}_{\sf N}}_{[L^2(\Fho)^d]^N}^2 
 \leq  \int_{\Omega} s(\brho_0) \dx + C_f T \abs{\Omega}. \nonumber
\end{align}
\end{theorem}

\begin{proof}
We multiply~\eqref{EQN::FULLY-DISCRETE-SCHEME-2-INIT} by~$\vWh^{\varepsilon,1}$. For the first two terms, using the~$L^2(\Omega)$-orthogonality of~$\Pi_p^0$, the fact that~$\bu = (s')^{-1}$, and the convexity of~$s$,
we find that
\begin{align*}
    \varepsilon \tau_1\DotProd{(I_N \otimes C)\vWh^{\varepsilon,1}}{\vWh^{\varepsilon,1}}
    +
\DotProd{\calU_h(\vWh^{\varepsilon,1}) - \mathbf{R}_h^{0}}{\vWh^{\varepsilon,1}} 
        & = 
        \varepsilon \tau_1\Norm{\bwh^{\varepsilon,1}}_{C}^2
        +\int_{\Omega} \left(\bu(\bwh^{\varepsilon,1}) - \Pi_p^0\brho_0\right)\cdot \bwh^{\varepsilon,1} \dx\\
        & = 
        \varepsilon \tau_1\Norm{\bwh^{\varepsilon,1}}_{C}^2
        +\int_{\Omega} \left(\bu(\bwh^{\varepsilon,1}) - \brho_0\right)\cdot \bwh^{\varepsilon,1} \dx\\
        &  = \varepsilon \tau_1\Norm{\bwh^{\varepsilon,1}}_{C}^2
        +\int_{\Omega} \left(\bu(\bwh^{\varepsilon,1}) - \brho_0\right)\cdot s'(\bu(\bwh^{\varepsilon,1})) \dx 
        \\
        & \geq 
        \varepsilon \tau_1\Norm{\bwh^{\varepsilon,1}}_{C}^2
        +
        \int_{\Omega} \left(s(\bu(\bwh^{\varepsilon,1})) - s(\brho_0) \right) \dx.
    \end{align*}
For the remaining terms, proceeding exactly as in the proof of Proposition~\ref{prop:entropy_ineq}, we obtain
\begin{align*}
\tau_1\DotProd{\left(I_N \otimes B^{{\sf T}}\bM^{-1} \right) \calA_h(\vWh^{\varepsilon,1}; \vSigmah^{\varepsilon,1})}{\vWh^{\varepsilon,1}}
&\ge \tau_1 \gamma \Norm{\bdsigmah^{\varepsilon,1}}_{[L^2(\Omega)^d]^N}^2,\\
\tau_1\DotProd{(I_N \otimes S) \vWh^{\varepsilon,1}}{\vWh^{\varepsilon,1}}
&=\tau_1 \bigNorm{\eta_F^{\frac12} \jump{\bwh^{\varepsilon,1}}_{\sf N}}_{[L^2(\Fho)^d]^N}^2,\\
\tau_1\DotProd{\bF_h(\vWh^{\varepsilon,1})}{\vWh^{\varepsilon,1}}
&\le \tau_1 C_f |\Omega|.
\end{align*}
All the above estimates immediately give~\eqref{EQN::DISCRETE-ENTROPY1-INIT}.
In order to prove~\eqref{EQN::DISCRETE-ENTROPY1},
we proceed as above. We write explicitly the estimate of the first two terms for completeness:
\begin{align*}
    \varepsilon \tau_{n+1}\DotProd{(I_N \otimes C)\vWh^{\varepsilon,n+1}}{\vWh^{\varepsilon,n+1}}
    +&
\DotProd{\calU_h(\vWh^{\varepsilon,n+1}) - \calU_h(\vWh^{\varepsilon,n})}{\vWh^{\varepsilon,1}} \\
        & = 
        \varepsilon \tau_{n+1}\Norm{\bwh^{\varepsilon,n+1}}_{C}^2
        +\int_{\Omega} \left(\bu(\bwh^{\varepsilon,n+1}) - \bu(\bwh^{\varepsilon,n})\right) \cdot \bwh^{\varepsilon,n+1} \dx\\
        &  = \varepsilon \tau_{n+1}\Norm{\bwh^{\varepsilon,n+1}}_{C}^2
        +\int_{\Omega} \left(\bu(\bwh^{\varepsilon,n+1}) - \bu(\bwh^{\varepsilon,n})\right) \cdot s'(\bu(\bwh^{\varepsilon,n+1})) \dx 
        \\
        & \geq 
        \varepsilon \tau_{n+1}\Norm{\bwh^{\varepsilon,n+1}}_{C}^2
        +
        \int_{\Omega} \left(s(\bu(\bwh^{\varepsilon,n+1})) - s(\bu(\bwh^{\varepsilon,n})) \right) \dx.
    \end{align*}
Finally, to obtain~\eqref{EQN::DISCRETE-ENTROPY2}, we 
multiply~\eqref{EQN::FULLY-DISCRETE-SCHEME-2-INIT} and~\eqref{EQN::FULLY-DISCRETE-SCHEME-2} by~$\vWh^{\varepsilon, 1}$ and~$\vWh^{\varepsilon, n + 1}$, respectively,
sum over all indices~$n = 0, \ldots, N_t - 1$, and use
the same arguments as above.
\end{proof}

\begin{theorem}\label{thm:existence_discrete}
For~$n = 0, \ldots, N_t - 1$, there exists a solution~$\vWh^{\varepsilon,n+1}$ to problem~\eqref{EQN::FULLY-DISCRETE-SCHEME-INIT} ($n=0$) or to problem~\eqref{EQN::FULLY-DISCRETE-SCHEME} ($n\ge 1$). 
\end{theorem}
\begin{proof}
We begin with~$n=0$. Consider the linearized problem: given $\vVh\in \mathbb{R}^{\dim\left(\Sp(\Th)^N\right)}$, find $\vWh^{\varepsilon}\in \mathbb{R}^{\dim\left(\Sp(\Th)^N\right)}$ such that
\begin{equation*}
\varepsilon \tau_1(I_N \otimes C)\vWh^{\varepsilon}
=
-\calU_h(\vVh)+\mathbf{R}_h^0
-\tau_1\left(I_N \otimes B^{{\sf T}}\bM^{-1} \right) \calA_h\left(\vVh; \vSigmah(\vVh) \right)
- \tau_1(I_N \otimes S) \vVh
    +\tau_1\bF_h(\vVh),
\end{equation*}
where $\vSigmah(\vVh)$ is the unique solution to
\begin{equation*}
    \Nh\left(\vVh; \vSigmah(\vVh)\right) = \calAt_h(\vVh; (I_N \otimes \bM^{-1} B) \vVh);
\end{equation*}
see the text above Proposition~\ref{prop:elim_Sigma}. 
As $C$ is positive definite, $\vWh^{\varepsilon}$ is uniquely defined. This defines a function 
\[
\Phi:\Sp(\Th)^N \to \Sp(\Th)^N,\qquad
\bvh \mapsto \bwh^{\varepsilon},
\]
where $\bvh\in \Sp(\Th)^N$ and $\bwh^{\varepsilon}\in \Sp(\Th)^N$ are the functions whose coefficient vectors are $\vVh$ and $\vWh^{\varepsilon}$, respectively.
Due to the continuity of $A$, $\vf$, and $\bu$, and to estimate~\eqref{eq:stab_sigma},
$\Phi$ is continuous.

We apply the Schaefer fixed-point theorem~\cite[Thm.~4.3.2]{Smart_1974}
to prove that $\Phi$ has a fixed point, which implies the existence of solutions to~\eqref{EQN::FULLY-DISCRETE-SCHEME-INIT}. In order to do so, it only remains to prove that 
the following set is bounded:
\[
\{\bwh\in \Sp(\Th)^N:\ 
\bwh=\delta\Phi(\bwh), \ \delta\in[0,1]\}.
\]
Let~$\bwh\ne 0$ be in this set, and let~$\vWh$ be its coefficient vector.
Then, $\bwh=\delta\Phi(\bwh)$ for some~$\delta\in (0,1]$, namely~$\vWh$ satisfies
\begin{align*}
&\displaystyle{\frac{\varepsilon \tau_1}{\delta}}(I_N \otimes C)\vWh 
+\left(\calU_h(\vWh)-\mathbf{R}_h^0\right)\\ 
 &\qquad\qquad+\tau_1\left(I_N \otimes B^{{\sf T}}\bM^{-1} \right) \calA_h\left(\vWh; \vSigmah(\vWh) \right) + \tau_1(I_N \otimes S) \vWh
=\tau_1\bF_h(\vWh).
\end{align*}
We multiply the previous equation by $\vWh$. 
It follows as in the proof of Theorem~\ref{THM::ENTROPY-INEQUALITY} that
\begin{align*}
\displaystyle{\frac{\varepsilon \tau_1}{\delta}}\DotProd{(I_N \otimes C)\vWh}{\vWh}
    +
\DotProd{\calU_h(\vWh) - \mathbf{R}_h^0}{\vWh}
&\ge
\displaystyle{\frac{\varepsilon \tau_1}{\delta}}\Norm{\bwh}_{C}^2
        +
        \int_{\Omega} \left(s(\bu(\bwh)) - s(\brho_0) \right) \dx,\\
\tau_1\DotProd{\left(I_N \otimes B^{{\sf T}}\bM^{-1} \right) \calA_h\left(\vWh; \vSigmah(\vWh) \right)}{\vWh}
&\ge \tau_1 \gamma \Norm{\bdsigmah}_{[L^2(\Omega)^d]^N}^2,\\
\tau_1\DotProd{(I_N \otimes S) \vWh}{\vWh}
&=\tau_1 \bigNorm{\eta_F^{\frac12} \jump{\bwh}_{\sf N}}_{[L^2(\Fho)^d]^N}^2,\\
\tau_1\DotProd{\bF_h(\vWh)}{\vWh}
&\le \tau_1 C_f |\Omega|,
\end{align*}
from which we obtain 
\begin{equation*}
\displaystyle{\frac{\varepsilon\tau_1}{\delta}}\Norm{\bwh}_{C}^2 +
        \int_{\Omega} s(\bu(\bwh))\dx
+   \tau_1 \gamma \Norm{\bdsigmah}_{[L^2(\Omega)^d]^N}^2
+\tau_1 \bigNorm{\eta_F^{\frac12} \jump{\bwh}_{\sf N}}_{[L^2(\Fho)^d]^N}^2  \le
\displaystyle{\int_{\Omega}  s(\brho_0)  \dx}
+\tau_1 C_f |\Omega|.
\end{equation*}
Due to Assumption~\ref{H2c}, $\Norm{\bwh}_{C}$ is uniformly bounded with respect to~$\delta$. Therefore, 
the Schaefer fixed-point theorem
implies the existence of a fixed point of $\Phi$ ($\delta=1$) and therefore the existence of a solution~$\vWh^{\varepsilon,1}$ to problem~\eqref{EQN::FULLY-DISCRETE-SCHEME-INIT}.
In particular, for the function $\bwh^{\varepsilon,1}$ corresponding to the coefficient vector $\vWh^{\varepsilon,1}$, we have
\[
\int_{\Omega} s(\bu(\bwh^{\varepsilon,1}))\dx\le 
\displaystyle{\int_{\Omega}  s(\brho_0)  \dx}
+\tau_1 C_f |\Omega|.
\]

For~$n\ge 1$, we proceed by induction. Assuming the existence of $\vWh^{\varepsilon,n}$ and the boundedness of 
$\int_{\Omega} s(\bu(\bwh^{\varepsilon,n}))\dx$,
we apply the same arguments as above to the linearized problem 
\begin{align*} 
\varepsilon \tau_{n+1}(I_N \otimes C)\vWh^{\varepsilon}
=&
-\calU_h(\vVh)+\calU_h(\vWh^{\varepsilon,n})
-\tau_{n+1}\left(I_N \otimes B^{{\sf T}}\bM^{-1} \right) \calA_h\left(\vVh; \vSigmah(\vVh) \right)\\
&- \tau_{n+1}(I_N \otimes S) \vVh
    +\tau_{n+1}\bF_h(\vVh),
\end{align*}
to deduce that
\begin{align*}
\displaystyle{\frac{\varepsilon\tau_{n+1}}{\delta}}\Norm{\bwh}_{C}^2 +
        \int_{\Omega} s(\bu(\bwh))\dx
+   \tau_{n+1} \gamma \Norm{\bdsigmah}_{[L^2(\Omega)^d]^N}^2
+&\tau_{n+1} \bigNorm{\eta_F^{\frac12} \jump{\bwh}_{\sf N}}_{[L^2(\Fho)^d]^N}^2 \\
& \le
\displaystyle{\int_{\Omega}  s(\bu(\bwh^{\varepsilon,n}))  \dx}
+\tau_{n+1} C_f |\Omega|.
\end{align*}
The boundendess of~$\int_{\Omega} s(\bu(\bwh^{\varepsilon,n}))\dx$ entails the uniform boundedness of~$\Norm{\bwh}_{C}$, and the existence of a solution~$\vWh^{\varepsilon,n+1}$ to problem~\eqref{EQN::FULLY-DISCRETE-SCHEME}
is derived as above. Moreover,
\[
\int_{\Omega} s(\bu(\bwh^{\varepsilon,n+1}))\dx
\le \int_{\Omega} s(\bu(\bwh^{\varepsilon,n}))\dx+\tau_{n+1}C_f |\Omega|,
\]
which completes the proof.
\end{proof}

\begin{remark}[Regularizing term]
\label{rem:regularizing}
The regularizing term with multiplicative parameter~$\varepsilon>0$ in the fully discrete scheme~\eqref{EQN::FULLY-DISCRETE-SCHEME} is a discrete version of the one introduced for the semidiscrete-in-time formulation in~\cite[\S 3]{Jungel_2015}.
Such a term is used to enforce a numerical control on the~$L^{\infty}(\Omega)$ norm of the entropy variable~$\bwh$. This prevents~$\bu(\bwh)$ from approaching values near~$\partial \calD$, where~$s''$ typically becomes singular. In those cases, the~$\varepsilon$-regularization is needed to recover the stability of the nonlinear solver, as illustrated in Section~\ref{SECT::NUMERICAL-EXP} below.

We cannot perform the simultaneous limit $(\varepsilon,h)\to(0,0)$ in 
our formulation. The limit $\varepsilon\to 0$ and then $h\to 0$ was possible in the finite-volume scheme of \cite{Jungel_Zurek_2021}. We refer to Section~\ref{SEC::ASS_CONV} for some mathematical details, and to Section~\ref{SUBSECT::POROUS-MEDIUM} for numerical experiments for various values of~$\varepsilon$. 
\eremk
\end{remark}

%
\section{Convergence of the fully discrete scheme} \label{SECT::CONVERGENCE}
We fix $\varepsilon>0$ and a partition ${\mathcal I}_{\tau}$ of the time interval~$(0,T)$ defined as in Section~\ref{SECT::FULLY-DISCRETE}, where the index $\tau$ denotes the maximum element length. 
Consider a sequence of spatial meshes indexed by $m\in \IN$, $\left\{{\mathcal T}_{h_m}\right\}_{m}$, where $h_m$ is the maximum element diameter of ${\mathcal T}_{h_m}$.
We assume that $\{h_m\}_m$ is a decreasing sequence with~$h_m \leq 1$ for all~$m \in \IN$ and $\lim_{m\to\infty} h_m=0$.
We introduce the Local Discontinuous Galerkin gradient operator~$\nabla_{\!\sf DG}:\Sp(\mathcal{T}_{h_m})\to \Mp(\mathcal{T}_{h_m})$, which is defined by 
\begin{equation}\label{EQ::DGGRAD}
\int_{\Omega}\nabla_{\!\sf DG}\lambda_m\cdot\bunderline{\theta}_m\dx
=
\int_\Omega\left(\nabla_h\lambda_m-\mathcal{L}(\lambda_m)\right)\cdot\bunderline{\theta}_m\dx
\quad\forall \bunderline{\theta}_m\in \Mp(\mathcal{T}_{h_m}),
\end{equation}
with the jump lifting operator $\mathcal{L}:
\Sp(\mathcal{T}_{h_m})\to \Mp(\mathcal{T}_{h_m})$ given by
\[
\int_\Omega \mathcal{L}(\lambda_m)\cdot \bunderline{\theta}_m\dx
=\int_{\mathcal{F}_{h_m}^{\mathcal{I}}}\jump{\lambda_m}_{\sf N}\cdot \mvl{\bunderline{\theta}_m}_{1-\alpha_F}\dS 
\quad\forall \bunderline{\theta}_m\in \Mp(\mathcal{T}_{h_m}).
\]

\subsection{Assumptions for~\texorpdfstring{$h$}{}-convergence
}\label{SEC::ASS_CONV}

In the following Section~\ref{SEC::hCONVERGENCE}, we prove the convergence of fully discrete solutions to semidiscrete-in-time functions, as $m\to\infty$. To this aim, we
make the following abstract assumption, whose validity is discussed in Section~\ref{SEC::VALIDITY} below.

\begin{assumption} 
\label{ASSUMPTION::DG-NORM}
We set~$\ell = 1$ if~$d = 1$ and, if~$d = 2, 3$,
$$\ell = \begin{cases}
1 & \text{if } s'' A \in \EFC{0}{\overline{\calD}}{\IR^{N\times N}}, \\
2 & \text{otherwise}.
\end{cases}
$$
We assume that, for~$m \in \IN$, $\mathcal{T}_{h_m}$ and $p$ are such that
\begin{equation}\label{EQ::INCLUSION}
\mathcal{S}_\ell^{\sf cont}(\mathcal{T}_{h_m}):=
\mathcal{S}_\ell(\mathcal{T}_{h_m})\cap H^1(\Omega)\subset
\Sp(\mathcal{T}_{h_m}),
\end{equation}
and that there exists a DG norm~$\Norm{\cdot}_{\sf DG}$ in~$\Sp(\mathcal{T}_{h_m})^N$, which satisfies the following conditions:
\begin{enumerate}[label=\roman*), ref = \textit{\roman*})]
\item \label{A1} There exists a positive constant~$C_{\sf DG}$ independent of~$h_m$ such that \begin{equation*}
\sum_{i = 1}^N c_{h_m}(w_{m, i}, w_{m, i}) \geq C_{\sf DG} \Norm{\bw_m}_{\sf DG}^2 \qquad \forall \bw_m \in \Sp(\mathcal{T}_{h_m})^N.
\end{equation*}
\item \label{A2} If~$d = 1$ or~$\ell = 2$,
the following discrete Sobolev embedding is valid: there exists a positive constant~$C_S$ independent of~$h_m$ such that
\begin{equation}
\label{EQN::DISCRETE-SOBOLEV-EMBEDDING}
\Norm{\bw_{m}}_{L^\infty(\Omega)^N} \leq C_S \Norm{\bw_m}_{\sf DG} \qquad \forall \bw_m \in \Sp(\mathcal{T}_{h_m})^N.
\end{equation}
\item \label{A3} For any  sequence~$\{\bw_m\}_m$ with~$\bw_m \in \Sp(\mathcal{T}_{h_m})^N$ that is uniformly bounded in the DG norm, there exist a subsequence still denoted by~$\{\bw_m\}_m$ and a function~$\bw \in H^{\ell}(\Omega)^N$ such that, as~$m \rightarrow \infty$,
\begin{equation*}
\begin{split}
\nabla_{\!\sf DG}\bw_m\rightharpoonup
\nabla \bw & \qquad\text{weakly in $[L^2(\Omega)^d]^N$},  \\
\bw_m\to
\bw  & \qquad\text{strongly in }L^q(\Omega)^N,
\end{split}
\end{equation*}
with~$1\le q<6$, if $d=3$, or $1\le q<\infty$, if $d=1,2$. Moreover, for any~$\blambda \in H^\ell(\Omega)^N$ there exists a sequence ~$\{\blambda_m\}_m$ with $\blambda_m\in\mathcal{S}_\ell^{\sf cont}(\mathcal{T}_{h_m})^N$ such that, as~$m \rightarrow \infty$, it converges strongly in $H^1(\Omega)^N$ to $\blambda$ and
\begin{equation}
\label{EQN::CONVERGENCE-Ch}
\bc_{h_m} (\bw_m, \blambda_m) \rightarrow \int_{\Omega} \bigg(\sum_{|\boldsymbol{\alpha}| = \ell} D^{\boldsymbol{\alpha}} \bw \cdot D^{\boldsymbol{\alpha}} \blambda + \bw \cdot \blambda\bigg) \dx.
\end{equation}
\end{enumerate}
\end{assumption}

\subsection{\texorpdfstring{$h$}{}-convergence}
\label{SEC::hCONVERGENCE}

For $n=0,\ldots,N_t$, we denote by $\bw_m^{\varepsilon,n}$ a
solution to the fully discrete scheme~\eqref{EQN::FULLY-DISCRETE-SCHEME-INIT}--\eqref{EQN::FULLY-DISCRETE-SCHEME} on the spatial mesh~${\mathcal T}_{h_m}$ at the discrete time~$t_n$ of the fixed temporal mesh~${\mathcal I}_{\tau}$.

\begin{theorem}[$h$-convergence]\label{THM::H-CONVERGENCE}
Fix $\varepsilon>0$ and a temporal mesh ${\mathcal I}_\tau$.
Let Assumption~\ref{ASSUMPTION::DG-NORM} be satisfied. Then:
\begin{enumerate}[label=\Roman*), ref = \textit{\Roman*)}]
\item \label{Part1}
Setting $\brho_m^{\varepsilon,0}:=\Pi_p^0 \brho_0$, for $m\to\infty$, we have
\begin{equation}
\label{EQN::LIMIT-Rho0}
\int_\Omega (\brho_0-\brho_m^{\varepsilon,0})\cdot \bw \dx \to 0 \qquad\forall \bw\in H^1(\Omega)^N.
\end{equation}
Moreover, for any $n=1,\ldots,N_t$, there exists $\bw^{\varepsilon,n}\in H^{\ell}(\Omega)^N$ with $\bu\left(\bw^{\varepsilon,n}\right)\in H^1(\Omega)^N$ 
and a subsequence of $\left\{{\mathcal T}_{h_m}\right\}_{m}$ still denoted by $\left\{{\mathcal T}_{h_m}\right\}_{m}$
such that, as $m\to\infty$,
\[
\brho_m^{\varepsilon,n}:=\bu\left(\bw_m^{\varepsilon,n}\right)\to
\brho^{\varepsilon,n}:=\bu\left(\bw^{\varepsilon,n}\right)\qquad
\text{strongly in $L^r(\Omega)^N$ for all $r\in[1,\infty)$}.
\]
\item \label{Part2} Set, for convenience,~$\brho^{\varepsilon, 0} := \brho_0$.
For $n=0,\ldots,N_t-1$, $\bw^{\varepsilon,n+1}$ solves
\begin{equation}
\label{EQN::SEMI-DISCRET-IN-TIME}
\begin{split}
\varepsilon & \tau_{n+1}
\int_{\Omega} \bigg(\sum_{|\boldsymbol{\alpha}| = \ell} D^{\boldsymbol{\alpha}} \bw^{\varepsilon, n + 1} \cdot D^{\boldsymbol{\alpha}} \blambda + \bw^{\varepsilon, n + 1} \cdot \blambda\bigg) \dx
+\int_\Omega\left(
\bu\left(\bw^{\varepsilon,n+1}\right)-\brho^{\varepsilon,n}\right)
\cdot\blambda\dx\\
& +\tau_{n+1}
\int_\Omega A\left(\bu(\bw^{\varepsilon,n+1})\right)
\nabla\bu(\bw^{\varepsilon,n+1}) 
\frob \nabla\blambda\dx
=\tau_{n+1}\int_\Omega \vf\left(\bu(\bw^{\varepsilon,n+1})\right)\cdot\blambda\dx \quad \forall \blambda\in H^{\ell}(\Omega)^N.
\end{split}
\end{equation}
\item \label{Part3}
For $n=0,\ldots,N_t-1$, $\bw^{\varepsilon,n+1}$ satisfies
\begin{equation*}
\label{EQN::DISCRETE-ENTROPY1_limith}
\begin{split}
\varepsilon  \tau_{n+1} \Norm{\bw^{\varepsilon, n + 1}}_{H^{\ell}(\Omega)^N}^2  
 + \int_{\Omega} s(\bu(\bw^{\varepsilon, n+1})) \dx + \gamma  \tau_{n+1} \Norm{\nabla \bu(\bw^{\varepsilon, n+1})}_{[L^2(\Omega)^d]^N}^2\\
 \leq  \int_{\Omega} s(\brho^{\varepsilon, n}) \dx + C_f \tau_{n+1} \abs{\Omega},
\end{split}
\end{equation*}
and
\begin{equation*}
\label{EQN::DISCRETE-ENTROPY2_limith}
\begin{split}
\varepsilon \sum_{n = 0}^{N_t - 1} \tau_{n+1} \Norm{\bw^{\varepsilon, n + 1}}_{H^{\ell}(\Omega)^N}^2  
 + \int_{\Omega} s(\bu(\bw^{\varepsilon, N_t})) \dx + \gamma \sum_{n = 0}^{N_t - 1} \tau_{n+1} \Norm{\nabla \bu(\bw^{\varepsilon, n+1})}_{[L^2(\Omega)^d]^N}^2\\ 
 \leq  \int_{\Omega} s(\brho_0) \dx + C_f T \abs{\Omega}.
\end{split}
\end{equation*}
\end{enumerate}
\end{theorem}

\begin{proof}
\emph{Part \ref{Part1}} 
The limit in~\eqref{EQN::LIMIT-Rho0} follows from the estimate
\[
\abs{\int_{\Omega}(\brho_0-\brho_m^{\varepsilon,0})\cdot \bw \dx}
= \abs{\int_{\Omega}\brho_0\cdot (\bw -\Pi_p^0\bw)\dx}
\le C h_m\Norm{\brho_0}_{L^2(\Omega)^N}|\bw|_{H^1(\Omega)^N},
\]
where $C>0$ is independent of $h_m$.
Since the right-hand side of~\eqref{EQN::DISCRETE-ENTROPY2} is uniformly bounded, estimate~\eqref{EQN::DISCRETE-ENTROPY1}, together with Assumption~\ref{ASSUMPTION::DG-NORM}, \ref{A1}, implies that $\left\{\bw_m^{\varepsilon,n}\right\}_m$ is bounded in the DG norm, uniformly with respect to~$h_m$.
Then, by Assumption~\ref{ASSUMPTION::DG-NORM}, \ref{A3},
there exist
a function $\bw^{\varepsilon,n}\in H^{\ell}(\Omega)^N$ and a subsequence of $\left\{\bw_m^{\varepsilon,n}\right\}_m$, still denoted by $\left\{\bw_m^{\varepsilon,n}\right\}_m$ such that, as $m\to \infty$, 
\[
\bw_m^{\varepsilon,n}
\to
\bw^{\varepsilon,n}
\qquad
\text{strongly in $L^q(\Omega)^N$},
\]
with $1\le q<6$, if $d=3$, or $1\le q<\infty$, if $d=1,2$. Up to extraction of another subsequence, we can also assume that $\bw_m^{\varepsilon,n}$ converges to $\bw^{\varepsilon,n}$ almost everywhere in~$\Omega$.  
As~$\bu\left(\bw_m^{\varepsilon, n + 1}\right) \in L^{\infty}(\Omega)^N$, the dominated convergence theorem implies that~$\brho_m^{\varepsilon,n}:=\bu\left(\bw_m^{\varepsilon,n}\right)$
converges strongly to~$\\brho^{\varepsilon,n} := \bu\left(\bw^{\varepsilon,n}\right)$ 
in~$L^r(\Omega)^N$ for all $r\in[1,\infty)$. This proves the first part of the theorem. 
\medskip

\emph{Part~\ref{Part2}}
Now we prove that the limit $\bw^{\varepsilon,n}$ solves  problem~\eqref{EQN::SEMI-DISCRET-IN-TIME} for~$n = 0, \ldots, N_t -1$. 
We write~\eqref{EQN::FULLY-DISCRETE-SCHEME-INIT}--\eqref{EQN::FULLY-DISCRETE-SCHEME} as a variational problem: 
\begin{subequations}
\begin{equation}
\begin{split}
&\int_{\Omega} A(\bu(\bw_m^{\varepsilon, n + 1}))^{{\sf T}} s''(\bu(\bw_m^{\varepsilon, n + 1})) \bdsigma_m^{\varepsilon, n + 1} 
\frob \bdvarphi_m \dx \\
\label{EQN::VARIATIONAL-FULLY-DISCRETE-1}
&\qquad\qquad = -\int_{\Omega} A(\bu(\bw_m^{\varepsilon, n + 1}))^{{\sf T}} \nabla_{\! \sf{DG}}\bw_m^{\varepsilon, n + 1} 
\frob \bdvarphi_m \dx \qquad  \forall \bdvarphi_m \in \Mp(\mathcal{T}_{h_m})^N, 
\end{split}
\end{equation}
\begin{equation}
\begin{split}
\varepsilon \tau_{n+1} & \bc_{h_m}(\bw_m^{\varepsilon, n + 1}, \blambda_m)
+ \int_{\Omega}\bu(\bw_m^{\varepsilon, n + 1}) \cdot \blambda_m   \dx \\
& - \tau_{n+1} \int_{\Omega}A(\bu(\bw_m^{\varepsilon, n + 1})) \bdsigma_m^{\varepsilon, n + 1} 
\frob \nabla_{\!\sf DG}  \blambda_m  \dx  + 
\tau_{n+1}\int_{\mathcal{F}_{h_m}^{\mathcal{I}}} \eta_F \jump{\bw_m^{\varepsilon, n + 1}}_{\sf N}\cdot \jump{\blambda_m}_{\sf N} \dS \\
\label{EQN::VARIATIONAL-FULLY-DISCRETE-2}
&
= \int_{\Omega} \brho_m^{\varepsilon, n} \cdot \blambda_m \dx + \tau_n \int_{\Omega} \vf(\bu(\bw_m^{\varepsilon, n + 1})) \cdot \blambda_m \dx 
\qquad \forall \blambda_m \in \Sp(\mathcal{T}_{h_m})^N.
\end{split}
\end{equation}
\end{subequations}
Here, $\bc_{h_m}(\cdot,\cdot)$ is the bilinear form in~\eqref{EQN::C-BILINEAR-FORM} and~$\nabla_{\sf DG} (\cdot)$ is the LDG gradient defined in~\eqref{EQ::DGGRAD}.
 
We infer from the discrete entropy inequalities in Theorem~\ref{THM::ENTROPY-INEQUALITY} that $\{\bdsigma_m^{\varepsilon,n+1}\}_m$ is bounded in the~$L^2(\Omega)$ norm. This implies that there exists $\bdsigma^{\varepsilon,n+1}\in [L^2(\Omega)^d]^N$ such that, up to extracting a subsequence, 
\[
\bdsigma_m^{\varepsilon,n+1}\rightharpoonup\bdsigma^{\varepsilon,n+1} \qquad\text{weakly in $[L^2(\Omega)^d]^N$}.
\]
Moreover, as $\{\bw_m^{\varepsilon,n+1}\}_m$ is bounded in the DG norm, by~Assumption~\ref{ASSUMPTION::DG-NORM}, \ref{A3}, there exists~$\bw^{\varepsilon, n + 1} \in H^1(\Omega)^N$ such that, up to a subsequence, 
\begin{equation*}
\begin{split}
\nabla_{\!\sf DG}\bw_m^{\varepsilon,n+1}\rightharpoonup
\nabla \bw^{\varepsilon,n+1} & \qquad\text{weakly in $[L^2(\Omega)^d]^N$},  \\
\bw_m^{\varepsilon,n+1}\to
\bw^{\varepsilon,n+1}  & \qquad\text{strongly in }L^q(\Omega)^N,
\end{split}
\end{equation*}
with~$1\le q<6$ if $d=3$, and $1\le q<\infty$ if $d=1,2$.
From Part~\ref{Part1}, we have that~$\bu(\bw_m^{\varepsilon,n + 1})\to \bu(\bw^{\varepsilon,n + 1})$ strongly in~$L^r(\Omega)^N$ for any $r\in[1,\infty)$ and therefore almost everywhere in~$\Omega$. 
Due to the continuity of~$A$, we also have that~$A(\bu(\bw_m^{\varepsilon,n + 1}))\to A(\bu(\bw^{\varepsilon,n + 1}))$ almost everywhere. 
Furthermore, as~$A$ is continuous in $\overline{\mathcal D}$ (see~\ref{H1}) and~$\bu : \IR^N \rightarrow \calD$, the sequence~$\{A(\bu(\bw_m^{\varepsilon,n + 1}))\}_m$ is uniformly bounded. Therefore, 
\[
A(\bu(\bw_m^{\varepsilon,n + 1}))\to A(\bu(\bw^{\varepsilon,n + 1}))\qquad\text{strongly in~$L^r(\Omega)^{N \times N}$ 
for all $r\in[1,\infty)$.}
\]
Similarly, we deduce that
\begin{equation}
\label{EQN::convergence-f}
\vf(\bu(\bw_m^{\varepsilon,n + 1}))\to \vf(\bu(\bw^{\varepsilon,n + 1}))\qquad\text{strongly in~$L^r(\Omega)^N$ 
for all $r\in[1,\infty)$.}
\end{equation}
The boundedness of $\{A(\bu(\bw_m^{\varepsilon, n + 1})) ^{{\sf T}} \nabla_{\sf DG} \bw_m^{\varepsilon, n + 1}\}_m$ in $[L^2(\Omega)^d]^N$implies that there exists $\bunderline{\Phi}\in [L^2(\Omega)^d]^N$ such that,
up to extracting a subsequence,
\[
A(\bu(\bw_m^{\varepsilon, n + 1})) ^{{\sf T}} \nabla_{\sf DG} \bw_m^{\varepsilon, n + 1}
\rightharpoonup
\bunderline{\Phi} \qquad\text{weakly in $[L^2(\Omega)^d]^N$}.
\]
As $A(\bu(\bw_m^{\varepsilon, n + 1})) ^{{\sf T}} \nabla_{\sf DG} \bw_m^{\varepsilon, n + 1}$ is the product of a term that converges strongly in $[L^r(\Omega)^d]^{N\times N}$ for any $r\in[1,\infty)$ and a term that converges weakly in $[L^2(\Omega)^d]^N$, it converges weakly in $[L^s(\Omega)^d]^N$ for any $s<2$ ($\frac1r+\frac12=\frac1s$) to the product of the two limits. Therefore, for the uniqueness of the weak limit, $\bunderline{\Phi}$ must be equal to the product of the two limits. This proves that
\begin{equation}\label{EQ::WEAK1}
A(\bu(\bw_m^{\varepsilon, n + 1}))^{{\sf T}} \nabla_{\! \sf{DG}}\bw_m^{\varepsilon, n + 1}
\rightharpoonup
A(\bu(\bw^{\varepsilon, n + 1}))^{{\sf T}} \nabla\bw^{\varepsilon, n + 1} \qquad\text{weakly in $[L^2(\Omega)^d]^N$}.
\end{equation}
Similarly, we have the convergence
\begin{equation}\label{EQ::WEAK2}
A(\bu(\bw_m^{\varepsilon, n + 1})) \bdsigma_m^{\varepsilon, n + 1}   \rightharpoonup A(\bu(\bw^{\varepsilon, n + 1})) \bdsigma^{\varepsilon, n + 1}\qquad\text{weakly in $[L^2(\Omega)^d]^N$}.
\end{equation}
Moreover, if~$d = 1$ or~$\ell = 2$, Assumption~\ref{ASSUMPTION::DG-NORM}, \ref{A2}, implies that~$\bw_m^{\varepsilon, n + 1}(\bx) \in \mathcal{K}$ a.e. in~$\Omega$, for some compact~$\mathcal{K} \subset \IR^N$ and all~$m \geq 0$. Therefore, $\bu(\bw_m^{\varepsilon, n + 1})(\bx) \in \widetilde{\mathcal{K}}$ a.e. in~$\IR^N$, for some compact~$\widetilde{\mathcal{K}} \subset \calD$ and all~$m \geq 0$. Since~$A^{{\sf T}} s''$ is continuous in~$\widetilde{\mathcal{K}}$, proceeding again as for~\eqref{EQ::WEAK1}, it follows that
\begin{equation}\label{EQ::WEAK}
A(\bu(\bw_m^{\varepsilon, n + 1}))^{{\sf T}} s''(\bu(\bw_m^{\varepsilon, n + 1})) \bdsigma_m^{\varepsilon, n + 1}
\rightharpoonup
A(\bu(\bw^{\varepsilon, n + 1}))^{{\sf T}} s''(\bu(\bw^{\varepsilon, n + 1})) \bdsigma^{\varepsilon, n + 1} \quad \ \text{ weakly in }[L^2(\Omega)^d]^N.
\end{equation} 
When~$s''A \in \EFC{0}{\overline{\calD}}{\IR^{N \times N}}$, the weak convergence in~\eqref{EQ::WEAK} follows from the boundedness of~$\calD$ without requiring Assumption~\ref{ASSUMPTION::DG-NORM}, \ref{A2}, to be satisfied.

In order to pass to the limit in both sides of equation~\eqref{EQN::VARIATIONAL-FULLY-DISCRETE-1}, we observe that, for every $\bdvarphi\in [L^2(\Omega)^d]^N$, there exists a sequence~$\{\bdvarphi_m\}_m\subset \Mp(\mathcal{T}_{h_m})^N$ that converges to $\bdvarphi$ strongly in $[L^2(\Omega)^d]^N$.
We test~\eqref{EQN::VARIATIONAL-FULLY-DISCRETE-1} with $\bdvarphi_m$. Then, the weak convergence~\eqref{EQ::WEAK} and the strong convergence of $\bdvarphi_m$ imply that
\[
\int_\Omega
A(\bu(\bw^{\varepsilon, n + 1}))^{{\sf T}} s''(\bu(\bw^{\varepsilon, n + 1})) \bdsigma^{\varepsilon, n + 1}
\frob 
\bdvarphi\dx
=-\int_\Omega
A(\bu(\bw^{\varepsilon, n + 1}))^{{\sf T}} \nabla\bw^{\varepsilon, n + 1} 
\frob 
\bdvarphi\dx\qquad\forall \bdvarphi\in [L^2(\Omega)^d]^N.
\]
This, together with the chain rule~$\nabla\bw^{\varepsilon, n + 1}=\nabla s^\prime(\bu(\bw^{\varepsilon, n + 1}))=s''(\bu(\bw^{\varepsilon, n + 1}))\nabla\bu(\bw^{\varepsilon, n + 1})$ 
and Assumption~\ref{H2a} leads to
\begin{equation}\label{EQ::SIGMA}
\bdsigma^{\varepsilon, n + 1}=-\nabla\bu(\bw^{\varepsilon, n + 1}).
\end{equation}

Next, we consider equation~\eqref{EQN::VARIATIONAL-FULLY-DISCRETE-2}. For any $\blambda\in H^\ell(\Omega)^N$, let $\{\blambda_m\}_m\subset \mathcal{S}_\ell^{\sf cont}(\mathcal{T}_{h_m})^N$ be a sequence as in Assumption~\ref{ASSUMPTION::DG-NORM}, \ref{A3}.
Due to the assumption $\mathcal{S}_\ell^{\sf cont}({\mathcal T}_{h_m})\subset \Sp({\mathcal T}_{h_m})$, we can test~\eqref{EQN::VARIATIONAL-FULLY-DISCRETE-2} with $\blambda_m$. 
Taking into account that $\blambda_m$ has zero jumps across interelement boundaries, the last term on the left-hand side in~\eqref{EQN::VARIATIONAL-FULLY-DISCRETE-2}, which involves $\jump{\blambda_m}_{\sf N}$, is equal to zero. 
We deduce from Assumption~\ref{ASSUMPTION::DG-NORM} (in particular,~\eqref{EQN::CONVERGENCE-Ch}), part~\ref{Part1} of the present theorem, and the limits~\eqref{EQN::convergence-f}  and~\eqref{EQ::WEAK2}, 
that the weak convergence to the appropriate limits of the remaining terms that involve trial functions. Together with the strong convergence of the terms containing test functions, we find that
\begin{equation}
\begin{aligned}
\nonumber
&\tau_{n+1}
\int_{\Omega} \bigg(\sum_{|\boldsymbol{\alpha}| = \ell} D^{\boldsymbol{\alpha}} \bw^{\varepsilon, n + 1} \cdot D^{\boldsymbol{\alpha}} \blambda + \bw^{\varepsilon, n + 1} \cdot \blambda\bigg) \dx
+  \int_\Omega \left( \bu(\bw^{\varepsilon, n + 1}) \cdot \blambda -\tau_{n+1} A(\bu(\bw^{\varepsilon, n + 1})) \bdsigma^{\varepsilon, n + 1} 
\frob \nabla  \blambda\right)  \dx \\
\nonumber
&\qquad\qquad
=  \int_\Omega \brho^{\varepsilon, n} \cdot \blambda \dx + \tau_{n+1}  \int_\Omega\vf(\bu(\bw^{\varepsilon, n + 1})) \cdot \blambda \dx 
\qquad \forall \blambda\in H^\ell(\Omega)^N.
\end{aligned}
\end{equation}
The combination of this with identity~\eqref{EQ::SIGMA} implies that, for~$n=0,\ldots,N_t-1$, $\bw^{\varepsilon,n + 1}$ solves~\eqref{EQN::SEMI-DISCRET-IN-TIME}. 
This completes the proof of second part of the theorem.
\medskip

\emph{Part \ref{Part3}} This part follows from~\eqref{EQN::SEMI-DISCRET-IN-TIME} in Part~\ref{Part2}, by proceeding as in Theorem~\ref{THM::ENTROPY-INEQUALITY}.
\end{proof}

\begin{remark}[Unbounded domains~$\calD$]
\label{REM::UNBOUNDED-D}
In the proof of the existence of discrete solutions and of the convergence to a solution to the $\varepsilon$-perturbed continuous problem~\eqref{EQN::SEMI-DISCRET-IN-TIME}, we repeatedly use the boundedness of~$\calD$ and the continuity of~$A$ and~$\vf$ on~$\overline{\calD}$ (see Assumption~\ref{H2a}).
Such a restriction can be lifted by using the argument
employed to prove limit~\eqref{EQ::WEAK}. 
More precisely, the presence of the regularizing term in the fully discrete scheme~\eqref{EQN::FULLY-DISCRETE-SCHEME} and Assumption~\ref{A2} guarantee that~$\bw_m^{\varepsilon, n}(\bx) \in \mathcal{K}$ a.e. in~$\Omega$, for some compact~$\mathcal{K} \subset \IR^N$, which implies that~$\bu(\bw_m^{\varepsilon, n})(\bx) \in \widetilde{\mathcal{K}}$ a.e. in~$\IR^N$, for some compact~$\widetilde{\mathcal{K}} \subset \calD$ and all~$h_m>0$.
Therefore, at each occurrence, the assumption of the boundedness of~$\calD$ can be replaced by the boundedness of~$\widetilde{\mathcal{K}}$ and the fact that the compact~$\widetilde{\mathcal{K}}$ is independent of~$h$.
\eremk
\end{remark}

\subsection{Convergence to a weak solution to the continuous problem}
\label{SECT::CONVERGENCE-epsilon-tau}
Let~$\mathcal{I}_{\tau}$ be a temporal mesh and~$\{\brho^{\varepsilon, n}\}_{n = 0}^{N_t}$ be the corresponding semidiscrete-in-time solution from Theorem~\ref{THM::H-CONVERGENCE}.
For simplicity, we assume~$\mathcal{I}_{\tau}$ to be uniform.
We define~$\brho^{(\varepsilon, \tau)} \in L^2(0, T; H^1(\Omega)^N)$ as the piecewise linear reconstruction in time of~$\{\brho^{\varepsilon, n}\}_{n = 0}^{N_t}$ defined by
\begin{equation}
\label{EQN::PIECEWISE-LINEAR-RECONSTR}
\brho^{(\varepsilon, \tau)}(\cdot, t) := \brho^{\varepsilon, n + 1}(\cdot) - \big(
t_{n+1}- t\big)\big(\brho^{\varepsilon, n + 1}(\cdot) - \brho^{\varepsilon, n}(\cdot)\big)/\tau \quad \text{ for } n\tau \leq t \leq (n + 1) \tau, \ 0 \leq n \leq N_t - 1.
\end{equation}
We also define the shift~${\tt{s}}_{\tau} \brho^{(\varepsilon, \tau)}(\cdot, t) = \brho^{\varepsilon, n}(\cdot)$ for~$n\tau \leq t \leq (n + 1)\tau,\ 0 \leq n \leq N_t - 1$.

We say that~$\brho$ is a \emph{weak solution to the continuous problem}~\eqref{EQN::MODEL-PROBLEM} if it satisfies
\begin{itemize}
\item $\brho \in L^2(0, T; H^1(\Omega)^N) \cap H^1(0, T; [H^1(\Omega)^N]')$; 
\item $\brho(\bx, t) \in \overline{\calD}$ a.e. in~$\Omega \times (0, T]$
(in particular, $\brho\in L^\infty(0,T; L^\infty(\Omega)^{N})$);
\item 
$\brho(\cdot, 0) = \brho_0(\cdot)$ in the sense of~$[H^1(\Omega)^N]'$;
\item $\displaystyle
\int_0^T \langle\dpt \brho, \blambda \rangle \dt + \int_0^T \int_{\Omega} A(\brho) \nabla \brho : \nabla \blambda\, \dx \dt = \int_0^T \int_{\Omega} \vf(\brho) \cdot \blambda\, \dx \dt \qquad \forall \blambda \in L^2(0, T; H^1(\Omega)^N)$,\\[0.2cm]
where~$\langle\cdot, \cdot \rangle$ denotes the duality between~$[H^1(\Omega)^N]'$ and~$H^1(\Omega)^N$.
\end{itemize}
\begin{theorem}
\label{THM::FULL-CONVERGENCE}
Let Assumption~\ref{ASSUMPTION::DG-NORM} be satisfied, and let~$\brho^{(\varepsilon, \tau)}$ be the piecewise linear  reconstruction of the semidiscrete-in-time solution from  Theorem~\ref{ASSUMPTION::DG-NORM}. Then, there exists a continuous weak solution~$\brho$ to problem~\eqref{EQN::MODEL-PROBLEM} such that, up to a subsequence that is not relabeled, for~$(\varepsilon, \tau) \rightarrow (0, 0)$, we have
\begin{align*}
\brho^{(\varepsilon, \tau)} \rightarrow \brho & \quad \text{ strongly in }L^r(0, T; L^r(\Omega)^N) \text{ for any~$r \in[1, \infty) 
$ and a.e. in~$\Omega \times (0, T]$} ,\\
\nabla \brho^{(\varepsilon, \tau)} \rightharpoonup \nabla \brho & \quad \text{ weakly in } L^2(0, T; [L^2(\Omega)^d]^N),\\
\tau^{-1}(\brho^{(\varepsilon, \tau)} - {\tt{s}}_{\tau} \brho^{(\varepsilon, \tau)}) \rightharpoonup \partial_t \brho & \quad \text{ weakly in }L^2(0, T; [H^{\ell}(\Omega)^N]'),
\end{align*}
where the integer~$\ell$ is as in Assumption~\ref{ASSUMPTION::DG-NORM}.
\end{theorem}
\begin{proof}
The proof follows closely the arguments in steps 2 and 3 of~\cite[\S3]{Jungel_2015}. 
\end{proof}

\subsection{Validity of Assumption~\ref{ASSUMPTION::DG-NORM}}\label{SEC::VALIDITY}

The proof of Theorem~\ref{THM::H-CONVERGENCE} strongly relies on
the validity of Assumption~\ref{ASSUMPTION::DG-NORM}. Due to our mesh assumptions, inclusion~\eqref{EQ::INCLUSION} is satisfied whenever $p\ge\ell$.
Before discussing
the existence of a bilinear form~$c_h(\cdot, \cdot)$ and a DG norm~$\Norm{\cdot}_{\sf DG}$ with the properties~\ref{A1}--\ref{A3} in Assumption~\ref{ASSUMPTION::DG-NORM},
we prove
the following estimate, which is an extension of~\cite[Thm. 4.4]{Buffa_Ortner_2009} to the cases~$q = 1, \ p = 1, 2$. 
\begin{lemma}[Broken trace estimate]
\label{LEMMA::BROKEN-TRACE}
For~$r \in \{1, 2\}$, we have the following estimate:
\begin{equation*}
\Norm{v}_{L^1(\partial \Omega)} \lesssim \Norm{v}_{L^r(\Omega)} + 
\Norm{\nabla_h v}_{L^r(\Omega)^d}
+ \Norm{\mathsf{h}^{\frac{1 - r}{2}}\jump{v}_{\sf N}}_{L^r(\Fho)^{d}} \qquad \forall v \in \Sp(\Th),
\end{equation*}
where the hidden constant is independent of~$h$ and~$v$, but it depends on~$\Omega$. 
\end{lemma}

\begin{proof}
Let~$v \in \Sp(\Th)$ and~$Q_h:\Sp(\Th)\to {\mathcal S}_1^{\sf cont}(\Th)$ 
be the reconstruction operator defined in~\cite[\S3]{Buffa_Ortner_2009}. By the triangle inequality,
\begin{equation}\label{eq:va}
\Norm{v}_{L^1(\partial \Omega)} \le \Norm{Q_h v}_{L^1(\partial \Omega)} +\Norm{v-Q_h}_{L^1(\partial \Omega)}.
\end{equation}
The trace theorem in $W^{1,1}(\Omega)$ gives
\begin{equation}\label{eq:Qva}
\Norm{Q_h v}_{L^1(\partial \Omega)} \lesssim
\Norm{Q_h v}_{L^1(\Omega)} + \Norm{\nabla Q_h v}_{L^1(\Omega)^d}.
\end{equation}
Thus, it follows from~\eqref{eq:va} and~\eqref{eq:Qva}, 
by applying the triangle inequality and~\cite[Thm.~3.1]{Buffa_Ortner_2009}, that
\begin{equation}
\label{EQN::BROKEN-TRACE-P1}
\begin{split}
\Norm{v}_{L^1(\partial \Omega)} & \lesssim
\Norm{v}_{L^1(\Omega)}+\Norm{v-Q_h v}_{L^1(\Omega)}+ \Norm{\nabla Q_h v}_{L^1(\Omega)^d}
+\Norm{v-Q_h v}_{L^1(\partial \Omega)}\\
&\lesssim \Norm{v}_{L^1(\Omega)}+
\Norm{\nabla_h v}_{L^1(\Omega)^d}
+ \Norm{\jump{v}_{\sf N}}_{L^1(\Fho)^{d}},
\end{split}
\end{equation}
which completes the proof for~$r = 1$. 

We now consider the case $r=2$.
We infer from the Cauchy--Schwarz inequality that
\begin{equation}
\label{EQN::L1-L2-v}
\Norm{v}_{L^1(\Omega)} = \sum_{K \in \Th} \int_{K} |v| \dx \leq \bigg(\sum_{K \in \Th} |K| \bigg)^{\frac12} \Norm{v}_{L^2(\Omega)} = |\Omega|^{\frac12} \Norm{v}_{L^2(\Omega)},
\end{equation}
and similarly,
\begin{equation}
\label{EQN::L1-L2-Nabla-v}
\Norm{\nabla_h v}_{L^1(\Omega)^d}
\leq |\Omega|^{\frac12} 
\Norm{\nabla_h v}_{L^2(\Omega)^d}.
\end{equation}
Moreover, by the definition of~$\mathsf{h}$ in~\eqref{EQN::DEF-h} and, for~$d=2,3$, the shape-regularity assumption, we have\footnote{
For $d=1$, $\int_{\Fho}\mathsf{h}\,\dS=\sum_{i = 1}^{M-1}\mathsf{h}(x_i) \lesssim \sum_{K \in \Th} |K|$, where $\{x_i\}_{i=0}^M$ are the meshpoints. 
\\
For $d=2,3$, $|K|=(\text{sum of facet $(d-1)$-measures})\times \varrho_K/d$, with~$\varrho_K$ being the inradius of~$K$. From the shape-regularity assumption~\eqref{EQ::SHAPEREG}, we deduce that~$|K| \geq \Upsilon h_K h_F /d$ for any facet~$F$ of~$K$, and obtain
$\int_{\Fho}\mathsf{h}\,\dS\lesssim\sum_{F \subset \Fho} \mathsf{h}_{|F} |F|\lesssim \sum_{K \in \Th} |K|$.
}
\[
\int_{\Fho}\mathsf{h}\,\dS\lesssim 
\sum_{K \in \Th} |K|,
\]
from which we deduce that
\begin{align}
\Norm{\jump{v}_{\sf N}}_{L^1(\Fho)^{d}} & = \int_{\Fho} \mathsf{h}^{\frac12} \mathsf{h}^{-\frac12} |\jump{v}_{\sf N}| \dS \leq
\bigg(\int_{\Fho}\mathsf{h}\,\dS\bigg)^{\frac12}
\Norm{\mathsf{h}^{-\frac12} \jump{v}_{\sf N}}_{L^2(\Fho)^{d}} \lesssim \bigg(\sum_{K \in \Th} |K| \bigg)^{\frac12} \Norm{\mathsf{h}^{-\frac12} \jump{v}_{\sf N}}_{L^2(\Fho)^{d}} \nonumber \\
& \lesssim \Norm{\mathsf{h}^{-\frac12} \jump{v}_{\sf N}}_{L^2(\Fho)^{d}}.
\label{EQN::L1-L2-jumps}
\end{align}
Combining the broken trace estimate for~$r = 1$ in~\eqref{EQN::BROKEN-TRACE-P1} with bounds~\eqref{EQN::L1-L2-v}, \eqref{EQN::L1-L2-Nabla-v}, and~\eqref{EQN::L1-L2-jumps}, the desired result for~$r = 2$ follows, completing the proof.
\end{proof}

We now discuss the validity of Assumption~\ref{ASSUMPTION::DG-NORM}. For this, we distinguish three cases.

\paragraph{$\bullet$ Case~$d = 1$ ($\ell=1$).} 

Choose~$c_{h_m}(\cdot, \cdot)$ and the $H^1$-type norm~$\Norm{\cdot}_{\sf DG}$ as
\begin{equation}\label{EQ::DGPROD}
c_{h_m}(w_m,v_m) := \int_\Omega w_m v_m\dx
+\int_\Omega\nabla_{\sf DG} w_m\cdot\nabla_{\sf DG} v_m\dx
+\int_{\mathcal{F}_{h_m}^\mathcal{I}}\mathsf{h}^{-1}\jump{w_m}_{\sf N}\cdot\jump{v_m}_{\sf N}\dS \qquad \forall w_m, v_m \in \Sp(\mathcal{T}_{h_m}),
\end{equation}
\begin{equation}\label{EQ::DGNORM}
\Norm{\bw_m}_{\sf DG}^2:=
\Norm{\bw_m}_{L^2(\Omega)^N}^2
+
\Norm{ \nabla_h \bw_m}_{[L^2(\Omega)^d]^N}^2
+ \bigNorm{{\mathsf{h}}^{-\frac12} \jump{\bw_m}_{\sf N}}_{[L^2(\mathcal{F}_{h_m}^\mathcal{I})^{d}]{}^N}^2 \qquad \forall \bw_m \in \Sp(\mathcal{T}_{h_m})^N.
\end{equation}

With this choice, property~\ref{A1} follows from the coercivity of the LDG discretization of the Laplace operator (see, e.g., \cite[Prop.~3.1]{Perugia_Schotzau_2002}), property~\ref{A3} follows from~\cite[Thm.~5.2 and Lemma~8]{Buffa_Ortner_2009}, and property~\ref{A2} follows from the following proposition. 

\begin{proposition}[Discrete Sobolev embedding in~1D]
Let $(a,b)$ be an interval in $\IR$, and let $\Sp(\Th)$ be defined on the partition $\Th$ given by~$a =: x_0 < x_1 < \ldots < x_M := b$. If the DG norm is chosen as in~\eqref{EQ::DGNORM} with~$N = 1$, then, for all~$v\in \Sp(\Th)$,
\[
\Norm{v}_{L^\infty(a,b)} \lesssim \Norm{v}_{\sf DG}, 
\]
where the hidden constant is independent of~$h$ and~$v$.
\end{proposition}

\begin{proof}
Let~$K_i := (x_{i-1}, x_i)$ and~$h_i := x_i - x_{i-1}$, for~$i = 1, \ldots, M$. 
For any~$v \in \Sp(\Th)$ and~$j = 1, \ldots, M$, by the Fundamental Theorem of Calculus and the H\"older inequality, we have for all $x\in K_j$,
\begin{equation*}\label{eq:proof1D}
\begin{aligned}
v(x) &=  v(a) + \sum_{i = 1}^{j-1} \bigg(\int_{K_i} v'(x) \text{d}x + v(x_i^+) - v(x_i^-)\bigg) + \int_{x_{j-1}}^x v'(x) \text{d}x \\
&\leq  |v(a)| + \sum_{i = 1}^{M}  \Norm{v'}_{L^1(K_i)} + \sum_{i = 1}^{M-1} |v(x_i^+)-v(x_i^-)| \\
&\leq |v(a)| + \sum_{i = 1}^{M} h_i^{\frac12} \Norm{v'}_{L^2(K_i)} + \sum_{i = 1}^{M-1} \mathsf{h}^{\frac12}(x_i) \mathsf{h}^{-\frac12}(x_i)|v(x_i^+)-v(x_i^-)| \\
&\leq |v(a)| + \bigg(\sum_{i = 1}^M h_i\bigg)^\frac12 \Norm{v'}_{L^2(a, b)} + \bigg(\sum_{i = 1}^{M-1}\mathsf{h}(x_i)\bigg)^\frac12 
\bigNorm{{\mathsf{h}}^{-\frac12} \jump{v}_{\sf N}}_{L^2(\mathcal{F}_{h}^\mathcal{I})} 
 \\
&\lesssim |v(a)| + \Norm{v'}_{L^2(a,b)}+
\bigNorm{{\mathsf{h}}^{-\frac12} \jump{v}_{\sf N}}_{L^2(\mathcal{F}_{h}^\mathcal{I})} 
\lesssim |v(a)| + \Norm{v}_{\sf DG}.
\end{aligned}
\end{equation*}
Lemma~\ref{LEMMA::BROKEN-TRACE} with $r=2$ implies that~$|v(a)| \lesssim \Norm{v}_{\sf DG}$, and the proof is complete.
\end{proof}

\paragraph{$\bullet$ Case~$d = 2, 3$ and~$s'' A\in \EFC{0}{\overline{\calD}}{\IR^{N\times N}}$ ($\ell=1$).} 
In this case, the enforcement of the~$L^\infty(\Omega)$-boundedness on the discrete entropy variable~$\bw_m^{\varepsilon, n}$, which is a consequence of property~\ref{A2}, is no longer necessary, as the weak convergence in~\eqref{EQ::WEAK} follows from the boundedness of~$\calD$ and the continuity of~$s'' A$ on~$\overline{\calD}$.
Moreover, for the bilinear form~$c_{h_m}(\cdot, \cdot)$ and the norm~$\Norm{\cdot}_{\sf DG}$ defined in~\eqref{EQ::DGPROD} and~\eqref{EQ::DGNORM}, respectively, properties~\ref{A1} and~\ref{A3} follow from the same results as in the case~$d = 1$. We present some examples for cross-diffusion systems satisfying $s'' A\in \EFC{0}{\overline{\calD}}{\IR^{N\times N}}$ in Appendix \ref{app}.

\paragraph{$\bullet$ Case~$d = 2, 3$ and~$\ell = 2$.}
We define the discrete LDG Hessian operator~$\calH_{\sf DG}: \Sp(\mathcal{T}_{h_m}) \rightarrow L^2(\Omega)^{d \times d}$ as
\begin{equation*}
\int_{\Omega} \calH_{\sf DG} \lambda_m : \Theta_m \dx = 
\int_\Omega \left(D_h^2\lambda_m - \calR(\lambda_m) + \calB(\lambda_m)\right) : \Theta_m \dx \qquad \forall \Theta_m \in \prod_{K \in \mathcal{T}_{h_m}} \Pp{p}{K}^{d \times d},
\end{equation*}
where~$D_h^2$ denotes the elementwise Hessian operator, and the lifting operators~$\calR: \Sp(\mathcal{T}_{h_m}) \rightarrow L^2(\Omega)^{d\times d}$ and~$\calB : \Sp(\mathcal{T}_{h_m}) \rightarrow L^2(\Omega)^{d\times d}$ are defined by
\begin{align*}
\int_{\Omega} \calR(\lambda_m) : \Theta_m \dx & = \sum_{K \in \mathcal{T}_{h_m}} \int_{(\partial K)^\circ} \mvl{\Theta_m} \bunderline{n}_K \cdot \nabla \lambda_m \dS &  \quad \forall \Theta_m \in \prod_{K \in \mathcal{T}_{h_m}} \Pp{p}{K}^{d \times d}, \\
\int_{\Omega} \calB(\lambda_m) : \Theta_m \dx & = \int_{\mathcal{F}_{h_m}^\mathcal{I}} \mvl{\nabla_h \cdot \Theta_m} \cdot \jump{\lambda_m}_{\sf N} \dS & \quad \forall \Theta_m \in \prod_{K \in \mathcal{T}_{h_m}} \Pp{p}{K}^{d \times d}.
\end{align*}
For piecewise smooth functions~$\bunderline{w}$, $\bw$, and $\bunderline{\bw}$ with~$d$, $N$, and~$N \times d$ 
components, respectively, we define the (vector-valued) total jump on each facet~$F = \partial K_1 \cap \partial K_2 \in\Fho$, for some~$K_1,K_2\in\Th$, with a prescribed unit normal vector, 
say, pointing from~$K_1$ to~$K_2$, as
\[
\jump{\bunderline{w}}:=\bunderline{w}_{|_{K_1}}-\bunderline{w}_{|_{K_2}},\qquad
\jump{\bw}:=\bw_{|_{K_1}}-\bw_{|_{K_2}},\qquad
\jump{\bunderline{\bw}}:=\bunderline{\bw}_{|_{K_1}}-\bunderline{\bw}_{|_{K_2}}.
\]
Finally, we choose~$c_{h_m}(\cdot, \cdot)$ and the $H^2$-type norm~$\Norm{\cdot}_{\sf DG}$ as
\begin{align}
\nonumber
c_{h_m}(w_m, v_m) := & \int_{\Omega} w_m v_m \dx 
+ \int_\Omega \nabla_{\sf DG} w_m\cdot\nabla_{\sf DG} v_m\dx
+ \int_{\Omega} \calH_{\sf DG}w_m : \calH_{\sf DG} v_m \dx  \\
\label{EQN::DGHess}
& + \int_{\mathcal{F}_{h_m}^\mathcal{I}} \mathsf{h}^{-1} \jump{\nabla_h w_m} \cdot \jump{\nabla_h v_m}\dS + \int_{\mathcal{F}_{h_m}^\mathcal{I}} \mathsf{h}^{-3} \jump{w_m}_{\sf N} \cdot \jump{v_m}_{\sf N} \dS &  \forall w_m, v_m \in \Sp(\mathcal{T}_{h_m}), \\
\nonumber
\Norm{\bw_m}_{\sf DG}^2 := & \Norm{\bw_m}_{L^2(\Omega)^N}^2 
+
\Norm{ \nabla_h \bw_m}_{[L^2(\Omega)^d]^N}^2
+ 
\Norm{D_h^2 \bw_m}_{[L^2(\Omega)^{d\times d}]^N}^2
\\
\label{EQN::DG2NORM}
& + \Norm{\mathsf{h}^{-\frac12} \jump{\nabla_h \bw_m}}_{[L^2(\mathcal{F}_{h_m}^\mathcal{I})^d]^N}^2 + \Norm{\mathsf{h}^{-\frac32} \jump{\bw_m}_{\sf N}}_{[L^2(\mathcal{F}_{h_m}^\mathcal{I})^d]^N}^2
& \forall \bw_m \in \Sp(\mathcal{T}_{h_m})^N.
\end{align}

Then, property~\ref{A1} follows from~\cite[Lemma~2.6]{Bonito_Guignard_Nochetto_Yang_2023}. 
The discrete compactness argument in Assumption~\ref{ASSUMPTION::DG-NORM}, \ref{A3}, can be proven similarly as in~\cite[Lemma~2.2]{Bonito_Guignard_Nochetto_Yang_2023} (see also~\cite[Appendix C]{Bonito_Guignard_Nochetto_Yang_2023}), 
whereas~\eqref{EQN::CONVERGENCE-Ch} follows from~\cite[Lemmas~2.4 and~2.5]{Bonito_Guignard_Nochetto_Yang_2023} and from the second estimate in Step 2 of the proof of ~\cite[Lemma~2.5]{Bonito_Guignard_Nochetto_Yang_2023}. 
For~$d=2$, Property~\ref{A2} is proven in the following proposition.\footnote{For~$d=3$, one could develop a similar proof based on the Fundamental Theorem of Calculus, provided that terms with third-order derivatives are added to the regularization form and to the ${\sf DG}$ norm. However, comparing with standard Sobolev embeddings, one expects the discrete Sobolev embedding to be valid also in 3D with definitions~\eqref{EQN::DGHess}--\eqref{EQN::DG2NORM}. This issue remains open.}

\begin{proposition}[Discrete Sobolev embedding in~2D]
Let~$\Omega \subset \IR^2$ be an open, bounded polytopic domain, and let the DG norm be defined as in~\eqref{EQN::DG2NORM} with~$N = 1$. Then, for all~$v \in \Sp(\Th)$,
$$\Norm{v}_{L^\infty(\Omega)} \lesssim \Norm{v}_{\sf DG},$$
where the hidden constant is independent of~$h$ and~$v$.
\end{proposition}

\begin{proof}
Let~$v \in \Sp(\Th)$ and~$(\ox, \oy)$ be an interior point of some element~$K \in \Th$. If~$\Omega$ is convex, we define an auxiliary domain~$\tOmega := [(-\infty, \ox) \times (-\infty, \oy)] \cap \Omega$, and an auxiliary mesh~$\tTh$ given by the ``intersection" of~$\Th$ and~$\tOmega$. We illustrate these definitions in Figure~\ref{FIG::OMEGA}.
If~$\Omega$ is not convex, let~$(\ox,y_{\partial\Omega})$ be the intersection of the half-line~$(-\infty, \ox)$ with~$\partial\Omega$ having the largest~$y$-coordinate, and~$(x_{\partial\Omega},\oy)$ be the intersection of the half-line~$(-\infty, \oy)$ with~$\partial\Omega$ having the largest~$x$-coordinate.   
We let~$\Gamma_x$ and $\Gamma_y$ be
the segments with endpoints $(\ox,\oy)$ and~$(\ox,y_{\partial\Omega})$ and~$(x_{\partial\Omega},\oy)$, respectively. Then, we define~$\tOmega$ as the connected subregion of $\Omega$ delimited by $\Gamma_x$, $\Gamma_y$
on the side where the angle between~$\Gamma_x$ and $\Gamma_y$ equals~$\pi/2$. 
\begin{figure}[H]
\centering
\includegraphics[width = 0.45\textwidth, trim={2cm 2cm 2cm 2cm},clip]{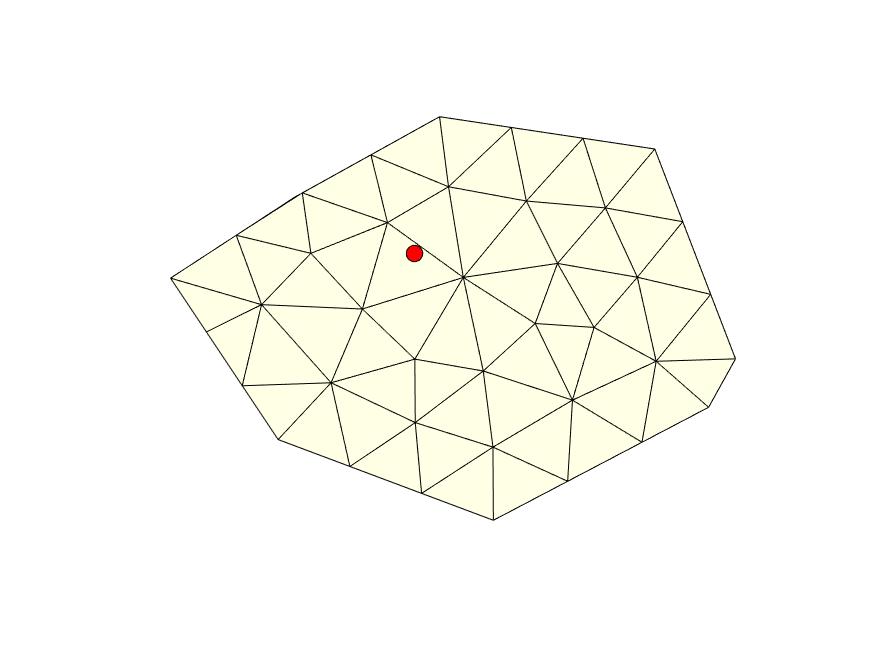}
\includegraphics[width = 0.45\textwidth, trim={2cm 2cm 2cm 2cm},clip]{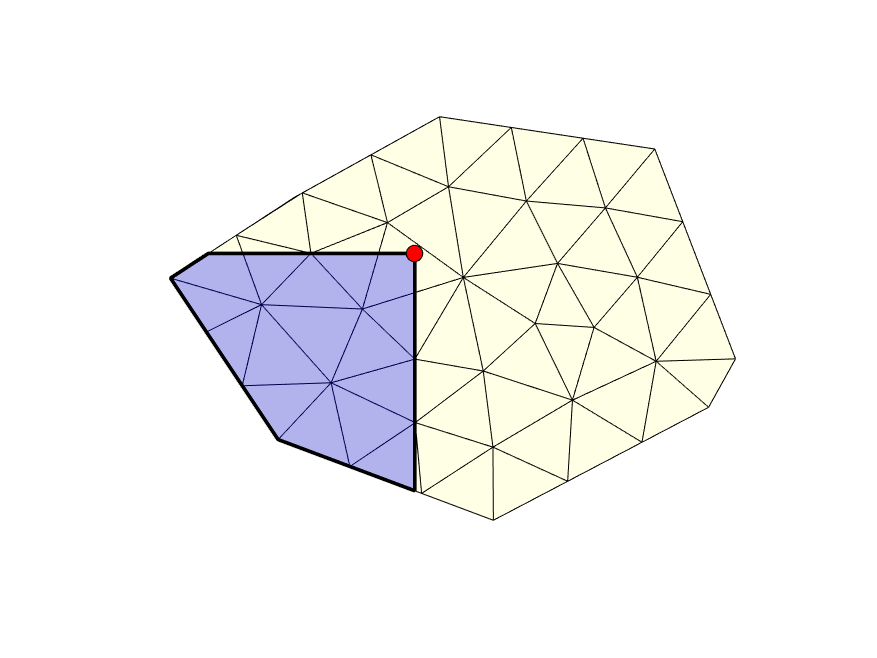}
\caption{Example of a two dimensional domain~$\Omega$ (in \textcolor{DarkYellow}{yellow}). 
\textbf{Left panel}: Triangular mesh~$\Th$ of~$\Omega$ and an interior point~$(\ox, \oy)$ (depicted with a \textcolor{red}{red} dot) of some element~$K\in \Th$. 
\textbf{Right panel}: Auxiliary domain~$\tOmega$ (in \textcolor{blue}{blue}) and auxiliary mesh~$\tTh$.}
\label{FIG::OMEGA}
\end{figure}
Integration by parts with respect to~$x$ gives
\begin{equation*}
\label{EQN::INT-BY-PARTS-X}
\sum_{\tK \in \tTh} \int_{\tK} \pxy v \,\dx = \int_{\tFho} \jumpx{\partial_{y, h} v} \dS + \int_{\tFhN} \partial_{y, h} v \,n_{\tOmega}^x \dS,
\end{equation*}
where~$\jumpx{\cdot}$ denotes the first component of the normal jump~$\jump{\cdot}_{\sf N}$, $\partial_{y, h}$ the elementwise partial $y$-derivative, and $n_{\tOmega}^x$ the first component of the unit normal vector pointing outside~$\tOmega$. The boundary of~$\tOmega$ can be split into three parts as~$\partial \tOmega = (\partial \Omega \cap \partial \tOmega) \cup \partial \tOmega^{\ox} \cup \partial \tOmega^{\oy}$, where~$\tOmega^{\ox}$ and~$\tOmega^{\oy}$ are the parts of~$\partial \tOmega$ along the lines ~$x=\ox$ and~$y=\oy$, respectively. Observe that 
$$n_{\tOmega}^x = \begin{cases} 
0 & \text{ on } \partial \tOmega ^{\oy},\\
1 & \text{ on } \partial \tOmega^{\ox},\\
n_{\Omega}^x & \text{ on } \partial \Omega \cap \partial \tOmega,
\end{cases}
$$
whence, 
\begin{equation*}
\sum_{\tK \in \tTh} \int_{\tK} \pxy v\,\dx = \int_{\tFho} \jumpx{\partial_{y, h} v} \dS + \int_{\tFhN \cap \partial \Omega} \partial_{y, h} v \,n_{\Omega}^x \,\dS +  \int_{\tFhN \cap \partial \tOmega^{\ox}} \partial_{y, h} v \,\dS.
\end{equation*}
We now focus on the last term of the previous identity. Let~$\{(\ox, y_j)\}_{j = 1}^{\ell}$, with~$\ell \in \IN$, be the set containing all internal vertices of~$\Th$ that lie on~$\partial\tOmega^{\ox}$, as well as all intersections between~$\partial\tOmega^{\ox}$ and those edges in~$\Fho$ that do not lie along~$\partial\tOmega^{\ox}$. 
We assume that the points in~$\{(\ox, y_j)\}_{j = 1}^{\ell}$ are ordered with decreasing $y$-coordinate. 
Furthermore, we denote by~$(\ox, y^{\partial})$ the intersection between~$\partial\tOmega^{\ox}$ and~$\partial\Omega$.
In Figure~\ref{FIG::VERTICES}, we illustrate the notation used for the vertices of~$\tTh$ lying on~$\partial \tOmega^{\ox}$.\footnote{The boundary~$\partial \tOmega^{\ox}$ crosses a vertex of~$\Th$ (green dot in the middle) and an internal edge of~$\Th$ (between the  two green dots from the bottom up). This is not an issue, as the domain~$\tOmega$ sees~$\partial \tOmega^{\ox}$ only from the interior.}
\begin{figure}[H]
\centering
\includegraphics[width = 0.6\textwidth]{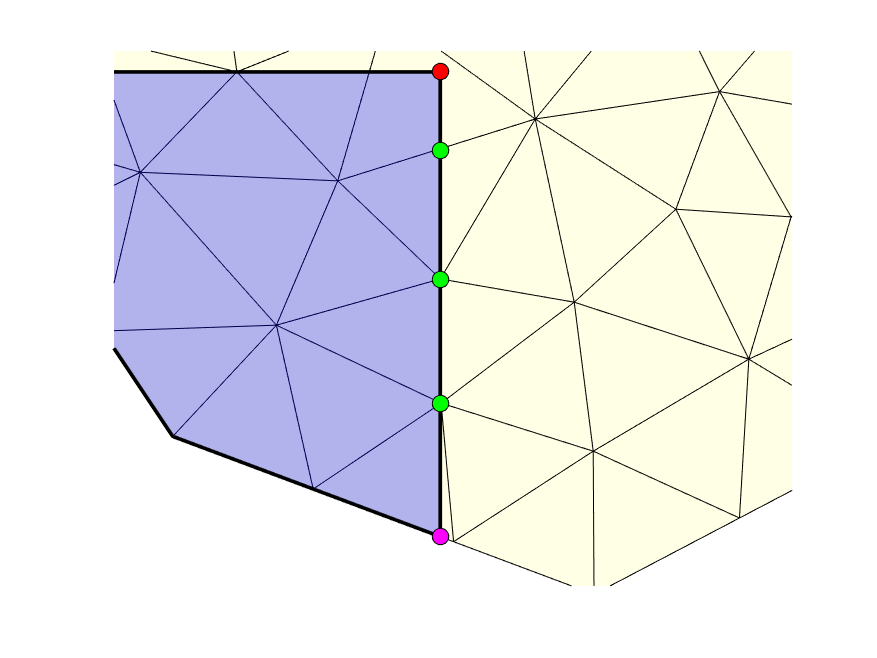}
\caption{Example of vertices of~$\tTh$ lying along~$\partial \tOmega^{\ox}$. The \textcolor{red}{red} dot has the coordinates~$(\ox, \oy)$;
the \textcolor{DarkGreen}{green} dots have the coordinates~$\{(\ox, y_j)\}_{j = 1}^{\ell}$ for some~$\ell \in \IN$;
the \textcolor{purple}{purple} dot belongs to~$\partial \Omega$ and has the coordinates~$(\ox, y^{\partial})$.
\label{FIG::VERTICES}}
\end{figure}
By the Fundamental Theorem of Calculus in one dimension,
we have
$$\int_{\tFhN \cap \partial \tOmega^{\ox}} \partial_{y, h} v \,\dS = v(\ox, \oy) - \sum_{j = 1}^{\ell} \jump{v(\ox, y_j)} - v(\ox, y^{\partial}),$$
where $\jump{v(\ox, y_j)}:=\lim_{\varepsilon\to 0}\abs{v(\ox, y_j+\varepsilon)-v(\ox, y_j-\varepsilon)}$.
Therefore,
\begin{equation*}
\begin{split}
v(\ox, \oy) &= \sum_{\tK \in \tTh} \int_{\tK} \pxy v \,\dx - \int_{\tFho} \jumpx{\partial_{y, h} v}\dS - \int_{\tFhN \cap \partial \Omega} \partial_{y, h} v \,n_{\Omega}^x\,\dS + \sum_{j = 1}^{\ell} \jump{v(\ox, y_j)} + v(\ox, y^{\partial}) \\
& = J_1 + J_2 + J_3 + J_4 + J_5.
\end{split}
\end{equation*}
We estimate the terms~$J_i$, $i = 1, \ldots, 5$, separately.
\paragraph{Bound on~$J_1$:} Proceeding as in~\eqref{EQN::L1-L2-v}, we find that
\begin{align*}
J_1 & \leq |\Omega|^{\frac12} \bigg(\sum_{K \in \Th} \Norm{\pxy v}_{L^2(K)}^2\bigg)^{\frac12} \lesssim \Norm{v}_{\sf DG}.
\end{align*}
\paragraph{Bound on~$J_2$:} Since~$|n_x| \leq \sqrt{n_x^2 + n_y^2} \leq 1$, 
proceeding as in~\eqref{EQN::L1-L2-jumps} gives
\begin{align*}
J_2 \leq \int_{\tFho} |\jumpx{\partial_{y,h} v}| \dS \leq \int_{\Fho} |\jump{\partial_{y, h} v}| \dS  \lesssim  \Norm{\mathsf{h}^{-\frac12} \jump{\partial_{y, h}v}}_{L^2(\Fho)}
\lesssim \Norm{\mathsf{h}^{-\frac12} \jump{\nabla_h v}}_{L^2(\Fho)^d}
\lesssim \Norm{v}_{\sf DG}.
\end{align*}

\paragraph{Bound on~$J_3$:} The broken trace estimate in Lemma~\ref{LEMMA::BROKEN-TRACE}
for~$r = 2$ implies that
\begin{align*}
J_3 \leq \Norm{\partial_{y, h}v }_{L^1(\partial \Omega)} \lesssim 
\Norm{\partial_{y,h} v}_{L^2(K)}
+ \Norm{\nabla_h (\partial_{y,h} v)}_{L^2(K)^{d}}
+ \Norm{\mathsf{h}^{-\frac12} \jump{\partial_{y, h} v}}_{L^2(\Fho)} \lesssim \Norm{v}_{\sf DG}.
\end{align*}

\paragraph{Bound on~$J_4$:} The \textcolor{DarkGreen}{green} dots in Figure~\ref{FIG::VERTICES} with coordinates~$\{(\ox, y_j)\}_{j = 1}^{\ell}$ may be either: \emph{i)} an internal point of some edge~$e \in \Fho$, or \emph{ii)} a vertex of~$\Th$. Both situations are represented in Figure~\ref{FIG::JUMPS}.

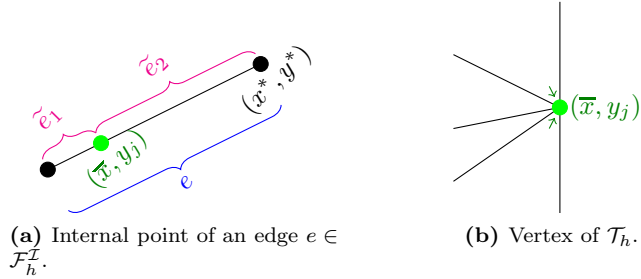
\begin{figure}[H]
\centering
\subfloat[Internal point of an edge~$e \in \Fho$.]{
\begin{tikzpicture}[scale = 1.4]
\draw[black] (0, 0) -- (2, 1);
\filldraw[black] (0, 0) circle (2pt);
\filldraw[black] (2, 1) circle (2pt);
\filldraw[green] (0.5, 0.25) circle (2pt);
\draw (0.5, 0.25) node[DarkGreen, below, rotate = 45] {$(\ox, y_j)$}; 
\draw (2, 1) node[black, below, rotate = 45] {$(x^*, y^*)$}; 
\draw [magenta, decorate,decoration={brace,amplitude=5pt,raise=1ex}]
  (0, 0) -- (0.5, 0.25) node[midway,xshift = -1em, yshift=1.5em, rotate = 45]{$\widetilde{e}_1$};
\draw [magenta, decorate,decoration={brace,amplitude=5pt,raise=1ex}]
  (0.5, 0.25) -- (2, 1) node[midway,xshift = -1em, yshift=1.5em, rotate = 45]{$\widetilde{e}_2$};
\draw [blue, decorate,decoration={brace,amplitude=5pt,mirror, raise=4ex}]
  (0,0) -- (2,1) node[midway,xshift = 1.2em, yshift=-2.5em, rotate = 45]{$e$};
\end{tikzpicture}
}
\hspace{0.5in}
\subfloat[Vertex of~$\Th$.]{
\begin{tikzpicture}[scale = 1.4]
\draw (0, 1) node[DarkGreen, right] {$(\ox, y_j)$}; 
\draw[black] (0, 0) -- (0, 2);
\draw[black] (0, 1) -- (-1, 1.5);
\draw[black] (0, 1) -- (-1, 0.8);
\draw[black] (0, 1) -- (-1, 0.3);
\filldraw[green] (0, 1) circle (2pt);
\draw[->, DarkGreen] (-0.1, 1.2) -- (-0.05, 1.1);
\draw[->, DarkGreen] (-0.1, 0.8) -- (-0.05, 0.9);
\end{tikzpicture}
}
\caption{
 Illustration of the two types of points in the set~$\{(\ox, y_j)\}_{j = 1}^{\ell}$ used in the bound on~$J_4$.
\label{FIG::JUMPS}
}
\end{figure}
\noindent We consider each case separately. 

\textbf{Case \emph{i)}} Let~$e \subset \Fho$ be the edge containing~$(\ox, y_j)$, and let~$\widetilde{e}_{\max}$ be the largest segment of~$e$ having~$(\ox, y_j)$ as a vertex. Let $(x^*,y^*)$ be the remaining vertex.
We deduce from~$h_{\widetilde{e}^{\max}} \geq \frac12 h_e$ that~$h_{\widetilde{e}^{\max}}^{-1} \leq 2 h_e^{-1}$. 
Set~$\Phi(t) := \jump{v(\ox + t(x^* - \ox), y_j +  t(y^* - y_j))}$. Since~$\Phi \in \Pp{p}{(0, 1)}$, the inverse trace inequality $\Norm{v}_{L^1(\partial D)} \lesssim h_D^{-1} \Norm{v}_{L^1(D)}$ shows that
\begin{equation*}
|\Phi(0)|  = |\jump{v(\ox, y_j)}| \lesssim h_{\widetilde{e}_{\max}}^{-1} \Norm{\jump{v}}_{L^1(\widetilde{e}_{\max})} \lesssim  \Norm{h_e^{-1} \jump{v}}_{L^1(e)}.
\end{equation*}

\textbf{Case \emph{ii)}} Adding and subtracting the values of~$v$ at~$(\ox, y_j)$ from all the elements having~$(\ox, y_j)$ as a vertex, and using the triangle inequality, one can proceed as in case~\emph{i)}.

\textbf{Conclusion of the bound on~$J_4$:} Since the jumps at different points in~$\{(\ox, y_j)\}_{j = 1}^{\ell}$
are ``lifted" to different edges, proceeding as in~\eqref{EQN::L1-L2-jumps}, we conclude that
$$J_4=\sum_{j = 1}^{\ell} \jump{v(\ox, y_j)} 
\leq \sum_{j = 1}^{\ell} |\jump{v(\ox, y_j)}| 
\lesssim \Norm{\mathsf{h^{-1}}\jump{v}}_{L^1(\Fho)} 
\lesssim \Norm{\mathsf{h^{-\frac32}}\jump{v}}_{L^2(\Fho)}
= \Norm{\mathsf{h}^{-\frac32} \jump{v}_{\sf N}}_{L^2(\Fho)^d} \lesssim \Norm{v}_{\sf DG}. $$

\paragraph{Bound on~$J_5$:} Let~$(\hat{x}, \hat{y})$ be a vertex of $\Omega$ such that the segment~$\tGamma := [(\hat{x}, \hat{y}), (\ox, y^{\partial})] \subset \partial \tOmega\cap \partial\Omega$ has positive~$1$-dimensional measure\footnote{The argument used to bound~$J_5$ is independent of whether~$(\ox, y^{\partial})$ is a mesh vertex or not.},
and let~$\{(\hat{x}_i, \hat{y}_i)\}_{i = 1}^{k}$, with~$k \in \IN$, be the vertices of~$\Th$ in the interior of~$\tGamma$. Then, by the Fundamental Theorem of Calculus,
\begin{equation}
\label{EQN::AUX-V-XYp}
|v(\ox, y^{\partial})| \leq |v(\hat{x}, \hat{y})| + \Norm{\nabla_{\tau, h} v}_{L^1(\tGamma)} + \sum_{i = 1}^{k} |\jump{v(\hat{x}_i, \hat{y}_i)}|,
\end{equation}
where~$\nabla_{\tau, h}$ denotes the broken tangential derivative of~$v$.
Furthermore, applying  Lemma~\ref{LEMMA::BROKEN-TRACE} with~$r = 1$ along the side~$\Gamma$ of the boundary of~$\Omega$
containing~$\tGamma$, we obtain
$$|v(\hat{x}, \hat{y})|
\lesssim \Norm{v}_{L^1(\Gamma)} + \Norm{\nabla_{\tau, h} v}_{L^1(\Gamma)} + \sum_{i = 1}^{M} |\jump{v(\hat{x}_i, \hat{y}_i)}| \leq \Norm{v}_{L^1(\partial \Omega)} + \Norm{\nabla_h v}_{L^1(\partial \Omega)^d} + \sum_{i = 1}^{M} |\jump{v(\hat{x}_i, \hat{y}_i)}|,$$
where~$\{(\hat{x}_i, \hat{y}_i)\}_{i = 1}^{M}$, with~$M\ge k$, are the vertices of~$\Th$ in the interior of~$\Gamma$.
This, combined with~\eqref{EQN::AUX-V-XYp}, leads to
$$|v(\ox, y^{\partial})| \lesssim \Norm{v}_{L^1(\partial \Omega)} + \Norm{\nabla_h v}_{L^1(\partial \Omega)^d} + \sum_{i = 1}^{M} |\jump{v(\hat{x}_i, \hat{y}_i)}| =: I_1 + I_2 + I_3.$$
Using Lemma~\ref{LEMMA::BROKEN-TRACE} with~$r = 2$, the terms~$I_1$ and~$I_2$ can be estimated as follows:
\begin{subequations}
\begin{align}
\label{EQN::I_1-Triang}
I_1 & \lesssim \Norm{v}_{L^2(\Omega)} + 
\Norm{\nabla_h v}_{L^2(\Omega)^{d}}
+ \Norm{\mathsf{h}^{-\frac12}\jump{v}_{\sf N}}_{L^2(\Fho)^d} \lesssim \Norm{v}_{\sf DG}, \\
\label{EQN::I_2-Triang}
I_2 & \lesssim 
\Norm{\nabla_h v}_{L^2(K)^d}
+ \Norm{D_h^2 v}_{L^2(K)^{d\times d}}
+ \Norm{\mathsf{h}^{-\frac12} \jump{\nabla_h v}}_{L^2(\Fho)^d} \lesssim \Norm{v}_{\sf DG}.
\end{align}
Moreover, proceeding as for bound~$J_4$, case~\emph{ii)}, the term~$I_3$ can be estimated as
\begin{equation}
\label{EQN::I_3-Triang}
I_3 = \sum_{i = 1}^{M} |\jump{v(\hat{x}_i, \hat{y}_i)}| \lesssim \Norm{\mathsf{h}^{-\frac32} \jump{v}_{\sf N}}_{L^2(\Fho)^d} \lesssim \Norm{v}_{\sf DG}.
\end{equation}
\end{subequations}
It follows from~\eqref{EQN::I_1-Triang}, \eqref{EQN::I_2-Triang}, and~\eqref{EQN::I_3-Triang} that $J_5 \lesssim \Norm{v}_{\sf DG}$,
which completes the proof.

\begin{figure}[H]
    \centering
    \includegraphics[width = 0.45\textwidth, trim={0.5cm 3cm 0.5cm 0cm}]{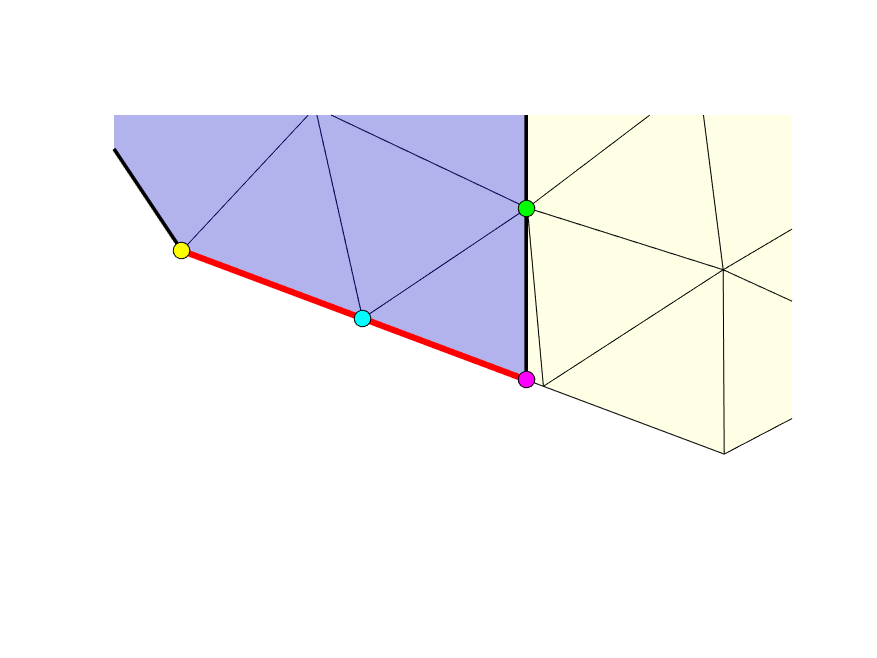} 
    \hspace{0.05in}
    \includegraphics[width = 0.45\textwidth, trim={0.5cm 3cm 0.5cm 0cm}]{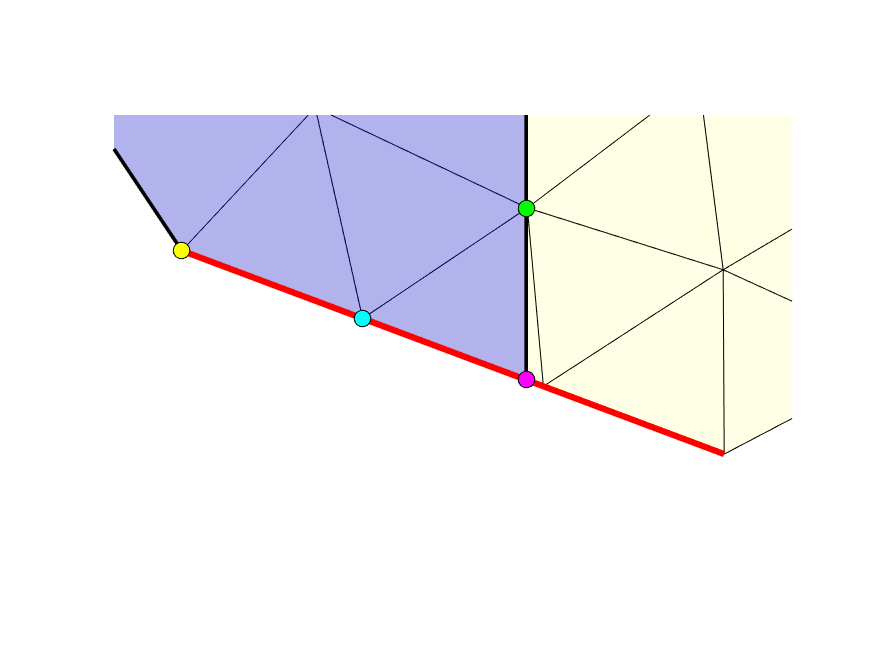} 
\caption{Example of the auxiliary segments $\tGamma$ (\textbf{left panel}) and~$\Gamma$ (\textbf{right panel}) in the bound on~$J_5$. The \textcolor{DarkYellow}{yellow} dot has the coordinates~$(\hat{x}, \hat{y})$ and is a vertex of~$\Omega$. The \textcolor{DarkBlue}{cyan} dot has the coordinates~$(\hat{x}_1, \hat{x}_2)$}.
\end{figure}
\end{proof}

\section{Numerical experiments \label{SECT::NUMERICAL-EXP}}
In this section, we assess the accuracy and entropy stability of the proposed method with some one- and two-dimensional test problems.
The solutions to the nonlinear systems of equations stemming from the fully discrete method~\eqref{EQN::FULLY-DISCRETE-SCHEME-INIT}--\eqref{EQN::FULLY-DISCRETE-SCHEME} are approximated using a quasi-Newton method, where the Jacobian of the nonlinear vector-valued function is evaluated on the approximation at the previous time. 
The tolerance~$(tol)$ and the maximum number of linear iterations~$(s_{\max})$ of the nonlinear solver are specified in each test.

We use Gaussian elimination~(for the one-dimensional problems) or a preconditioned BICG method (for the two-dimensional problems) to solve the linear system at each iteration of the nonlinear solver. 
In order to reduce the stencil of the gradient operator matrix~$B$, we use directional numerical fluxes. 
More precisely, for all~$F \in \Fho$, we set~$\alpha_F = 1$ and take~$\vn_F = 1$ in 1D, or~$\vn_F = \vnK$ in 2D, where~$K \in \Th$ is the element such that~$F \subset \partial K$ and~$(1,1)^{\top} \cdot \vnK \le 0$. These choices of~$\alpha_F$ and~$\vn_F$ yield a compact stencil in 1D and for structured simplicial meshes in 2D; see~\cite{Castillo_2010}.

\subsection{One-dimensional porous medium equation \label{SUBSECT::POROUS-MEDIUM}}
Given a real number~$m > 1$, an initial datum~$\rho_0 : \Omega \rightarrow \calD$, and a Neumann boundary datum~$g_N : \partial \Omega \times (0, T) \rightarrow \IR$, we consider the following problem
on a space--time cylinder~$\QT = \Omega \times (0, T]$:
\begin{equation}
\label{EQN::POROUS-MEDIUM}
\begin{cases}
\dpt \rho - \partial_{xx} \rho^m = 0 & \text{ in } \QT, \\
\partial_x (\rho^m) \vnOmega = g_N & \text{ on } \partial \Omega \times (0, T), \\
\rho = \rho_0 & \text{ on } \Omega \times \{0\},
\end{cases}
\end{equation}
where the first equation can be written as~\eqref{EQN::MODEL-PROBLEM-1} with~$N = 1$, $A(\rho) = m \rho^{m - 1}$, and~$f(\rho) = 0$.

We set~$\calD = (0, 1)$ and define the entropy density~$s : \overline{\calD} \rightarrow (0, \infty)$ as follows:
\begin{equation*}
s(\rho) := \rho \log(\rho) + (1 - \rho)\log(1 - \rho) + \log(2),
\end{equation*}
whence, $s'(\rho) = \log\big(\frac{\rho}{1- \rho} \big)$, $s''(\rho) = \frac{1}{\rho(1 - \rho)}$, and~$u(w) = \frac{e^w}{1 + e^w}$. 
For this choice of~$s(\cdot)$, Assumptions~\ref{H1}--\ref{H2c} are satisfied with~$\gamma = m$ and~$C_f = 0$, provided that~$m \in (1, 2]$; see~\cite[Prop.~4.2]{Braukhoff_Perugia_Stocker_2022}.

\paragraph{$h$-convergence.} In order to appraise the accuracy of the proposed method, we consider problem~\eqref{EQN::POROUS-MEDIUM} with~$\Omega = (0, 1)$ and~$m = 2$, and choose the initial datum~$\rho_0$ and the Neumann boundary datum~$g_N$ so that the exact solution is given by
\begin{equation}
\label{EQN::EXACT-SOL-POROUS}
\rho(x, t) = \Bigg[\frac{(m - 1)(x - \alpha)^2}{2m(m +1)(\beta - t)} \Bigg]^{\frac{1}{m - 1}},
\end{equation}
with~$\alpha = 2$ and~$\beta = 5$; cf. \cite[\S4.2]{Braukhoff_Perugia_Stocker_2022}.

We choose the parameters of the nonlinear solver as~$tol = 10^{-12}$ and~$s_{\max} = 50$.
We consider a set of meshes with uniformly distributed points for the spatial domain~$\Omega$, and choose~$\tau = \mathcal{O}(h^{p+1})$ to balance the expected convergence rates in space with the first-order accuracy of the backward Euler time stepping scheme. Moreover, we set the regularization parameter to~$\varepsilon = 0$.

In Figure~\ref{FIG::POROUS-MEDIUM-CONVERGENCE}, we show (in \emph{log-log} scale) the following errors obtained at the final time~$T = 1$:
\begin{equation}
\label{EQN::ERRORS-1D}
\Norm{\rho - u(w_h)}_{L^2(\Omega)} \quad \text{ and } \quad \Norm{\partial_x \rho + \sigma_h}_{L^2(\Omega)}.
\end{equation}
We observe, as expected, 
convergence rates of order~$\mathcal{O}(h^{p+1})$ and~$\mathcal{O}(h^{p})$, respectively.

\begin{figure}[ht]
    \centering
    \includegraphics[width = 0.4\textwidth]{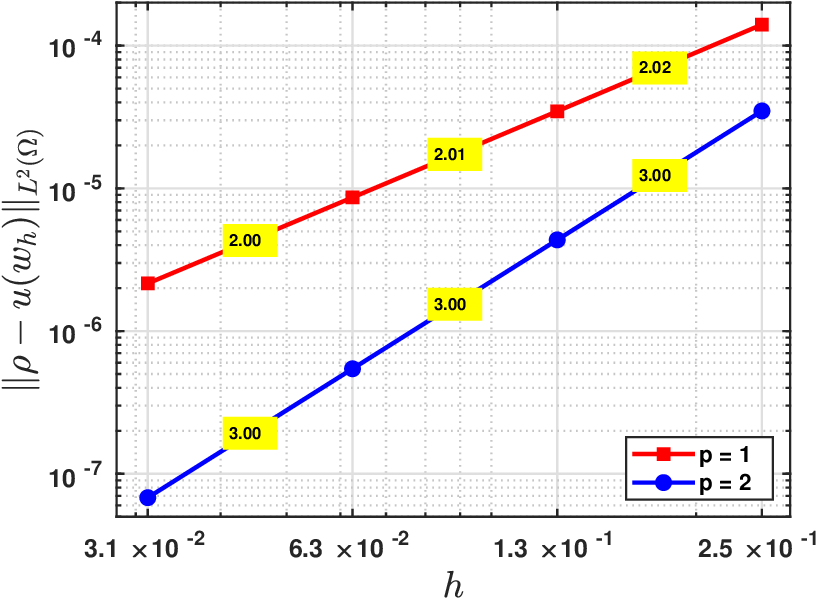}
    \hspace{0.2in}
    \includegraphics[width = 0.4\textwidth]{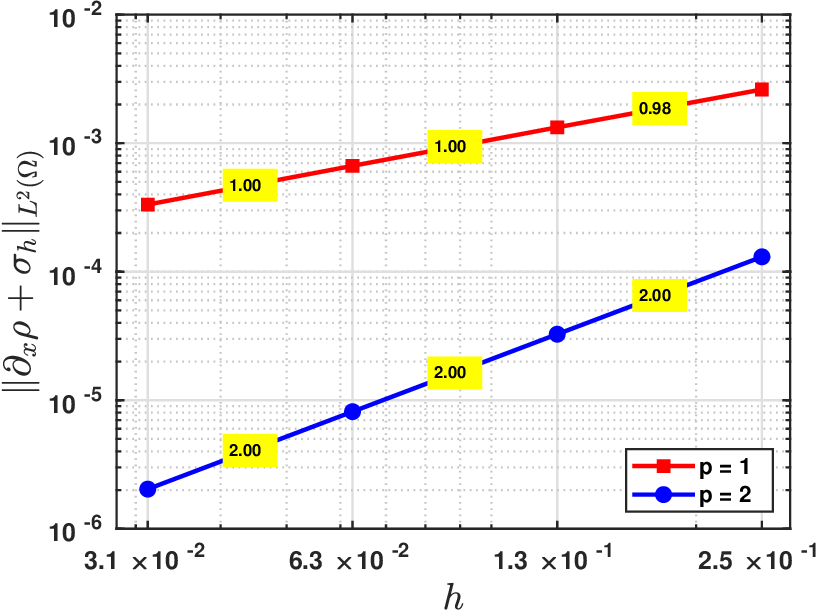}
    \caption{$h$-convergence of the errors in~\eqref{EQN::ERRORS-1D} at the final time~$T = 1$ for the porous medium equation with exact solution~$\rho$ in~\eqref{EQN::EXACT-SOL-POROUS}. The numbers in the
yellow rectangles denote the experimental rates of convergence. }
    \label{FIG::POROUS-MEDIUM-CONVERGENCE}
\end{figure}

\paragraph{Entropy stability.} We now consider 
problem~\eqref{EQN::POROUS-MEDIUM} with~$\Omega = (-\pi/4, 5\pi/4)$, $m = 2$, homogeneous Neumann boundary conditions, and 
initial datum given by
\begin{equation}
\label{EQN::INITIAL-CONDITION-WAITING-TIME}
\rho_0(x) = \begin{cases}
\sin^{2/(m - 1)}(x) & \text{ if } 0 \leq x \leq \pi, \\
0 & \text{ otherwise,}
\end{cases}
\end{equation}
whose exact solution keeps the support~$[0, \pi]$ of the initial condition until the waiting time~$t^* = (m-1)/(2m(m+1))$; see~\cite{Aronson_1970}.

We choose the parameters for the nonlinear solver as~$tol = 10^{-6}$ and~$s_{\max} = 100$, and consider~$T = 0.2$ as the final time. 
Moreover, we set the regularization parameter as~$\varepsilon = 10^{-6}$ and the bilinear form~$c_h(\cdot, \cdot)$ as in~\eqref{EQ::DGPROD}.
In Figure~\ref{FIG::POROUS-MEDIUM-FINITE-PROPAGATION}(first panel), we show the discrete approximation obtained for~$p = 5$, a spatial mesh with uniformly distributed points and mesh size~$h \approx 0.04$, and a fixed time step~$\tau = 10^{-3}$. 
To represent the discrete solution, we have used linear interpolation in time, which preserves the uniform boundedness of the discrete approximation.
In Figure~\ref{FIG::POROUS-MEDIUM-FINITE-PROPAGATION}(second panel), we show the value of the discrete approximation at~$x = 0$, where the expected behavior until~$t = t^*$ is observed; cf. \cite[\S4.2]{Braukhoff_Perugia_Stocker_2022}. Since~$C_f = 0$, we expect a (not necessarily monotonous) decreasing behavior of the discrete entropy values~$\{\mathcal{E}_n\}_{n = 0}^{N_t}$, where
\begin{equation}
\label{EQN::ENTROPY-FUNCTIONAL}
\mathcal{E}_0 := \int_{\Omega} s(\rho_0) \mathrm{d}x \quad \text{ and } \quad  \mathcal{E}_n := \int_{\Omega} s(u(w_h^{\varepsilon, n + 1}(x))) \mathrm{d}x \quad \text{ for } n = 1, \ldots, N_t.
\end{equation}
Such an expected behaviour is numerically observed in Figure~\ref{FIG::POROUS-MEDIUM-FINITE-PROPAGATION}(third panel).

Moreover, we define the discrete mass values~$\{\mathcal{M}_n\}_{n = 0}^{N_t}$ as
\begin{equation}
\label{EQN::MASS-FUNCTIONAL}
\mathcal{M}_0 := \int_{\Omega} \rho_0 \mathrm{d}x \quad \text{ and } \quad 
\mathcal{M}_n := \int_{\Omega} u(w^{\varepsilon, n}) \mathrm{d}x \quad \text{ for } n = 1, \ldots, N_t.
\end{equation}
Since~$f(\rho) = 0$, mass is conserved for analytical solutions. Standard arguments can be used to show that, for any solution~$\{w^{\varepsilon, n + 1}\}_{n = 0}^{N_t - 1}$ to the fully discrete scheme~\eqref{EQN::FULLY-DISCRETE-SCHEME-INIT}--\eqref{EQN::FULLY-DISCRETE-SCHEME}, for~$n = 0, \ldots, N_t - 1$, 
\begin{align*}
|\mathcal{M}_{n+1} - \mathcal{M}_0|  & \leq  \varepsilon \sum_{m = 1}^n \tau_{m+1} \int_\Omega |w^{\varepsilon, m + 1}| \mathrm{d}x  \leq \varepsilon |\Omega|^{\frac12} \sum_{m = 1}^{N_t - 1} \tau_{m+1} \Norm{w^{\varepsilon, m+1}}_{L^2(\Omega)} \\
& \leq \sqrt{\varepsilon} |\Omega|^{\frac12} \bigg(\sum_{m = 1}^{N_t - 1} \tau_{m+1}\bigg)^{\frac12} \bigg( \sum_{m = 1}^{N_t - 1} \varepsilon \tau_{m+1} \Norm{w^{\varepsilon, m + 1}}_{L^2(\Omega)}^2 \bigg)^{\frac12}  \leq \sqrt{\varepsilon} |\QT|^{\frac12} \bigg(\int_{\Omega} s(\rho_0) \dx\bigg)^{\frac12}.
\end{align*}
In Figure~\ref{FIG::POROUS-MEDIUM-FINITE-PROPAGATION}(fourth panel), we show (in \emph{semilogy} scale) the error evolution of the mass  values for different
regularization parameters~$\varepsilon$, where a mass loss of order~$\mathcal{O}(\varepsilon)$ is numerically observed.

\begin{figure}[ht]
    \centering
\includegraphics[width = 0.43\textwidth]{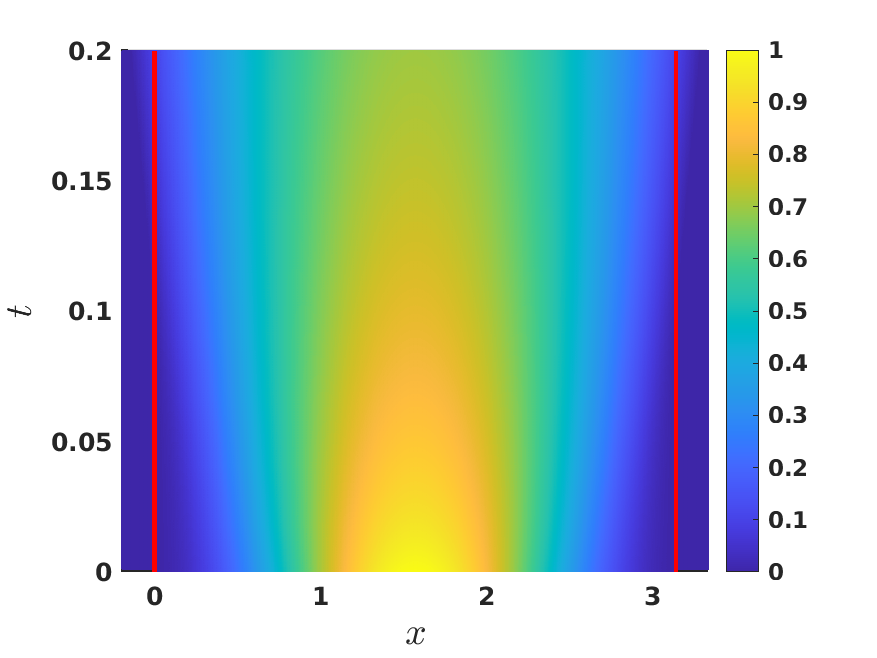}
\hspace{0.2in}
\includegraphics[width = 0.4\textwidth]{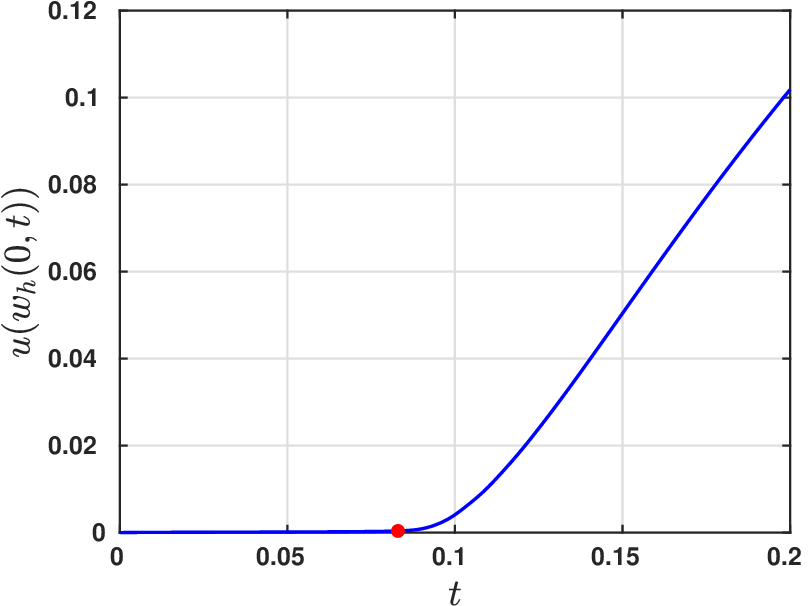}\\[1em]
\includegraphics[width = 0.4\textwidth]{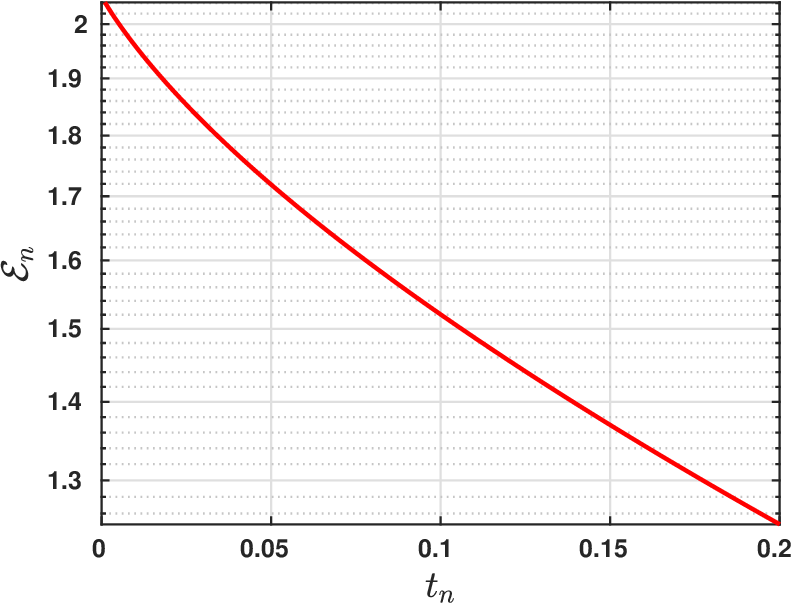}
\hspace{0.3in}
\includegraphics[width = 0.4\textwidth]{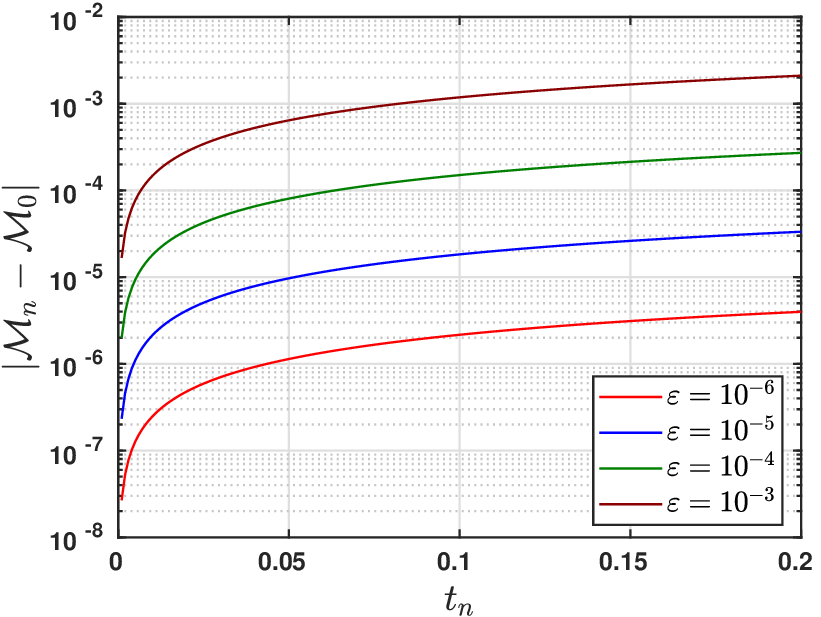}
\caption{Results obtained for the porous medium equation with initial condition~\eqref{EQN::INITIAL-CONDITION-WAITING-TIME}. \textbf{First panel:} discrete approximation~$u(w_h^{\varepsilon}(x, t))$ and support of the initial condition~(\textcolor{red}{red} lines). 
\textbf{Second panel:} evolution of the value of the discrete approximation at the extreme~$x = 0$. The theoretical waiting time has been highlighted with a \textcolor{red}{red}  dot. 
\textbf{Third panel:} evolution of the entropy values in~\eqref{EQN::ENTROPY-FUNCTIONAL} in \emph{semilogy} scale.
\textbf{Fourth panel:} Error evolution of the mass values in~\eqref{EQN::MASS-FUNCTIONAL} in \emph{semilogy} scale for different values of the regularization parameter~$\varepsilon$.}
\label{FIG::POROUS-MEDIUM-FINITE-PROPAGATION}
\end{figure}
\paragraph{Situation where the regularizing term is not necessary.} We consider problem~\eqref{EQN::POROUS-MEDIUM} with~$\Omega = (0, 1)$ and~$m = 2$, and choose the initial datum~$\rho_0$ and the Neumann boundary datum~$g_N$ so that the exact solution is given by \eqref{EQN::EXACT-SOL-POROUS} with~$\alpha = 2$ and~$\beta = 5$. We choose the parameters of the nonlinear solver as~$tol = 10^{-10}$ and~$s_{\max} = 50$.
We consider a set of meshes with uniformly distributed points for the spatial domain~$\Omega$. In Figure~\ref{FIG::GOOD-SOL}, we show some results for this problem, where the regularizing term with parameter~$\varepsilon$ {\em is not needed}. We focus on the behavior of the Newton method for the first time step. In Figure~\ref{FIG::GOOD-SOL},
\begin{itemize}
\setlength{\itemsep}{-1.5pt}
 \item First panel: We plot~$\rho$ in~\eqref{EQN::EXACT-SOL-POROUS}, and observe that it does not take values close to~$0$ or~$1$.
 \item Second panel: We show the condition number of the Jacobian matrix in each linear iteration~$s$.
 \item Third panel: We present the evolution of the~$\ell_{\infty}$ norm of the vector solution~$W^{s+1}$ at the~$s$th linear iteration.
 \item Fourth panel: We show the evolution of the stopping criterion.
\end{itemize}
Clearly, the behavior of the Newton method is similar for all~$\varepsilon$. In fact, in this experiment, we can set~$\varepsilon = 0$.

\begin{figure}[ht]
\centering
\includegraphics[width = 0.45\textwidth]{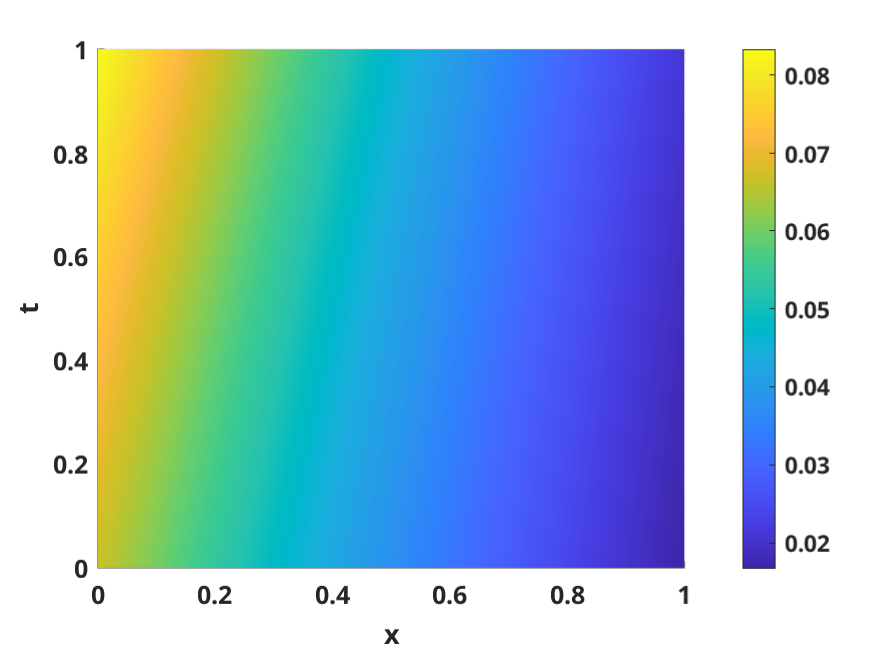}
\hspace{0.1in}
\includegraphics[width = 0.4\textwidth]{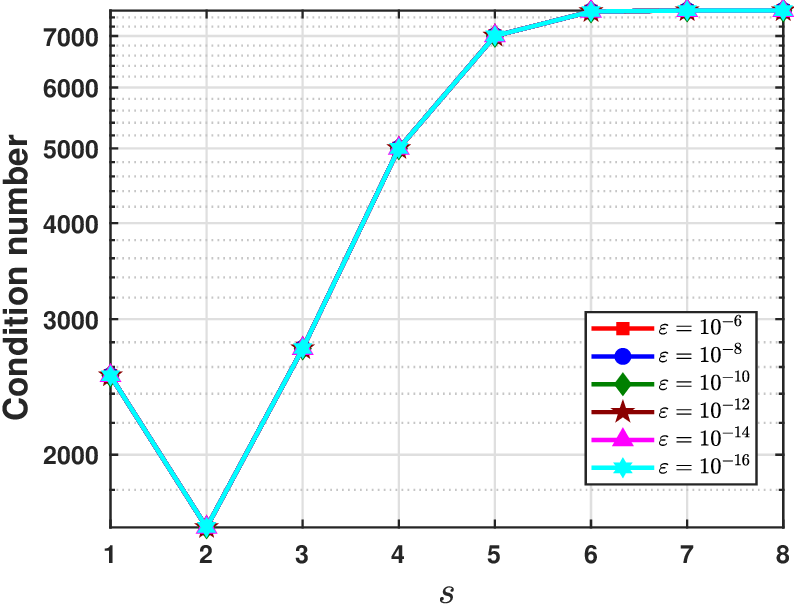} \\
\includegraphics[width = 0.4\textwidth]{
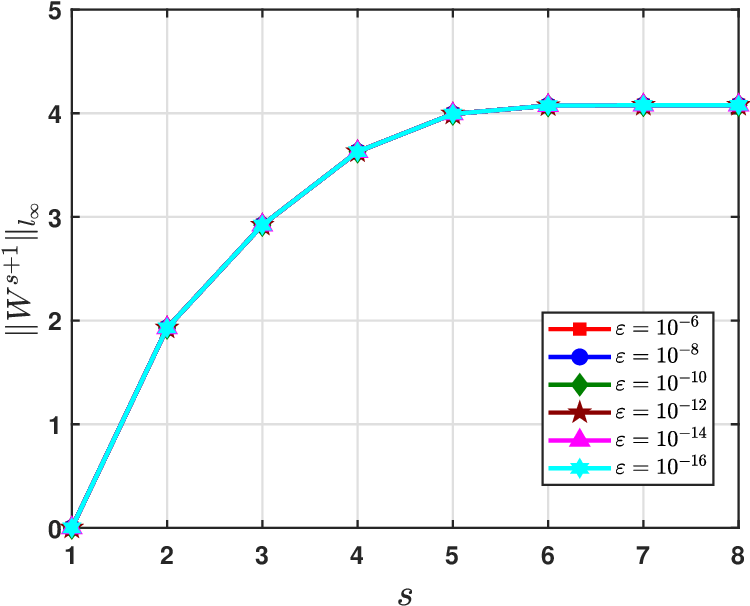}
\hspace{0.1in}
\includegraphics[width = 0.4\textwidth]{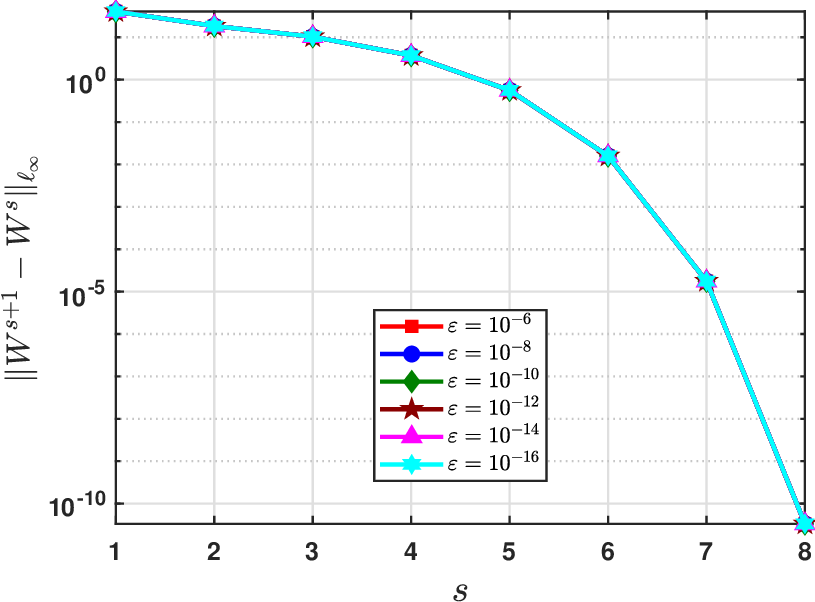}
\caption{Numerical results for the porous medium equation~\eqref{EQN::POROUS-MEDIUM} with exact solution~\eqref{EQN::EXACT-SOL-POROUS} as a function of the linear iteration number $s$}.
\label{FIG::GOOD-SOL}
\end{figure}

\paragraph{Situation where the regularizing term is necessary.} We now consider problem~\eqref{EQN::POROUS-MEDIUM} with~$\Omega = (-\pi/4$, $5\pi/4)$, $m = 2$, homogeneous Neumann boundary conditions, and 
initial datum given by \eqref{EQN::INITIAL-CONDITION-WAITING-TIME}. We choose the parameters for the nonlinear solver as~$tol = 10^{-12}$ and~$s_{\max} = 50$, and consider~$T = 0.2$ as the final time. In Figure~\ref{FIG::BAD-SOL}, we show some results for this problem, where the regularizing term with parameter~$\varepsilon$ {\em is needed}. Again, we focus on the behavior of the Newton method for the first time step. In Figure~\ref{FIG::BAD-SOL},
\begin{itemize}
\setlength{\itemsep}{-1.5pt}
 \item First panel: We plot~$u(w_h)$ for~$\varepsilon = 10^{-6}$, and observe that it takes values close to~$0$ and~$1$, especially at the beginning. This suggests that the Newton method may already encounter issues in the first time step.
 \item Second panel: We show the condition number of the Jacobian matrix in each linear iteration~$s$. The condition numbers grow less for larger values of~$\varepsilon$, and explode when~$\varepsilon$ gets closer to~$0$.
 This is in line with our intuition, as the matrix~$s''(\cdot)$ is singular at~$0$ and~$1$ (see also Remark~\ref{rem:regularizing}).
 \item Third panel: We present the evolution of the~$\ell_{\infty}$ norm of the vector solution~$W^{s+1}$ at the~$s$th linear iteration. As expected from the theory, larger values of~$\varepsilon$ enfoce a  stronger bound on~$w_h$ in the~$L^{\infty}(\Omega)$ norm.
 \item Fourth panel: We show the evolution of the stopping criterion. Clearly, the number of linear iterations necessary to reach the desired tolerance increases when~$\varepsilon$ decreases.
\end{itemize}

\begin{figure}[!ht]
\centering
\includegraphics[width = 0.45\textwidth]{exp2-plot.eps}
\hspace{0.1in}
\includegraphics[width = 0.4\textwidth]{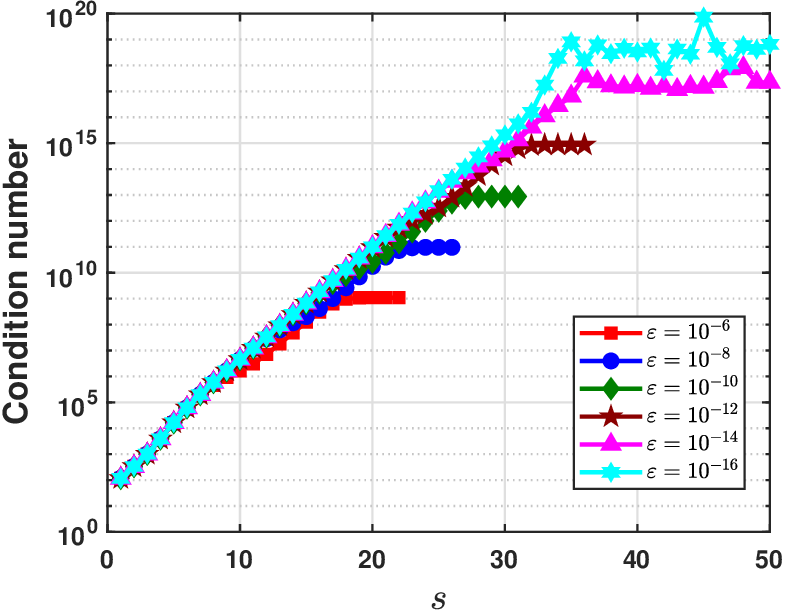} \\
\includegraphics[width = 0.4\textwidth]{
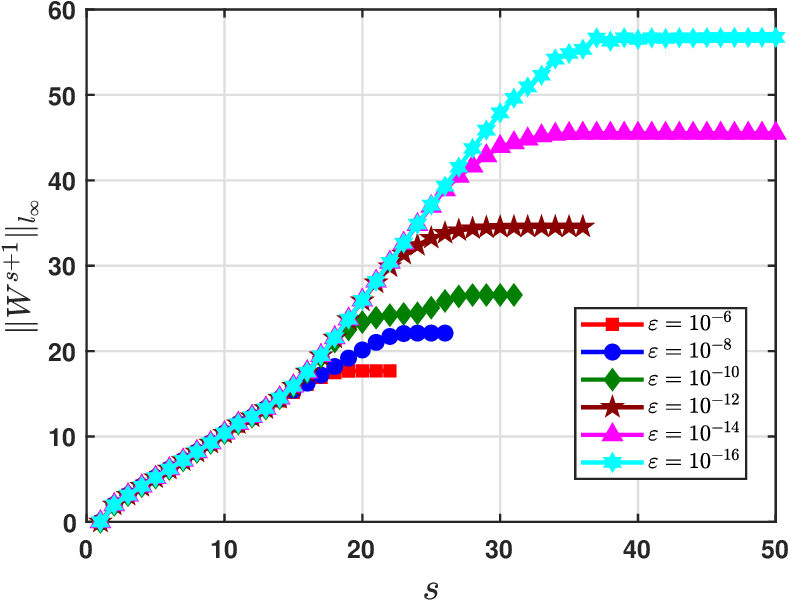}
\hspace{0.1in}
\includegraphics[width = 0.4\textwidth]{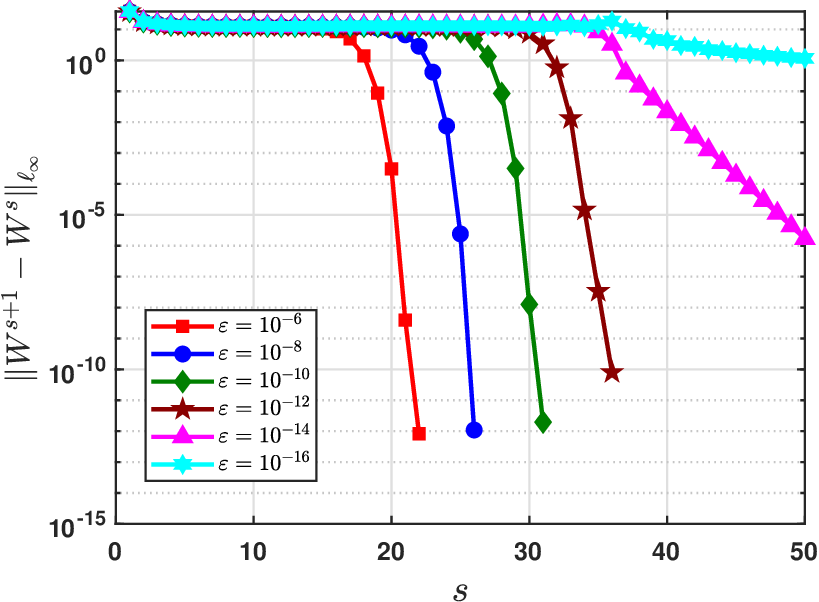}
\caption{Numerical results for the porous medium equation~\eqref{EQN::POROUS-MEDIUM} with initial condition~\eqref{EQN::INITIAL-CONDITION-WAITING-TIME} as a function of the linear iteration number $s$}.
\label{FIG::BAD-SOL}
\end{figure}

\subsection{Two-dimensional SKT model}
We consider the two-dimensional Shigesada-Kawasaki-Teramoto (SKT) population system~\cite{SKT_1979} with~$N = 2$ species. The diffusion matrix and Lotka--Volterra reaction terms (see \eqref{EQN::MODEL-PROBLEM}) read as
\begin{subequations}
\begin{align}
\label{EQN::DIFFUSION-SKT}
  A_{ij}(\brho) & = \delta_{ij} \bigg(a_{i0} + \sum_{k = 1}^2 a_{ik} \rho_k \bigg) + a_{ij} \rho_i,  \quad i,j = 1, 2, \\
  \label{EQN::REACTION-SKT}
  \vf_i(\brho) & = \rho_i 
  \bigg(b_{i0} - \sum_{j = 1}^2 b_{ij} \rho_j\bigg),  \quad \qquad \qquad  i = 1, 2,
\end{align}
\end{subequations}
for some coefficients~$\{a_{ij}\}$ and~$\{b_{ij}\}$ satisfying $a_{ii} > 0$, $b_{ii} \geq 0$ for~$i = 1, 2$, and~$a_{ij} \geq 0$, $b_{ij} \geq 0$ for~$i \neq j$.

We set~$\calD = (0, \infty)^2$ and define the entropy density~$s : (0, \infty)^2 \rightarrow (0, \infty)$ as  (see~\cite[Eq.~(6)]{Jungel_Zurek_2021})
\begin{equation}
\label{EQN::ENTROPY-SKT}
s(\brho) := \sum_{i = 1}^2 \pi_i (\rho_i(\log (\rho_i) - 1) + 1),
\end{equation}
where $\pi_1=a_{21}$ and $\pi_2=a_{12}$. Then~$s'(\brho) = (\pi_1 \log \rho_1, \pi_2 \log \rho_2)$, $s''(\brho) = \text{diag}(\pi_i/\rho_i)$, and~$\bu(\bw) = (\exp(w_1/\pi_1),$ $\exp(w_2/\pi_2))$.

Assumption~\ref{H2a} is satisfied with~$\gamma = \min_{i = 1, 2} \pi_i a_{ii} > 0$; see~\cite[\S3.1]{Jungel_Zurek_2021}. 
Moreover, if the coefficients~$\{b_{ij}\}$ are all equal to zero, then Assumption~\ref{H2b} is trivially satisfied. For general coefficients~$\{b_{ij}\}$, the reaction term satisfies the bound
\begin{equation*}
\vf(\brho) \cdot s'(\brho) \leq C_f(1 + s(\brho)) \quad \forall \brho \in \calD, \quad \text{ with } C_f = \frac{2}{\log (2)} \max_{i = 1, 2} \bigg(b_{i0} + \frac{1}{e \pi_i} \sum_{j = 1}^2 \pi_j b_{ji} \bigg),
\end{equation*}
which substitutes Assumption~\ref{H2b} in our theoretical results, by requiring that~$\tau < 1/C_f$. Notice that the domain $\calD=(0,\infty)^2$ is not bounded, as required in Assumption~\ref{H1}. As a consequence, we are not able to prove upper bounds for $\rho_i$ but only the nonnegativity of $\rho_i$; see \cite{Jungel_Zurek_2021}.

\paragraph{$h$-convergence.} We consider the SKT system with~$\Omega = (0, 1)^2$, vanishing Lotka--Volterra terms, and the diffusion parameters (cf.\ \cite[Example~5.1]{Sun_Carrillo_Shu_2019})
\begin{equation*}
a_{i0} = 0 \ \text{ for } i = 1, 2, \quad a_{ij} = 1 \ \text{ for } i,j = 1, 2.
\end{equation*}
We choose the initial datum~$\brho_0$ and add a source term so that the exact solution is given by
\begin{equation}
\label{EQN::EXACT-SOL-SKT}
    \rho_1(x, y, t) = 0.25 \cos(2\pi x) \cos(\pi y) \exp(-t) + 0.5, \quad \rho_2(x, y, t) = 0.25 \cos(\pi x) \cos(2 \pi y) \exp(-t) + 0.5.
\end{equation}

We choose the parameters of the nonlinear solver as~$tol = 10^{-6}$ and~$s_{\max} = 50$. We consider a set of structured simplicial meshes for the spatial domain~$\Omega$, choose a fixed time step~$\tau = \mathcal{O}(h^{p+1})$ as in Section~\ref{SUBSECT::POROUS-MEDIUM}, and set the regularization parameter equal to~$\varepsilon = 0$.

In Figure~\ref{FIG::CONVERGENCE-SKT}, we show (in \emph{log-log} scale) the following errors obtained at the final time~$T = 0.5$:
\begin{equation}
\label{EQN::ERRORS-2D}
\Norm{\rho_1 - u_1(w_{1, h})}_{L^2(\Omega)} \quad \text{ and } \quad \Norm{\nabla \rho_1 + \bunderline{\sigma}_{1, h}}_{L^2(\Omega)^2}, 
\end{equation}
where convergence rates of order~$\mathcal{O}(h^{p+1})$ and~$\mathcal{O}(h^p)$ are observed, respectively. Similar results were obtained for the approximation of~$\rho_2$, so they are omitted.
\begin{figure}[ht]
    \centering
    \includegraphics[width = 0.4\textwidth]{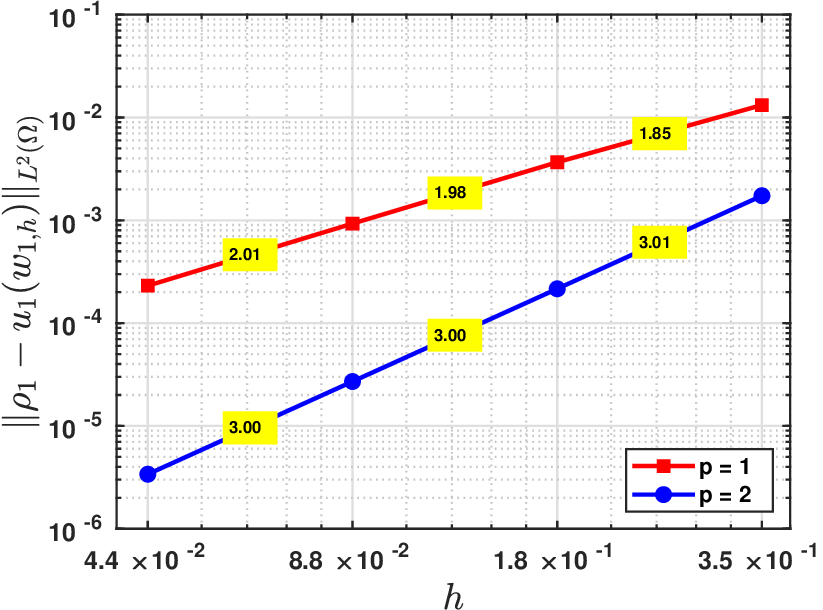}
    \hspace{0.2in}
    \includegraphics[width = 0.4\textwidth]{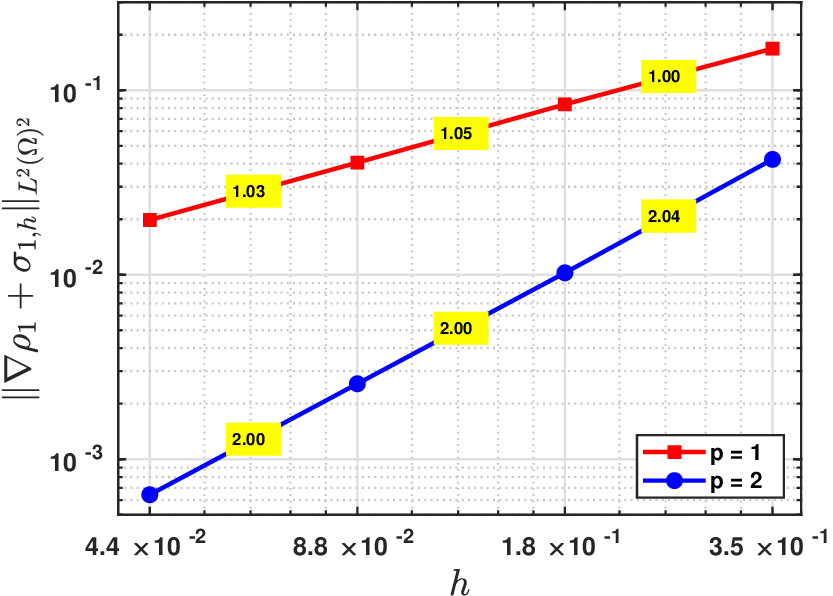}
    \caption{$h$-convergence of the errors in~\eqref{EQN::ERRORS-2D} at the final time~$T = 0.5$ for the SKT system with exact solution~$\brho$ in~\eqref{EQN::EXACT-SOL-SKT}.}
    \label{FIG::CONVERGENCE-SKT}
\end{figure}

\paragraph{Turing pattern.} We now consider a test from~\cite[\S7.3]{Jungel_Zurek_2021}. More precisely, we choose~$\Omega = (0, 1)^2$, and the coefficients for the diffusion matrix in~\eqref{EQN::DIFFUSION-SKT} and the reaction term in~\eqref{EQN::REACTION-SKT} as follows:
\begin{equation}
\label{EQN::SKT-PARAMETERS}
\begin{tabular}{llllll}
$a_{10} = 0.05$, & $a_{11} = 2.5 \times 10^{-5}$, & $a_{12} = 1.025$,  & 
$a_{20} = 0.05$, & $a_{21} = 0.075$, & $a_{22} = 2.5 \times 10^{-5}$, \\
$b_{10} = 59.7$, & $b_{11} = 24.875$, & $b_{12} = 19.9$, & 
$b_{20} = 49.75$, & $b_{21} = 19.9$, & $b_{22} = 19.9$.
\end{tabular}
\end{equation}
The initial datum is chosen as a perturbation of the equilibrium~$\brho^* = (2, 0.5)$:
\begin{equation}
\label{EQN::INITIAL-CONDITION-SKT}
\rho_1(x, y, 0) = 2 + 0.31 g(x - 0.25, y - 0.25) + 0.31 g(x - 0.75, y - 0.75), \quad \rho_2(x, y, 0) = 0.5,
\end{equation}
where~$g(x, y) = \max\{1 - 8^2 x^2 - 8y^2, 0\}$.

We select the parameters of the nonlinear solver as~$tol = 10^{-6}$ and~$s_{\max} = 50$. 
We consider a rather coarse mesh with~$h \approx 1.41 \times 10^{-1}$ and use high-order approximations of degree~$p = 3$. 
As for the time step, we use the adaptive strategy proposed in~\cite[\S7.1]{Jungel_Zurek_2021}, i.e., at the $n$th time step, if the desired tolerance has not been reached after 50 iterations, the time step~$\tau_{n+1}$ is reduced by a factor of~$0.2$ and the nonlinear solver is restarted, whereas, at the beginning of each time step, we increase the previous
one by a factor of~$1.1$. The initial time step is set as~$\tau_1 = 10^{-4}$. 
As in the previous experiment, we set the regularization parameter as~$\varepsilon=0$.

As discussed in~\cite[\S7.3]{Jungel_Zurek_2021}, due to the cross-diffusion, the equilibrium~$\brho^*$ is unstable for the SKT system (see~\cite[Thm.~3.1]{Tian_etal_2010}), and the choice of the parameters~$\{b_{ij}\}$ leads to the coexistence of the two species (see~\cite[\S6.2]{Shigesada_Kawasaki_1997}). 
In Figure~\ref{FIG::TURING-PATTERN}, we show the evolution of the approximations obtained for the densities~$\rho_1$ and~$\rho_2$ at times~$t = 0.5$ and~$t = 10$, which exhibits the same Turing pattern formation obtained in~\cite[Fig.~1]{Jungel_Zurek_2021}.
\begin{figure}[ht]
\centering
\includegraphics[width = 0.45\textwidth]{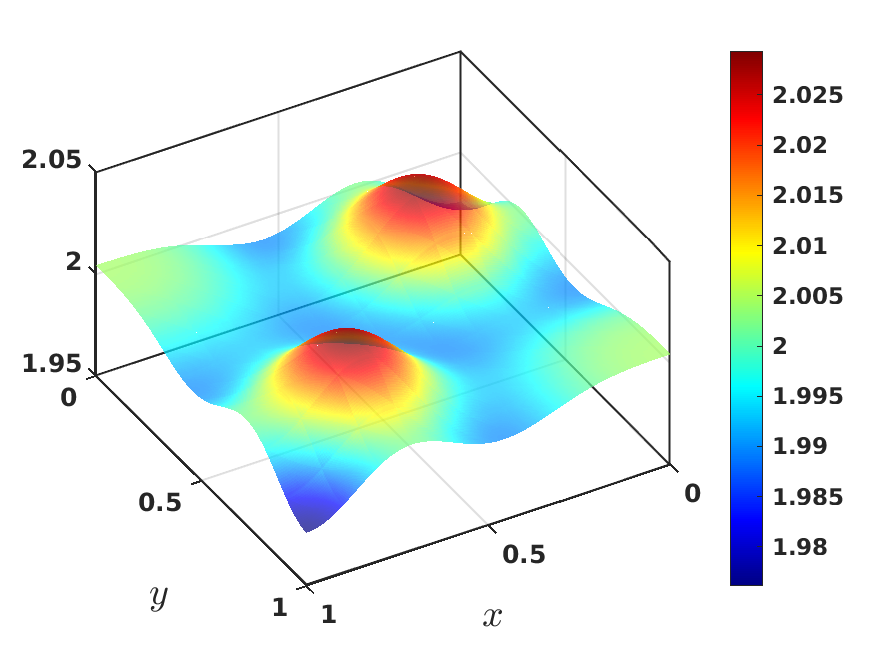}
\hspace{0.2in}
\includegraphics[width = 0.45\textwidth]{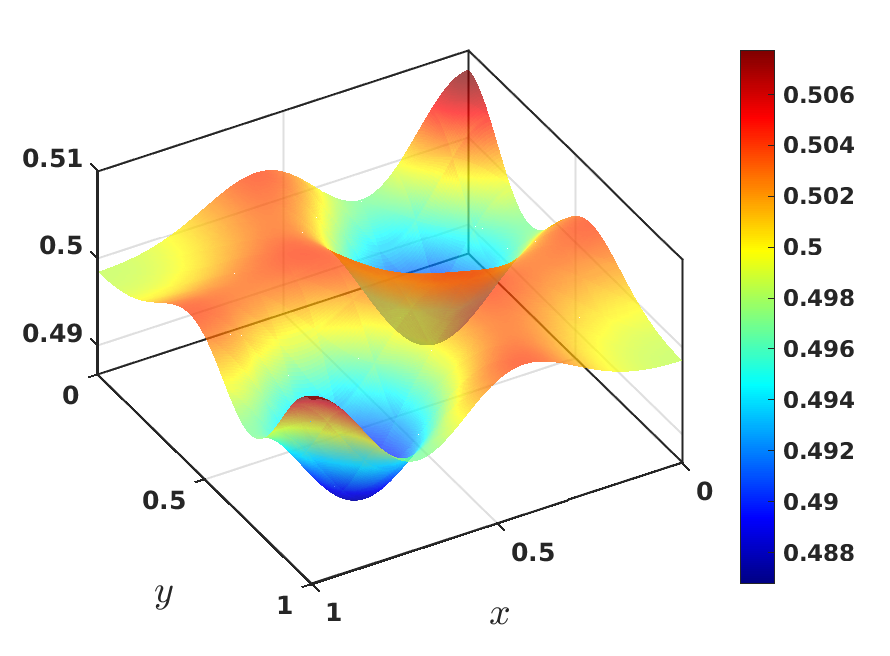}\\
\includegraphics[width = 0.45\textwidth]{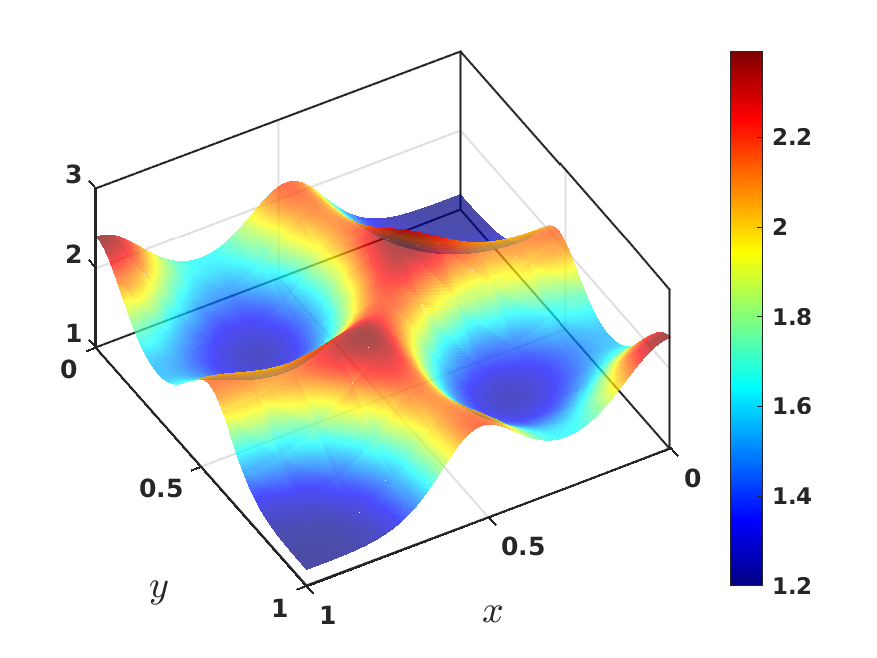}
\hspace{0.2in}
\includegraphics[width = 0.45\textwidth]{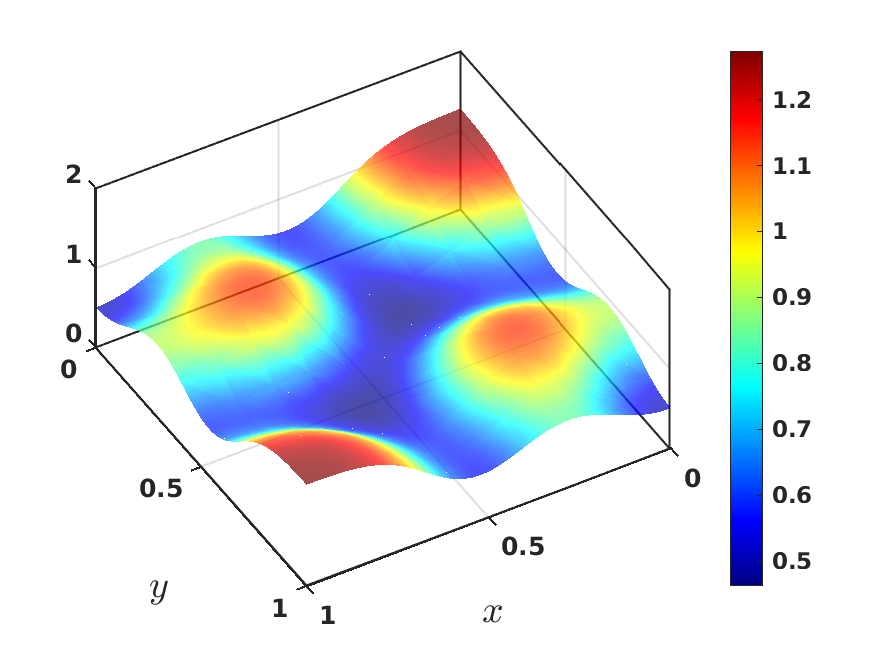}
\caption{Evolution of the approximations of the densities~$\rho_1$ (left panels) and~$\rho_2$ (right panels) for the SKT system with parameters~\eqref{EQN::SKT-PARAMETERS} and initial datum~\eqref{EQN::INITIAL-CONDITION-SKT} at times~$t = 0.5$ (first row) and~$t = 10$ (second row).}
\label{FIG::TURING-PATTERN}
\end{figure}
\section{Conclusions}\label{SECT::CONCLUSIONS}
We designed and analyzed a structure-preserving backward Euler-LDG scheme for nonlinear cross-diffusion systems, which provides approximate solutions that respect the entropy structure of the system, and the positivity or boundedness of the physical unknown in a strong (pointwise) sense.
The existence of discrete solutions and the asymptotic convergence to continuous weak solutions have been proven under some assumptions on the regularizing term and the discrete spaces, whose validity for different cases is verified.
Moreover, high-order convergence rates are numerically observed for some~$L^2(\Omega)$ errors at the final time.


\begin{appendix}
\section{Appendix}\label{app}

We present two examples of cross-diffusion systems satisfying Assumption~\ref{H2a} with~$s'' A\in \EFC{0}{\overline{\calD}}{\IR^{N\times N}}$. 

\subsection{Volume-filling diffusion model for fluid mixtures}

We define the diffusion matrix $A(\brho)=(A_{ij}(\brho))_{i,j=1}^N$ with $N\in\IN$ by 
\begin{align*}
  A_{ii}(\brho) = p_i\rho_i(1-\rho_i), \quad
  A_{ji}(\brho) = -p_i\rho_i\rho_j\quad\mbox{for }i,j=1,\ldots,N,
\end{align*}
where $p_i>0$ are pressure coefficients, and the entropy functional~$s(\brho)$ is defined by
\begin{align*}
  s(\brho) = \sum_{i=1}^N\rho_i(\log\rho_i-1) + \rho_0(\log\rho_0-1) 
  + N + 1,
  \quad\mbox{where }\rho_0 := 1-\sum_{i=1}^N\rho_i,
\end{align*}
and $\brho\in\calD:=\{\brho\in(0,1)^N:\sum_{i=1}^N\rho_i<1\}$. The cross-diffusion system with this diffusion matrix can be derived from a multi-phase viscous fluid model in the diffusion limit similarly as in \cite[\S4.2]{Jungel_2016}, assuming that the partial pressures of the mixture are linear. The fluid mixture consists of $N+1$ components with the volume fractions $\rho_0,\ldots,\rho_N$, which sum up to one.

The entropy $s:\calD\to(0,\infty)$ is convex and $s'$ is invertible on $\calD$. The Hessian $s''(\brho)=(H_{ij})_{i,j=1}^N$ of the entropy has the entries $H_{ij}=\delta_{ij}/\rho_i + 1/\rho_0$ and therefore, for $i\neq j$,
\begin{align*}
  (A(\brho)^{{\sf T}}s''(\brho))_{ii}
  &= A_{ii}(\brho)H_{ii} + \sum_{k\neq i} 
  A_{ki} (\brho)H_{ki}
  = p_i\rho_i(1-\rho_i)\bigg(\frac{1}{\rho_i}+\frac{1}{\rho_0}\bigg)
  - p_i\rho_i\sum_{k\neq i}\frac{\rho_k}{\rho_0} = p_i, \\
  (A(\brho)^{{\sf T}}s''(\brho))_{ij}
  &= \sum_{k=1}^N 
  A_{ki}(\brho)\frac{\delta_{jk}}{\rho_j}
  + \sum_{k=1}^N A_{ki}
  (\brho)\frac{1}{\rho_0}
  = -p_i\rho_i + \frac{p_i}{\rho_0}\rho_i\bigg(1-\rho_i
  - \sum_{k\neq i}\rho_k\bigg) = 0.
\end{align*}
Thus, $s'' A\in \EFC{0}{\overline{\calD}}{\IR^{N\times N}}$ holds and Assumption~\ref{H2a} is satisfied with $\gamma=\min\{p_1,\ldots,p_N\}>0$. 

\subsection{Tumor-growth model}

The growth of an avascular tumor can be described by a cross-diffusion system with the diffusion matrix
\begin{align*}
  A(\brho) = \begin{pmatrix}
  2\rho_1(1-\rho_1) - \beta\theta \rho_1\rho_2^2 
  & -2\beta \rho_1\rho_2(1+\theta\rho_1) \\
  -2\rho_1\rho_2 + \beta\theta(1-\rho_2)\rho_2^2 
  & 2\beta \rho_2(1-\rho_2)(1+\theta \rho_1)
  \end{pmatrix},
\end{align*}
where $\rho_0$, $\rho_1$, and $\rho_2$ denote the volume fractions of the interstitial fluid (water, nutrients), tumor cells, and extracellular matrix, respectively. The parameters $\beta>0$ and $\theta>0$ appear in the partial pressures for the extracellular fluid and tumor cells, respectively. We refer to \cite[\S4.2]{Jungel_2016} for details about the modeling. We choose the same entropy and domain $\calD$ as in the previous subsection. 
A straightforward computation shows that
\begin{align*}
  A(\brho)^{{\sf T}}s''(\brho) = \begin{pmatrix}
	2 & \beta\theta \rho_2 \\ 0 & 2\beta(\theta \rho_1+1)
	\end{pmatrix}
\end{align*}
is positive definite if~$\theta<4/\sqrt{\beta}$. More precisely, $s'' A\in \EFC{0}{\overline{\calD}}{\IR^{N\times N}}$ and there exists $\gamma>0$ such that Assumption~\ref{H2a} is satisfied. The constant $\gamma$ vanishes if $\theta=4/\sqrt{\beta}$, so the strict inequality is needed.
\end{appendix}


\section*{Funding}
The first author is member of the Gruppo Nazionale Calcolo Scientifico-Istituto Nazionale di Alta Matematica (GNCS-INdAM) and acknowledges 
the kind hospitality of the Erwin Schr\"odinger International Institute for Mathematics and Physics (ESI), where part of this
research was developed, and 
support from the Italian Ministry of University and Research through the project PRIN2020 ``Advanced polyhedral discretizations of heterogeneous PDEs for multiphysics problems". This research was funded in part by the Austrian Science Fund (FWF) projects 10.55776/F65 (AJ, IP), 10.55776/P33010 (AJ), and 10.55776/P33477 (IP). This work has received funding from the European Research Council (ERC) under the European Union's Horizon 2020 research and innovation programme, ERC Advanced Grant NEUROMORPH, no.~101018153.

\end{document}